\documentclass[oneside, a4paper,11pt,reqno]{amsart}
\usepackage{yfonts}

\RequirePackage[colorlinks,citecolor=blue,urlcolor=blue]{hyperref}

\makeatletter
\def\timenow{\@tempcnta\time
  \@tempcntb\@tempcnta
  \divide\@tempcntb60
  \ifnum10>\@tempcntb0\fi\number\@tempcntb
  \multiply\@tempcntb60
  \advance\@tempcnta-\@tempcntb
  :\ifnum10>\@tempcnta0\fi\number\@tempcnta}
\makeatother

\usepackage[ddmmyyyy,hhmmss]{datetime}
\usepackage{color}

\usepackage{euscript}

\usepackage[applemac]{inputenc}
\usepackage{amssymb,amsmath,amsthm,dsfont}
\usepackage{graphics,graphicx}
\usepackage[T1]{fontenc}
\usepackage{color}
\usepackage{epsfig,psfrag}

\usepackage{enumitem}

\newtheorem{theo}{Theorem}[section]
\newtheorem{prop}[theo]{Proposition}
\newtheorem{lemme}[theo]{Lemma}

\newtheorem{remarque}[theo]{Remark}

\newtheorem{defi}[theo]{Definition}

\newtheorem{as}{Assumption}

\def\tA2{\tilde{A}({\tilde L_2})}
\def \tts2{\tilde \tau^+_2(  h_t/2 ))}

\title{General self-similarity properties for Markov processes and exponential functionals of L\'evy processes}

\author{Gr\'egoire V\'echambre}
\address{NYU-ECNU Institute of Mathematical Sciences at NYU Shanghai, 3663 Zhongshan Road North, Shanghai, 200062, China}

\email{ghv2@nyu.edu}


\date{\today\ \`a \currenttime}

\begin{document} 

\maketitle

\begin{abstract}
Positive self-similar Markov processes (pssMp) are positive Markov processes that satisfy the scaling property and it is known that they can be represented as the exponential of a time-changed L\'evy process via Lamperti representation. In this work, we are interested in the following problem: what happens if we consider Markov processes in dimension $1$ or $2$ that satisfy self-similarity properties of a more general form than a scaling property ? Can they all be represented as a function of a time-changed L\'evy process ? If not, how can Lamperti representation be generalized ? We show that, not surprisingly, a Markovian process in dimension $1$ that satisfies self-similarity properties of a general form can indeed be represented as a function of a time-changed L\'evy process, which shows some kind of universality for the classical Lamperti representation in dimension $1$. However, and this is our main result, we show that a Markovian process in dimension $2$ that satisfies self-similarity properties of a general form is represented as a function of a time-changed exponential functional of a bivariate L\'evy process, and processes that can be represented as a function of a time-changed L\'evy process form a strict subclass. This shows that the classical Lamperti representation is not universal in dimension $2$. We briefly discuss the complications that occur in higher dimensions. In dimension $2$ we present an example, built from a self-similar fragmentation process, where our representation in term of an exponential functional of a bivariate L\'evy process appears naturally and has a nice interpretation in term of the self-similar fragmentation process. To prove the general representations in dimension $1$ and $2$, some of our arguments work in the context of a general state space and show that, under some conditions, we can exhibit a topological group structure on the state space of a Markov process that satisfies general self-similarity properties, which allows to write a Lamperti type representation for this process in term of a L\'evy process on the group. 
\end{abstract}


\pagestyle{myheadings}
\markboth{Right}{General self-similarity properties for Markov processes}

\section{Introduction}

We consider $X$, a strongly Markovian c\`ad-l\`ag process on a locally compact separable metric space $E$, and we will assume that $X$ satisfies some self-similarity properties. For $y \in E$, let us denote $X_y$ for the process $X$ starting from $y$. We denote by $P_y$ the law of the process $X_y$. 

We do not assume that $X$ is a Feller process (this restriction would not be natural for the processes we are considering) so, even if for simplicity we do not authorize instantaneous killing for $X_y$ (except in the example given in Section \ref{prelres} below), $X_y$ may have finite life-time if, in finite time, every compact subset of $E$ containing $y$ has been left by $X_y$. 
Let us denote by $\zeta(X_y)$ the life-time of $X_y$, that can possibly be finite or infinite. In the case where $\mathbb{P}(\zeta(X_y) < +\infty) > 0$ we consider the usual compactification $E \cup \{ \Delta \}$ of $E$ by adding a cemetery point $\Delta$ of which the neighborhoods are $K^c \cup \{ \Delta \}$, where $K$ goes along compact sets of $E$, and we put $X_y(t) = \Delta$ for all $t \geq \zeta(X_y)$. Because of the definition of $\zeta(X_y)$, we have that $\Delta$ is an absorbing state for $X_y$ and can only be reached continuously. Thanks to this, we systematically consider that our processes are defined on the time interval $[0, +\infty[$. Note that any homeomorphism $f$ of $E$ is naturally extended to an homeomorphism of $E \cup \{ \Delta \}$ by imposing $f(\Delta)=\Delta$. When we compose our processes by homeomorphisms on $E$ we implicitly consider the extension of the homeomorphisms to $E \cup \{ \Delta \}$, if necessary. 


Classically, a strongly Markovian c\`ad-l\`ag process $X$ on $]0, +\infty[$ is called a \textit{positive self-similar Markovian process} (pssMp) if there is $\alpha \in \mathbb{R}$ such that 
\begin{eqnarray}
\forall y > 0, \lambda > 0, \ \left ( X_{\lambda y} (t), \ t \geq 0 \right ) \overset{\mathcal{L}}{=} \left ( \lambda X_y(\lambda^{-\alpha } t), \ t \geq 0 \right ), \label{self-simclassic}
\end{eqnarray}
that is, the process starting from $\lambda y$ is equal in law to a scaled version of the process starting from $y$. In the above we use the convention $\lambda \Delta = \Delta$. These processes appear as the scaling limits of Markov processes (see Lamperti \cite{Lamperti1}) and in many examples of processes built from stable L\'evy processes (stable L\'evy processes conditioned to stay positive, stable L\'evy processes killed when entering $]-\infty, 0[$, norm of an isotropic $d$-dimensional stable L\'evy processes,...). A famous result of Lamperti \cite{Lamperti2} characterizes and gives a representation, the so-called \textit{Lamperti representation}, of a pssMp as the exponential of a time-changed L\'evy process. 

In the recent decades, some generalizations to pssMp's have been introduced. \textit{Real self-similar Markovian processes} (rssMp's) are strongly Markovian c\`ad-l\`ag processes on $\mathbb{R}$ for which $0$ is an absorbing state and that satisfy \eqref{self-simclassic} for all $y \in \mathbb{R} \setminus \{0\}, \lambda > 0$. It is known that a rssMp can be represented in term of a time-changed Markov additive process. This representation is called Lamperti-Kiu representation. It is attributed to Kiu \cite{KIU1980183}, a complete proof is given in \cite{chaumont2013}. 

More generally, \textit{self-similar Markovian processes on $\mathbb{R}^d$}, commonly denoted ssMp's, are strongly Markovian c\`ad-l\`ag processes on $\mathbb{R}^d$ that satisfy \eqref{self-simclassic} for all $y \in \mathbb{R}^d \setminus \{0\}, \lambda > 0$. The generalized Lamperti-Kiu representation allows to express a ssMp in term of a time-changed Markov additive process that is a little more complicated than the one in the case of a rssMp. This representation is attributed to Kiu \cite{KIU1980183}, a complete proof is given in \cite{alili2017}, see also \cite{Graversen1986}. See Pardo, Rivero \cite{pardosurvey} and Kyprianou \cite{5e0b08215cf94b4f9738d908edd90e6f} for recent and complete accounts on pssMp's and rssMp's (for \cite{pardosurvey}), and on pssMp's, rssMp's and ssMp's (for \cite{5e0b08215cf94b4f9738d908edd90e6f}). 


\subsection{More general self-similarity properties} \label{introgssmp}

We are interested in self-similarity properties that are the most general possible and in characterizing and representing the processes that satisfy such self-similarity properties. \eqref{self-simclassic} says that the process starting from the point $\lambda y$ is equal in law to the process starting from the point $y$, linearly time-changed by $(t \longmapsto \lambda^{-\alpha } t)$ and space-changed by composition of the linear function $(z \longmapsto \lambda z)$. It is an open problem to determine the processes that satisfy self-similarity properties given by more general space-time-changes. The linearity of the time change can not be relaxed in the context of homogenous Markov processes. In order to define more general self-similarity properties, we thus allow arbitrary time-changes as long as they are linear and arbitrary space-changes. We thus define a general notion of self-similarity in the following way: 

\begin{defi} [general self-similarity properties, invariance components] \label{self-simoursens}

We say that a strongly Markovian c\`ad-l\`ag process $X$ on a locally compact separable metric space $E$ satisfies \textbf{general self-similarity} if for some point $y_0 \in E$ we have 
\begin{eqnarray}
\forall y \in E, \ \left ( X_y (t), \ t \geq 0 \right ) \overset{\mathcal{L}}{=} \left ( f_y \left ( X_{y_0}(c_y t) \right ), \ t \geq 0 \right ), \label{self-simoursenseq}
\end{eqnarray}
where $(f_y, \ y \in E)$ is a family of homeomorphisms of $E$ such that $(y,x) \longmapsto f_y(x)$ and $(y,x) \longmapsto f_y^{-1}(x)$ are continuous from $E \times E$ to $E$, and where $(c_y, \ y \in E)$ is a family of positive constants such that $y \longmapsto c_y$ is continuous from $E$ to $\mathbb{R}^*_+$. Then, we say that $((f_y, c_y), \ y \in E)$ is a family of \textbf{invariance components} for $X$, relatively to the reference point $y_0$. 

\end{defi}

\begin{remarque} \label{cemetery}
Note that, in the above definition, for the case where the processes involved reach the cemetery point $\Delta$ in finite time, we have assumed that \eqref{self-simoursenseq} holds for the extension of $f_y(.)$ to $E \cup \{ \Delta \}$ defined by $f_y(\Delta) = \Delta$. It is not difficult to see that the extensions to $E \times E \cup \{ \Delta \}$ of the functions $(y,x) \longmapsto f_y(x)$ and $(y,x) \longmapsto f_y^{-1}(x)$ are continuous. 
\end{remarque}


\begin{remarque} \label{continuity}
The continuity of the applications $(y,x) \longmapsto f_y(x)$ and $(y,x) \longmapsto f_y^{-1}(x)$, assumed in Definition \ref{self-simoursens}, is equivalent to the continuity of the applications $y \longmapsto f_y(.)$ and $y \longmapsto f_y^{-1}(.)$ from $E$ to $\mathcal{C}^0(E, E)$ (respectively $E$ to $\mathcal{C}^0(E \cup \{ \Delta \}, E \cup \{ \Delta \})$ for the case where it is needed to consider the extension to $E \cup \{ \Delta \}$), the set of continuous functions from $E$ to $E$ (respectively from $E \cup \{ \Delta \}$ to $E \cup \{ \Delta \}$) equipped with the natural topology of uniform convergence on every compact sets of $E$ (respectively uniform convergence on $E \cup \{ \Delta \}$). 
\end{remarque}

As mentioned above, natural examples of ssMp's can be built from stable L\'evy processes. For example the norm of an isotropic stable L\'evy process in $\mathbb{R}^d$ is a pssMp, a stable L\'evy process in $\mathbb{R}^d$ is itself a ssMp, etc... For processes satisfying Definition \ref{self-simoursens}, natural examples can be built from exponential functionals of L\'evy processes. For example, let $(\xi_1, \xi_2, \xi_3)$ be a L\'evy process in $\mathbb{R}^3$, and for any $y = (y_1, y_2, y_3) \in \mathbb{R}^3$, let us define 
\[ X_y(t) := \left( y_1 + \xi_1(t), \ y_2 + e^{y_1} \int_0^t e^{\xi_1(s-)} d \xi_2(s), \ y_3 + e^{2 y_1} \int_0^t e^{2\xi_1(s-)} d \xi_3(s) \right). \]
Then $X$ satisfies Definition \ref{self-simoursens} in $\mathbb{R}^3$ and a family of invariance components $((f_y, c_y), \ y \in \mathbb{R}^3)$, with respect to the reference point $(0,0,0)$, is given by 
\begin{eqnarray}
\forall x = \left( x_1, x_2, x_3 \right) \in \mathbb{R}^3, \ f_y(x) = \left( y_1 + x_1, \ y_2 + e^{y_1} x_2, \ y_3 + e^{2y_1} x_3 \right) \ \ \ \text{and} \ \ \ c_y = 1. \label{exampleinvcomp}
\end{eqnarray}
As we will see, exponential functionals of L\'evy processes play a very important role in the representation of processes satisfying Definition \ref{self-simoursens} and, at least in dimension $2$, they even allow to represent all of them, provided that they satisfy some regularity conditions. More generally, it seems that the methods that we develop in this paper could be applied to study some complicated processes built from exponential functionals of L\'evy processes. 

A natural problem is to characterize the generalized self-similar Markovian processes that satisfy Definition \ref{self-simoursens} and to investigate the following questions: 1) Can they all, similarly to pssMp's, be represented as a function of a time-changed L\'evy process ? In other words, do they satisfy some kind of Lamperti representation ? If not, 2
) how can Lamperti representation be generalized ? The object of the present paper is to characterize processes in dimensions $1$ and $2$ that satisfy Definition \ref{self-simoursens} plus some reasonable assumptions and to provide a Lamperti type representation for them. In particular we will see that the answer to 1) is positive in dimension $1$ and negative in dimension $2$. In dimension $2$, there are two cases: the case of processes that can be represented as a function of a time-changed L\'evy process, and another case for which we provided a generalized Lamperti representation in term of an exponential functional of a bivariate L\'evy process. The two cases are unified in one general representation. To prove the claimed representations, we show that, under some regularity and smoothness conditions, we can in general exhibit a topological group structure on the state space of a Markov process that satisfies Definition \ref{self-simoursens}. This allows to write a Lamperti type representation for such a process in term of a L\'evy process on the group and, in dimension $1$ and $2$, we can express these L\'evy processes on groups in term of classical L\'evy processes on $\mathbb{R}$ or $\mathbb{R}^2$. 
We mention in Subsection \ref{dim>2} what kind of results can be expected in dimension greater or equal to $3$ and what are the difficulties in those dimensions. 

The class of processes satisfying Definition \ref{self-simoursens} (which is not, rigorously speaking, a new class of processes since, as our results will show, these processes can be represented in term of functionals of L\'evy processes), can be seen as a generalization of pssMp's, but in a completely different direction than ssMp's. Indeed, there are two important differences between ssMp's and processes satisfying Definition \ref{self-simoursens}. On one hand the assumption of self-similarity that we make is very strong and says that the processes, starting from any point, can be obtained from the process starting at $y_0$. Such a property is not true in general for rssMp's and ssMp's, but it is for pssMp's, and we will see that it implies some similarities between pssMp's and processes satisfying Definition \ref{self-simoursens}, like the appearance of L\'evy processes in the representation of these processes, while Markov additive processes appear in the representation of rssMp's and ssMp's, or the fact that the life-time of the process is always the exponential functional of a L\'evy process. On the other hand, the type of self-similarity that we assume is much more general than a scaling property. In particular, the higher the dimension is, the more complicated may the "structure" of the self-similarity be. We will see that in dimension $1$ this structure is always simple, which leads us to extend Lamperti representation to processes in dimension $1$ that satisfy Definition \ref{self-simoursens} plus some regularity assumptions. In dimension $2$ there are two possible structures, and they can be unified into a generalized Lamperti representation where the L\'evy process is replaced by an exponential functional of a bivariate L\'evy process. For higher dimensions, there are even more possible structures, as explained in Subsection \ref{dim>2} and it is not clear whether all cases can be unified into one nice representation. 

Let us first mention some simple consequences of Definition \ref{self-simoursens}. 

\textbf{Combination with Strong Markov Property:} For any starting point $z$, if $\tilde X_{y_0} \sim P_{y_0}$ is independent from $X_z$ and $S$ is a stopping time for $X_z$, then, according to the strong Markov property at $S$ and Definition \ref{self-simoursens} we have 
\[ \left ( X_z (S+t), \ t \geq 0 \right ) \overset{\mathcal{L}}{=} \left ( f_{X_z(S)} \left ( \tilde X_{y_0}(c_{X_z(S)} t) \right ), \ t \geq 0 \right ). \]

\textbf{Stability by homeomorphism:} Let $E$ and $F$ be locally compact separable metric spaces, and $X$ be a process on $E$ satisfying Definition \ref{self-simoursens} with invariance components $((f_y, c_y), \ y \in E)$, relatively to some reference point $y_0$. Then, if $h : E \longrightarrow F$ is an homeomorphism, we have that $h(X)$ satisfies Definition \ref{self-simoursens} on $F$ with invariance components $((h \circ f_{h^{-1}(z)} \circ h^{-1}, c_{h^{-1}(z)}), \ z \in F)$, relatively to the reference point $h(y_0)$. 

\textbf{Change of reference point:} For a process satisfying Definition \ref{self-simoursens}, the reference point $y_0$ can be chosen arbitrary. Indeed, let $X$ be such a process and let $((f_y, c_y), \ y \in E)$ be a family of invariance components for $X$, relatively to some reference point $y_1$. Let $y_2$ be any other point in $E$, then it is easy to see that $((f_y \circ f_{y_2}^{-1}, c_y/c_{y_2}), \ y \in E)$ is a family of invariance components for $X$, relatively to the reference point $y_2$. 

\textbf{Non-uniqueness of invariance components:} For a given reference point $y_0$ there is not, in general, unicity for the choice of the family of invariance components. Indeed, let $X$ be a process satisfying Definition \ref{self-simoursens} and let $((f_y, c_y), \ y \in E)$ be a family of invariance components for $X$, relatively to some reference point $y_0$. If there is an homeomorphism $\Psi$ from $E$ to $E$ and a constant $\lambda$ such that 
\begin{eqnarray}
\left ( X_{y_0} (t), \ t \geq 0 \right ) \overset{\mathcal{L}}{=} \left ( \Psi \left ( X_{y_0}(\lambda t) \right ), \ t \geq 0 \right ), \label{invarsurplace}
\end{eqnarray}
then clearly $((f_y \circ \Psi, c_y \lambda), \ y \in E)$ is also a family of invariance components for $X$, relatively to the reference point $y_0$. Note that it is also possible to let $\Psi$ and $\lambda$ vary smoothly relatively to $y$ in the set of $(\Psi, \lambda)$ that satisfy \eqref{invarsurplace}, to produce a family of invariance components $((f_y \circ \Psi_y, c_y \lambda_y), \ y \in E)$. In particular, the set of all possible families of invariance components of a process satisfying Definition \ref{self-simoursens} can possibly be quite complicated. Among all possible choices of families of invariance components it will be convenient to work with those that satisfy some nice properties. In the logic of Remark \ref{cemetery}, note that in \eqref{invarsurplace} (and also in \eqref{defsym} below, and everywhere where such equalities in law for processes appear), for the case where the processes involved reach the cemetery point $\Delta$ in finite time, we actually consider the extension of $\Psi$ (respectively $h$) to $E \cup \{ \Delta \}$ that satisfies $\Psi(\Delta)=\Delta$ (respectively $h(\Delta)=\Delta$), but the extension of $\Psi$ (respectively $h$) is still denoted by $\Psi$ (respectively $h$) for simplicity. 

An important part of this work consists in showing the existence of what we will call \textit{good invariance components}. For this, some assumptions on $E$ and $X$ are needed. For $y \in E$, let $Supp(X_y)$ denote the support of $X_y$ in $E$: 
\begin{align}
Supp(X_y) & := \overline{\bigcup_{t \geq 0} Supp(X_y(t))} \label{defsupp} \\ 
& = \left \{ z \in E, \ \text{s.t.} \ \forall \epsilon > 0, \exists t \geq 0 \ \text{for which} \ \mathbb{P} (X_y(t) \in B(z, \epsilon)) > 0 \right \}. \nonumber
\end{align}
By reducing $E$ if necessary, we can assume that $E = Supp(X_{y_0})$. Therefore, we define the following assumption 
\begin{as} \label{hypsupport1}
\[ E = Supp(X_{y_0}). \]
\end{as}
Because of the use of continuity arguments it is often convenient to assume that $E$ is connected. 
Moreover, when $E$ is connected, Assumption \ref{hypsupport1} is implied by a simple assumption of non-degeneracy: 
\begin{as} \label{hypsupport2}
\[ \exists y \in E \ \text{s.t.} \ y \in \mathring{\widehat{Supp(X_{y})}}. \]
\end{as}
Assumption \ref{hypsupport2} can be motivated by the need to avoid some degenerate cases, for example arithmetic processes on $E = \mathbb{R}$. 
The fact that Assumption \ref{hypsupport2} implies Assumption \ref{hypsupport1} is proved in Lemma \ref{equivassumpt12} of Section \ref{technical}. We also need to take into account the symmetries of $X$. For any $y \in E$ let us define $Sym(X_y) \subset \mathcal{C}^0(E, E)$ to be the group of symmetries of $X_y$, i.e. the group of homeomorphisms that leave invariant the law of $X_y$: 
\begin{eqnarray}
Sym(X_y) := \left \{ h \in Hom(E), \ X_y \overset{\mathcal{L}}{=} h(X_y) \right \}. \label{defsym}
\end{eqnarray}
If $((f_y, c_y), \ y \in E)$ is a family of invariance components for $X$, relatively to some reference point $y_0$, then it is easy to see that we have $Sym(X_y) = f_y \circ Sym(X_{y_0}) \circ f_y^{-1}$. Note that, if $y \longmapsto h_y$ is a continuous application from $E$ to $Sym(X_{y_0}) \subset \mathcal{C}^0(E, E)$, then $((f_y \circ h_y, c_y), \ y \in E)$ is still a family of invariance components for $X$, relatively the reference point $y_0$ (the required assumptions of continuity are satisfied because of Remark \ref{continuity}). Therefore, the symmetries of $X$ can interfere with the self-similarity property, in the sens that they can make the set of families of invariance components very complicated. In order to avoid too much problems, we will often need to prove or assume that $Sym(X_{y_0})$ is discrete (as a subset of $\mathcal{C}^0(E, E)$ equipped with the topology of uniform convergence on every compact sets). Note that, since $Sym(X_y) = f_y \circ Sym(X_{y_0}) \circ f_y^{-1}$, "$Sym(X_{y_0})$ is discrete" is always equivalent to "$Sym(X_{y})$ is discrete for some $y \in E$". It will be shown in Lemma \ref{discretendim1} of Section \ref{technical} that, under Assumption \ref{hypsupport1}, "$Sym(X_{y_0})$ is discrete" is always true in dimension $1$. 

Finally, we need assumptions of regularity in order to obtain the existence of what we will call \textit{good invariance components}. When the state space $E$ is equipped with a differential structure of class $\mathcal{C}^k$, we can define \textit{$\mathcal{C}^k$ invariance components} in the following way: 
\begin{defi} [$\mathcal{C}^k$ invariance components] \label{smoothcomponents}
For $k \geq 1$, let $X$ be a process satisfying Definition \ref{self-simoursens} on a $\mathcal{C}^k$-differentiable manifold $\mathcal{M}$. Let $((f_y, c_y), \ y \in \mathcal{M})$ be a family of invariance components for $X$, relatively to some reference point $y_0$. We say that $((f_y, c_y), \ y \in E)$ is a family of \textbf{$\mathcal{C}^k$ invariance components} if the applications $((y,x) \longmapsto f_y(x))$, $((y,x) \longmapsto f_y^{-1}(x))$, and $(y \longmapsto c_y)$ are of class $\mathcal{C}^k$. 

\end{defi}

Note that the procedure of "change of reference point" described above transforms a family of $\mathcal{C}^k$ invariance components into a family of $\mathcal{C}^k$ invariance components. In the remainder, when we consider a process satisfying Definition \ref{self-simoursens} with $\mathcal{C}^k$ invariance components on an open interval $I \subset \mathbb{R}$ (respectively on an open simply connected domain $\mathcal{D} \subset \mathbb{R}^2$), it is implicit that the differential structure on $I$ (respectively on $\mathcal{D}$) is the natural one, arising from the fact that it is an open subset of $\mathbb{R}$ (respectively of $\mathbb{R}^2$), and it induces the usual differentiation on $\mathbb{R}$ (respectively on $\mathbb{R}^2$). 


\subsection{General results}

We now state Lamperti type representations for processes satisfying Definition \ref{self-simoursens} with $\mathcal{C}^k$ invariance components in dimension $1$ and $2$. In dimension $1$, it turns out that they can be expressed as the image by some function of a time-changed L\'evy process (or equivalently as the image by some function of a pssMp). This shows some kind of universality for Lamperti representation, since it is shared by processes in dimension $1$ that satisfy self-similarity properties, no matter what is the form of their invariance components. 

Recall that the starting point of a L\'evy process on $\mathbb{R}$ (or more generally on a group) is always $0$ (respectively the neutral element of the group). 

In dimension $1$ our result is the following: 

\begin{theo} [Dimension $1$] \label{casr} 

For $k \geq 1$, let $X$ be a process satisfying Definition \ref{self-simoursens} with $\mathcal{C}^k$ invariance components on $I$, an open interval of $\mathbb{R}$. Let us fix a point $y_0 \in I$ to be the reference point. We assume that either Assumption \ref{hypsupport1} or Assumption \ref{hypsupport2} is satisfied for $E=I$. Then, there is a $\mathcal{C}^k$-diffeomorphism $\psi : \mathbb{R} \longrightarrow I$, a real L\'evy process $\xi$ and $\alpha \in \mathbb{R}$ such that if we set 
\begin{eqnarray}
\forall t \in [0, +\infty], \ \varphi(t) := \int_0^t e^{\alpha \xi(s)} ds, \label{timechangingdim1}
\end{eqnarray}
then 
\begin{eqnarray}
\zeta(X_{y_0}) = \varphi(+\infty) = \int_0^{+\infty} e^{\alpha \xi(s)} ds \label{tpsviedim1}
\end{eqnarray}
and 
\begin{eqnarray}
\forall \ 0 \leq t < \varphi(+\infty), \ X_{y_0}(t) = \psi \left ( \xi \left (\varphi^{-1}(t) \right ) \right ). \label{represdim1}
\end{eqnarray}


Reciprocally, for an open interval $I \subset \mathbb{R}$, an homeomorphism $\psi : \mathbb{R} \longrightarrow I$, a real L\'evy process $\xi$, and $\alpha \in \mathbb{R}$, let us fix $y \in I$ and define $\forall t \in [0, +\infty], \ \varphi_y(t) := \int_0^t e^{\alpha (\psi^{-1}(y) + \xi(s))} ds$ and $\forall 0 \leq t < \varphi_y(+\infty), X_{y}(t) := \psi ( \psi^{-1}(y) + \xi (\varphi_y^{-1}(t) ))$. If $\varphi_y(+\infty)< +\infty$ then $\forall t \geq \varphi_y(+\infty), X_y(t) := \Delta$, where $\Delta$ is a cemetery point. Then we have $\zeta(X_{y}) = \varphi_y(+\infty)$ a.s. and $X$ satisfies Definition \ref{self-simoursens} on $I$. A family of invariance components $((f_y, c_y), \ y \in I)$, relatively to the reference point $\psi(0)$, is given by $f_y(.) := \psi(\psi^{-1}(y) + \psi^{-1}(.))$ and $c_y := e^{-\alpha \psi^{-1}(y)}$. 



\end{theo}

In dimension $2$, it turns out that the class of processes satisfying Definition \ref{self-simoursens} with $\mathcal{C}^k$ invariance components is larger than (i.e. contains strictly) the class of processes that can be expressed as the image by some function of a time-changed two-dimensional L\'evy process. 
More precisely, our main result is the fact that, in dimension $2$, processes satisfying Definition \ref{self-simoursens} with $\mathcal{C}^k$ invariance components are identified with the class of processes that can be expressed as the image by some function of a time-changed exponential functional of a bivariate L\'evy process, and this class contains, as a subclass, processes that can be expressed as the image by some function of a time-changed two-dimensional L\'evy process. For $i \in \{1, 2\}$, let $\pi_i(x)$ denote the linear projection of an element $x \in \mathbb{R}^2$ to its $i^{th}$ coordinate. Our result is the following: 

\begin{theo} [Dimension $2$] \label{casr2}

For $k \geq 2$, let $X$ be a process satisfying Definition \ref{self-simoursens} with $\mathcal{C}^k$ invariance components on $\mathcal{D}$, an open simply connected domain of $\mathbb{R}^2$. Let us fix a point $y_0 \in \mathcal{D}$ to be the reference point. We assume that either Assumption \ref{hypsupport1} or Assumption \ref{hypsupport2} is satisfied for $E=\mathcal{D}$ and that $Sym(X_{y_0})$ is discrete. Then, there is a $\mathcal{C}^k$-diffeomorphism $\psi : \mathbb{R}^2 \longrightarrow \mathcal{D}$, a L\'evy process $(\xi, \eta)$ on $\mathbb{R}^2$, and $\alpha, \beta \in \mathbb{R}$ such that if we set 
\begin{eqnarray}
\forall t \in [0, +\infty], \ \varphi(t) := \int_0^t e^{\alpha \xi(s)} ds, \label{timechangingdim2}
\end{eqnarray}
then 
\begin{eqnarray}
\zeta(X_{y_0}) = \varphi(+\infty) = \int_0^{+\infty} e^{\alpha \xi(s)} ds \label{tpsviedim2}
\end{eqnarray}
and 
\begin{eqnarray}
\forall \ 0 \leq t < \varphi(+\infty), \ X_{y_0}(t) = \psi \left ( \xi \left (\varphi^{-1}(t) \right ), \int_0^{\varphi^{-1}(t)} e^{\beta \xi(s-)} d\eta(s) \right ). \label{represdim2}
\end{eqnarray}




Reciprocally, for an open simply connected domain $\mathcal{D} \subset \mathbb{R}^2$, an homeomorphism $\psi : \mathbb{R}^2 \longrightarrow \mathcal{D}$, a $\mathbb{R}^2$-valued L\'evy process $(\xi, \eta)$, and $\alpha, \beta \in \mathbb{R}$, let us fix $y \in \mathcal{D}$ and define $\forall t \in [0, +\infty], \ \varphi_y(t) := \int_0^t e^{\alpha (\pi_1(\psi^{-1}(y))+\xi(s))} ds$, and $\forall \ 0 \leq t < \varphi_y(+\infty)$, 
\[ X_y(t) := \psi \left ( \pi_1(\psi^{-1}(y)) + \xi (\varphi_y^{-1}(t) ), \pi_2(\psi^{-1}(y)) + \int_0^{\varphi_y^{-1}(t)} e^{\beta (\pi_1(\psi^{-1}(y)) + \xi(s-))} d\eta(s) \right ). \]
If $\varphi_y(+\infty)< +\infty$ then $\forall t \geq \varphi_y(+\infty), X_y(t) := \Delta$, where $\Delta$ is a cemetery point. Then we have $\zeta(X_{y}) = \varphi_y(+\infty)$ a.s. and $X$ satisfies Definition \ref{self-simoursens} on $\mathcal{D}$. A family of invariance components $((f_y, c_y), \ y \in \mathcal{D})$, relatively to the reference point $\psi(0,0)$, is given by $f_y(.) := \psi(\pi_1(\psi^{-1}(y) + \psi^{-1}(.)), \pi_2(\psi^{-1}(y) + e^{\beta \pi_1(\psi^{-1}(y))} \psi^{-1}(.)))$ and $c_y := e^{-\alpha \pi_1(\psi^{-1}(y))}$. 

\end{theo}


In the above theorem, a process satisfying Definition \ref{self-simoursens} with $\mathcal{C}^k$ invariance components is expressed as a function of a time-changed process of the form $( \xi(.), \int_0^. e^{\beta \xi(s-)} d \eta(s) )$, where $(\xi, \eta)$ is a two-dimensional L\'evy process and $\beta \in \mathbb{R}$. The particular case where $\beta = 0$ corresponds to the case where $X$ can be expressed as a function of a time-changed two-dimensional L\'evy process. 

\begin{remarque}
In the above two theorems, $\alpha$ plays a similar role as the self-similarity index in Lamperti representation, and in the case of a pssMp, our $\alpha$ even coincides with the classical self-similarity index, as it will be seen in Subsection \ref{examples}. It is possible to have $\alpha = 0$ which corresponds to the case where the change of time $\varphi^{-1}(.)$ is trivial. 
\end{remarque}

An important intermediary result in the proof of the above theorems is the fact that, in general, when the state space is equipped with a differential structure, we can establish a generalized Lamperti representation that involves L\'evy processes. This can seem surprising since there is no natural meaning for a L\'evy process on a general state space. However, we show that the state space can be equipped with a structure of Lie group such that a process satisfying Definition \ref{self-simoursens} with $\mathcal{C}^k$ invariance components is represented as a time-changed left L\'evy process on this Lie group (in the remainder, when we work on a non-commutative group, we only deal with left L\'evy processes and we call these simply L\'evy processes). This result is the following. 

\begin{theo} \label{casgeneral}

For $k \geq 1$, let $X$ be a process satisfying Definition \ref{self-simoursens} with $\mathcal{C}^k$ invariance components on a connected $\mathcal{C}^k$-differentiable manifold $\mathcal{M}$. Let us fix a point $y_0 \in \mathcal{M}$ to be the reference point. We assume that either Assumption \ref{hypsupport1} or Assumption \ref{hypsupport2} is satisfied for $E = \mathcal{M}$ and that $Sym(X_{y_0})$ is discrete. Then, there is an intern composition law $\star$ on $\mathcal{M}$ and a $\mathcal{C}^k$-function $h : \mathcal{M} \longrightarrow \mathbb{R}_+$ such that 
\begin{itemize}
\item $(\mathcal{M},\star)$ is a $\mathcal{C}^k$-Lie group with neutral element $y_0$ ;
\item $h$ is a $\mathcal{C}^k$-Lie group homomorphism from $(\mathcal{M},\star)$ to $(\mathbb{R}_+^*, \times)$ ;
\item There is a L\'evy process $L$ (starting from $y_0$) on $(\mathcal{M},\star)$ such that, if we define 
$\forall t \in [0, +\infty], \varphi(t) := \int_0^t 1/h(L(s)) ds$ then $\zeta(X_{y_0}) = \varphi(+\infty)$ and 
\begin{eqnarray}
\forall \ 0 \leq t < \varphi(+\infty), \ X_{y_0}(t) = L \left (\varphi^{-1}(t) \right ). \label{represmanifold}
\end{eqnarray}
\end{itemize}
Moreover, note that if we let $\xi(.) := \log(1/h(L(.)))$, then $\xi$ is a real L\'evy process and we have $\zeta(X_{y_0}) = \int_0^{+\infty} e^{\xi(s)} ds$. 

\end{theo}



The above theorem says that, under some assumptions, processes satisfying Definition \ref{self-simoursens} with $\mathcal{C}^k$ invariance components can be identified as time-changed L\'evy processes on the state space equipped with some group structure. This can be seen as a generalization of Lamperti representation for general state spaces. 
Indeed, since exponentials of L\'evy processes on $(\mathbb{R}, +)$ coincide with L\'evy processes on $(\mathbb{R}_+^*, \times)$, the classical Lamperti representation can be reformulated in the following way: 
\begin{theo} [Reformulated classical Lamperti representation] \label{classic lamp}

Let $X$ be a pssMp of index $\alpha \in \mathbb{R}$. Let us define the continuous group homomorphism $h$ from $(\mathbb{R}_+^*, \times)$ to $(\mathbb{R}_+^*, \times)$ by $h(x) = x^{-\alpha}$. Then, there is a L\'evy process $L$ (starting from $1$) on $(\mathbb{R}_+^*, \times)$ such that, if we define 
$\forall t \in [0, +\infty], \varphi(t) := \int_0^t 1/h(L(s)) ds$ then $\zeta(X_{1}) = \varphi(+\infty)$ and 
\begin{eqnarray}
\forall \ 0 \leq t < \varphi(+\infty), \ X_1(t) = L \left (\varphi^{-1}(t) \right ). \label{represlamp}
\end{eqnarray}


\end{theo}

\begin{remarque}
Thanks to the above reformulation of Lamperti representation we see that, in Theorem \ref{casgeneral}, the group homomorphism $h$ generalizes the self-similarity index from Lamperti representation, and the $\alpha$ from Theorems \ref{casr} and \ref{casr2}. 
It is possible to have $h \equiv 1$ (this means $\alpha = 0$ in the case of a pssMp) which corresponds to the case where the change of time, $\varphi^{-1}(.)$, is trivial, in that case $X_{y_0}$ is a L\'evy process on $\mathcal{M}$, equipped with the group structure constructed by the theorem. 
\end{remarque}

The interest of Theorem \ref{casgeneral} is that, not only the state space is quite general, but, besides regularity assumptions, we do not assume any particular form for the invariance components, that is, for the type of self-similarity satisfied by the process. 
It is not difficult to prove that the reciprocal of Theorem \ref{casgeneral} is true: time-changed L\'evy processes on topological groups give rise to processes satisfying Definition \ref{self-simoursens}. We state this in the following proposition. 

\begin{prop} \label{levychangentpssensfacil}

Let $(E,\star)$ be a topological group with neutral element denoted by $e_0$, and where $E$ is a locally compact separable metric space. Let $h$ be a continuous group homomorphism from $(E,\star)$ to $(\mathbb{R}_+^*, \times)$, and $L$ be a L\'evy process on $(E,\star)$. For a fixed $y \in E$ let us define
\begin{eqnarray}
\forall t \in [0, +\infty], \ \varphi_y(t) := \frac1{h(y)} \int_0^t \frac1{h(L(s))} ds, \label{timechangingrecip}
\end{eqnarray}
and 
\begin{eqnarray}
\forall \ 0 \leq t < \varphi_y(+\infty), \ X_y(t) := y \star L \left (\varphi_y^{-1}(t) \right ). \label{timechangedrecip} 
\end{eqnarray}
If $\varphi_y(+\infty)< +\infty$ then 
\[ \forall t \geq \varphi_y(+\infty), \ X_y(t) := \Delta, \]
where $\Delta$ is a cemetery point. Then we have $\zeta(X_{y}) = \varphi_y(+\infty)$ a.s. and $X$ processes satisfies Definition \ref{self-simoursens} on $E$. 
A family of invariance components $((f_y, c_y), \ y \in E)$ is given by $f_y(.) := y \star .$ and $c_y := h(y)$, for the reference point $e_0$. Moreover, note that if we let $\xi(.) := \log(1/h(L(.)))$, then $\xi$ is a real L\'evy process and we have $\zeta(X_{y}) = \frac1{h(y)} \int_0^{+\infty} e^{\xi(s)} ds$. 

\end{prop}


\begin{remarque} \label{startingpointnonexpl}
In the representations \eqref{represdim1}, \eqref{represdim2}, \eqref{represmanifold}, and \eqref{represlamp} from the above theorems, if one replaces $\xi$, $( \xi ( . ), \int_0^{.} e^{\beta \xi(s-)} d\eta(s) )$, $L$, and $L$ by respectively $\psi^{-1}(y) + \xi$, \\ $(\pi_1(\psi^{-1}(y)) + \xi(.), \pi_2(\psi^{-1}(y)) + \int_0^{.} e^{\beta (\pi_1(\psi^{-1}(y)) + \xi(s-))} d\eta(s))$, $y \star L$, and $y \times L$ (where, respectively, $y \in I$, $y \in \mathcal{D}$, $y \in E$, and $y \in \mathbb{R}_+^*$), 
then one obtains a version of $X_{y}$. 


\end{remarque}

\begin{remarque} \label{exittimealwaysfonctexpo}
An interesting consequence of the previous results is that, for any process $X$ to which one of our theorems is applicable, the exit time of $X$ from its domain, no matter how complicated this domain is, is always equal in law to the exponential functional of a real L\'evy process, just as in the case of a pssMp (an heuristic explanation for this similarity is proposed a little after Remark \ref{continuity}). This may seem surprising since, according to (generalized) Kiu-Lamperti representation, the life-time of a rssMp and more generally of a ssMp is in general not the exponential functional of a L\'evy process but of a Markov additive process (see for example the end of Sextion 3.2 in \cite{pardosurvey}). Also, a pssMp that has finite life-time always leaves the domain $]0, +\infty[$ at $\{0\}$ (when $\alpha > 0$) or at $\{+\infty\}$ (when $\alpha < 0$). In the case of a process satisfying Definition \ref{self-simoursens}, the way to leave the domain can be very different, since the exit from the domain does not have to be made at a particular point of the boundary. 
It is thus remarkable that the law of the exit time remains the same as in the case of a pssMp. The results of the next subsection make explicite, in some cases, the decompositions from the above theorems so in particular the L\'evy process $\xi$, of which the exponential functional is the the life-time of $X$, can be constructed explicitly from $X$ and its good invariance components (see Remarks \ref{levyexpldim1} and \ref{levyexplarbspace} below). 
\end{remarque}



\subsection{Good components and explicit results}

\begin{defi} [Good invariance components] \label{goodcomponents}
Let $X$ be a processes satisfying Definition \ref{self-simoursens} on a locally compact separable metric space $E$. Let $((f_y, c_y), \ y \in E)$ be a family of invariance components for $X$, relatively to some reference point $y_0$. We say that $((f_y, c_y), \ y \in E)$ are \textbf{good invariance components} if they have the following properties: 
\begin{align}
& f_{y_0} = id_E \ \text{and} \ c_{y_0} = 1, \label{refpoint} \\
& \forall y, z \in E, \ c_{f_y(z)} = c_y \times c_z. \label{relmorphisme}
\end{align}
\end{defi}

For example we see that the invariance components defined in \eqref{exampleinvcomp}, in Proposition \ref{levychangentpssensfacil}, and in the reciprocals of Theorems \ref{casr}-\ref{casr2}, for the processes $X$ defined there, are good invariance components. As we will see, good invariance components are exactly those that can be naturally related to a group structure on the state space $E$, as in Proposition \ref{levychangentpssensfacil}. In the proofs of our results, making appear a group structure on the state space is a key point to represent a process satisfying Definition \ref{self-simoursens} in term of a L\'evy process, but for this we need to work with good invariance components. Therefore, the existence of good invariance components has a theoretical interest, for example for the proofs of Theorems \ref{casr}, \ref{casr2} and \ref{casgeneral}, and on the other hand, having good invariance components is also useful to get explicite relations between the invariance components and the representations of $X$ given in Theorems \ref{casr}, \ref{casr2} and \ref{casgeneral}. In particular, in the context of good invariance components, the group structure appearing in Theorem \ref{casgeneral} is constructed naturally in term of these components, as we will see in Theorem \ref{casgeneralexpl}. 

Note that if $X$ is a process satisfying Definition \ref{self-simoursens} and $((f_y, c_y), \ y \in E)$ is a family of invariance components for $X$, relatively to some reference point $y_0$, then $((f_y \circ f_{y_0}^{-1}, c_y / c_{y_0}), \ y \in E)$ is a family of invariance components for $X$, relatively to the reference point $y_0$, and it satisfies \eqref{refpoint}. However, the existence of a family of invariance components that satisfies \eqref{relmorphisme} is much more difficult to prove, and is established in the following proposition. 

\begin{prop} [Existence of $\mathcal{C}^k$ good invariance components] \label{existsgoodcomponents}

For $k \geq 1$, let $X$ be a process satisfying Definition \ref{self-simoursens} with $\mathcal{C}^k$ invariance components on a connected $\mathcal{C}^k$-differentiable manifold $\mathcal{M}$ and let $y_0$ be an element of $\mathcal{M}$. We assume that either Assumption \ref{hypsupport1} or Assumption \ref{hypsupport2} is satisfied for $E = \mathcal{M}$ and that $Sym(X_{y_0})$ is discrete. Then $X$ admits a family of $\mathcal{C}^k$ good invariance components relatively to the reference point $y_0$. 

\end{prop}

Thanks to this proposition, we will sometimes assume that a process satisfying Definition \ref{self-simoursens} is given along with a family of good invariance components. 

For processes given along with good invariance components, we have the following results that can be thought of as explicit versions of Theorems \ref{casr}, \ref{casr2}, and \ref{casgeneral}. 

\begin{theo} [Explicit result in dimension $1$] \label{casrexpl}

For $k \geq 1$, let $X$ be a process satisfying Definition \ref{self-simoursens} with $\mathcal{C}^k$ invariance components on $I$, an open interval of $\mathbb{R}$, and let $((f_y, c_y), \ y \in I)$ be a family of \textbf{$\mathcal{C}^k$ good invariance components} for $X$, relatively to some reference point $y_0 \in I$. We assume that either Assumption \ref{hypsupport1} or Assumption \ref{hypsupport2} is satisfied for $E=I$. Let $g : I \longrightarrow \mathbb{R}$ and $\alpha \in \mathbb{R}$ be defined via 
\begin{eqnarray}
\forall y \in I, \ g(y) := \int_{y_0}^y \frac{1}{f_y' (y_0)} dy, \ \ \ \alpha := -\log(c_{g^{-1}(1)}), \label{defgdim1}
\end{eqnarray}
where $f_y'$ denotes the derivative of the function $f_y : I \longrightarrow I$. Then, $g$ and $\alpha$ are well-defined, $g$ is a $\mathcal{C}^k$-diffeomorphism from $I$ to $\mathbb{R}$, and there is a real L\'evy process $\xi$ such that if we set 
\begin{eqnarray}
\forall t \in [0, +\infty], \ \varphi(t) := \int_0^t e^{\alpha \xi(s)} ds, \label{timechangingdim1expl}
\end{eqnarray}
then 
\begin{eqnarray}
\zeta(X_{y_0}) = \varphi(+\infty) = \int_0^{+\infty} e^{\alpha \xi(s)} ds \label{tpsviedim1expl}
\end{eqnarray}
and 
\begin{eqnarray}
\forall \ 0 \leq t < \varphi(+\infty), \ X_{y_0}(t) = g^{-1} \left ( \xi \left (\varphi^{-1}(t) \right ) \right ). \label{represxdim1}
\end{eqnarray}




\end{theo}

\begin{remarque} \label{levyexpldim1}
In the above theorem $g$ is constructed directly from the family of good invariance components of $X$ and we see that the L\'evy process $\xi$ can be expressed as $\xi(t) = g(X_{y_0}(\varphi(t)))$. Moreover, as we will justify in the end of the proof of the theorem, $\varphi^{-1}(t)$ has an alternative expression in term of $X_{y_0}$ and of the good invariance components: for all $0 \leq t < \zeta(X_{y_0})$, $\varphi^{-1}(t) = \int_0^t c_{X_{y_0}(u)} du$. This, together with the expression $\xi(t) = g(X_{y_0}(\varphi(t)))$, allows to express explicitly $\xi$ in term of $X_{y_0}$ and of the good invariance components, as mentioned in Remark \ref{exittimealwaysfonctexpo}. Similarly, the $\xi$ in Theorem \ref{casr2expl} below can also be expressed explicitly in term of $X_{y_0}$ and of the good invariance components. 
\end{remarque}

Before stating the result in dimension $2$, let us introduce the notion of \textit{commutative invariance components}. 
\begin{defi} [Commutative invariance components] \label{commutativecomponents}
Let $X$ be a process satisfying Definition \ref{self-simoursens} on a locally compact separable metric space $E$. Let $((f_y, c_y), \ y \in E)$ be a family of invariance components for $X$, relatively to some reference point $y_0$. We say that $((f_y, c_y), \ y \in E)$ are \textbf{commutative invariance components} if they have the following property: 
\begin{eqnarray}
& \forall y, z \in E, \ f_y(z) = f_z(y). \label{commute}
\end{eqnarray}
\end{defi}

\begin{theo} [Explicit result in dimension $2$] \label{casr2expl}

For $k \geq 2$, let $X$ be a process satisfying Definition \ref{self-simoursens} on $\mathcal{D}$, an open simply connected domain of $\mathbb{R}^2$, and let $((f_y, c_y), \ y \in \mathcal{D})$ be a family of \textbf{$\mathcal{C}^k$ good invariance components} associated with $X$, relatively to some reference point $y_0 \in \mathcal{D}$. We assume that either Assumption \ref{hypsupport1} or Assumption \ref{hypsupport2} is satisfied for $E=\mathcal{D}$ and that $Sym(X_{y_0})$ is discrete. 
\begin{itemize}
\item If $((f_y, c_y), \ y \in \mathcal{D})$ are \textbf{commutative} invariance components, then let $Y_0$ be the gradient at $y_0$ of the application $(y \longmapsto c_{y})$. If $Y_0 \neq \begin{pmatrix}
0 \\
0
\end{pmatrix}$ let $M := \begin{pmatrix}
\pi_1(Y_0) & \pi_2(Y_0) \\
-\pi_2(Y_0) & \pi_1(Y_0)
\end{pmatrix}$, if $Y_0 = \begin{pmatrix}
0 \\
0
\end{pmatrix}$ let $M$ be the identity $2\times 2$ matrix. Finally, let $g : \mathcal{D} \longrightarrow \mathbb{R}^2$ and $\alpha \in \mathbb{R}$ be defined via 
\begin{eqnarray}
\forall y \in \mathcal{D}, \ g(y) := \int_a^b M. \left [ J f_{\gamma(s)}(y_0) \right ]^{-1}. \gamma'(s) ds, \ \ \ \alpha := -\log(c_{g^{-1}(1,0)}), \label{defgdim2com}
\end{eqnarray}
where $J f_z$ denotes the Jacobian matrix of the function $f_z(.) : \mathcal{D} \longrightarrow \mathcal{D}$ and where $\gamma : [a, b] \longrightarrow \mathcal{D}$ is any path, locally $\mathcal{C}^1$, with $\gamma(a) = y_0$ and  $\gamma(b) = y$. Then, $g$ and $\alpha$ are well-defined, $g$ is a $\mathcal{C}^k$-diffeomorphism from $\mathcal{D}$ to $\mathbb{R}^2$, and there is a L\'evy process $(\xi, \eta)$ on $\mathbb{R}^2$ such that if we set 
\begin{eqnarray}
\forall t \in [0, +\infty], \ \varphi(t) := \int_0^t e^{\alpha \xi(s)} ds, \label{timechangingdim2expl}
\end{eqnarray}
then 
\begin{eqnarray}
\zeta(X_{y_0}) = \varphi(+\infty) = \int_0^{+\infty} e^{\alpha \xi(s)} ds \label{tpsviedim2expl}
\end{eqnarray}
and 
\begin{eqnarray}
\forall \ 0 \leq t < \varphi(+\infty), \ X_{y_0}(t) = g^{-1} \left ( \xi \left (\varphi^{-1}(t) \right ), \eta \left (\varphi^{-1}(t) \right ) \right ). \label{represxdim2com}
\end{eqnarray}

\item If $((f_y, c_y), \ y \in \mathcal{D})$ are \textbf{not commutative} invariance components, then let $g : \mathcal{D} \longrightarrow \mathbb{R}^2$ be defined as in Lemma \ref{isomorphismecasnonclassique}, and $\alpha := -\log(c_{g^{-1}(1,0)}) \in \mathbb{R}$. Then, $g$ and $\alpha$ are well-defined, $g$ is a $\mathcal{C}^k$-diffeomorphism from $\mathcal{D}$ to $\mathbb{R}^2$, and there is a L\'evy process $(\xi, \eta)$ on $\mathbb{R}^2$ such that if we set 
\begin{eqnarray}
\forall t \in [0, +\infty], \ \varphi(t) := \int_0^t e^{\alpha \xi(s)} ds, \label{timechangingdim2expl2}
\end{eqnarray}
then 
\begin{eqnarray}
\zeta(X_{y_0}) = \varphi(+\infty) = \int_0^{+\infty} e^{\alpha \xi(s)} ds \label{tpsviedim2expl2}
\end{eqnarray}
and 
\begin{eqnarray}
\forall \ 0 \leq t < \varphi(+\infty), \ X_{y_0}(t) = g^{-1} \left ( \xi \left (\varphi^{-1}(t) \right ), \int_0^{\varphi^{-1}(t)} e^{\xi(s-)} d\eta(s) \right ). \label{represxdim2noncom}
\end{eqnarray}

\end{itemize}

\end{theo}

\begin{theo} \label{casgeneralexpl}

Let $X$ be a process satisfying Definition \ref{self-simoursens} on a connected locally compact separable metric space $E$, and let $((f_y, c_y), \ y \in E)$ be a family of \textbf{good invariance components} for $X$, relatively to some reference point $y_0$. We assume that either Assumption \ref{hypsupport1} or Assumption \ref{hypsupport2} is satisfied for $E$ and that $Sym(X_{y_0})$ is discrete. Let us define an interne composition law $\star$ on $E$ by $y \star x := f_y(x)$. Then: 

\begin{itemize}
\item $(E, \star)$ is a topological group with neutral element $y_0$ ; 
\item $(y \longmapsto c_y)$ is a continuous group homomorphism from $(E,\star)$ to $(\mathbb{R}_+^*, \times)$ ; 
\item There is a L\'evy process $L$ (starting from $y_0$) on $(E,\star)$ such that if we define 
$\forall t \in [0, +\infty], \varphi(t) := \int_0^t 1/c_{L(s)} ds$ then $\zeta(X_{y_0}) = \varphi(+\infty)$ and 
\begin{eqnarray}
\forall \ 0 \leq t < \varphi(+\infty), \ X_{y_0}(t) = L \left (\varphi^{-1}(t) \right ). \label{represgene}
\end{eqnarray} 
\end{itemize}
Let $\xi(.) := \log(1/c_{L(.)})$, then $\xi$ is clearly a real L\'evy process and we have $\zeta(X_{y_0}) = \int_0^{+\infty} e^{\xi(s)} ds$. 


If moreover, for $k \geq 1$, $E$ is a $\mathcal{C}^k$-differentiable manifold and $((f_y, c_y), \ y \in E)$ are $\mathcal{C}^k$ good invariance components, then $(E,\star)$ is even a $\mathcal{C}^k$-Lie group and $(y \longmapsto c_y)$ is a $\mathcal{C}^k$-Lie group homomorphism. 

\end{theo}

\begin{remarque} \label{levyexplarbspace}
As in Remark \ref{levyexpldim1}, let us mention that in the above theorem we see that the L\'evy process $L$ can be expressed as $L(t) = X_{y_0}(\varphi(t))$ and, as we will justify in the end of the proof of the theorem, $\varphi^{-1}(t)$ has an alternative expression in term of $X_{y_0}$ and of the good invariance components: for all $0 \leq t < \zeta(X_{y_0})$, $\varphi^{-1}(t) = \int_0^t c_{X_{y_0}(u)} du$. This, together with the expression $\xi(.) := \log(1/c_{L(.)})$, allows to express explicitly $\xi$ in term of $X_{y_0}$ and of the good invariance components, as mentioned in Remark \ref{exittimealwaysfonctexpo}. 
\end{remarque}

Recall that the family of invariance components defined in Proposition \ref{levychangentpssensfacil} always satisfies Definition \ref{goodcomponents} so it is actually a family of good invariance components. Therefore, Proposition \ref{levychangentpssensfacil} can be seen as the reciprocal of the above theorem. 

\begin{remarque} \label{startingpoint}
In the representations \eqref{represxdim1}, \eqref{represxdim2com}, \eqref{represxdim2noncom}, and \eqref{represgene} from the above theorems, if one replaces $\xi$, $(\xi, \eta)$, $( \xi ( . ), \int^{.} e^{\xi(s-)} d\eta(s) )$, and $L$ by respectively $\psi^{-1}(y) + \xi$, $\psi^{-1}(y) + (\xi, \eta)$, $(\pi_1(\psi^{-1}(y)) + \xi(.), \pi_2(\psi^{-1}(y)) + \int_0^{.} e^{\pi_1(\psi^{-1}(y)) + \xi(s-)} d\eta(s))$, and $y \star L$ (where, respectively, $y \in I$, $y \in \mathcal{D}$, $y \in \mathcal{D}$ and $y \in E$)
then on obtains a version of 
$X_{y}$. 


\end{remarque}

As we said above, the results of this subsection are not only explicit versions of Theorems \ref{casr}, \ref{casr2} and \ref{casgeneral}, but they are also the main ingredients for their proofs. Indeed, as we will see, the direct part of Theorem respectively \ref{casr}, \ref{casr2} and \ref{casgeneral} is a consequence of Proposition \ref{existsgoodcomponents} together with Theorem respectively \ref{casrexpl}, \ref{casr2expl} and \ref{casgeneralexpl}. 

\subsection{Sketch of proof and organization of the paper} \label{sketchofproof} 

\ 

\textit{In the remaining part of this section} we discuss the case of dimension greater or equal to $3$ and then we present some examples and particular cases, for which the results above take a simple or nice form. In particular we present an example in dimension $2$ of a process built from a self-similar fragmentation process, for which the decomposition of Theorem \ref{casr2} in term of the exponential functional of a bivariate L\'evy process appears naturally and is meaningful in term of the underlying self-similar fragmentation process. 

\textit{In Section \ref{prelres}} we prove some preliminary results that are crucial for the rest of the paper. More precisely we prove Proposition \ref{compatibility} which is a compatibility relation that is satisfied by invariance components of a process satisfying Definition \ref{self-simoursens}. 
We then use that relation to prove Proposition \ref{existsgoodcomponents} which ensures the existence of good invariance components. 

\textit{In Section \ref{relwithgroups}} we first prove Proposition \ref{levychangentpssensfacil}. Then, we use the compatibility relation from Proposition \ref{compatibility} to prove Proposition \ref{groupappears} which makes appear a group structure on the state space of a process $X$ satisfying Definition \ref{self-simoursens} and that admits good invariance components. Then, we prove Proposition \ref{levychangentps} which identifies $X$ with a time-changed L\'evy process on its state space equipped with the above-mentioned group structure. We then prove Theorems \ref{casgeneralexpl} and \ref{casgeneral}. 

\textit{In Section \ref{explicitisom}} we construct some explicit Lie group isomorphisms between the Lie groups constructed in the previous section (when the state space is an open interval of $\mathbb{R}$ or an open simply connected domain of $\mathbb{R}^2$) and some canonical Lie groups in dimension $1$ and $2$. Thanks to that, a process that satisfies our assumptions on an open interval of $\mathbb{R}$, or on an open simply connected domain of $\mathbb{R}^2$, can be expressed explicitly as a function of a time-changed L\'evy process on one of these canonical groups. 

\textit{In Section \ref{levysurlegroupnoncomdim2}} we prove Proposition \ref{levynonclassiques} that gives a convenient expression of L\'evy processes on the most complicated of these canonical Lie groups, in term of exponential functionals of bivariate L\'evy processes. 

\textit{In Section \ref{proofsofmainth}} we put the pieces together to prove Theorems \ref{casrexpl} and \ref{casr2expl}. Then, Theorems \ref{casr} and \ref{casr2} follow from those theorems together with Proposition \ref{existsgoodcomponents}. 

\textit{In Section \ref{technical}} we prove some technical lemmas about the assumptions discussed in the end of Subsection \ref{introgssmp}. 

\subsection{The case of dimension greater or equal to $3$} \label{dim>2}

We now briefly describe which points in the present methodology can be generalized in dimension $d \geq 3$ in order to obtain results analogous to Theorems \ref{casrexpl}, \ref{casr}, \ref{casr2expl} and \ref{casr2} and which difficulties make the problem much more difficult (and maybe not possible) to solve for those dimensions. 

In the case where the state space of a process $X$ satisfying Definition \ref{self-simoursens} is $\mathcal{D}$, a simply connected open subset of $\mathbb{R}^d$ for some $d \geq 3$, the following can still be applied: under some assumptions, Proposition \ref{existsgoodcomponents} allows to build good invariance components and Theorem \ref{casgeneralexpl} to identify a Lie group structure $(\mathcal{D}, \star)$ of dimension $d$ on $\mathcal{D}$ for which $X$ is a time-changed L\'evy process. Then, one would need to choose a 'canonical' representative for each isomorphism class of Lie group structure of dimension $d$, to build an explicit isomorphism between $(\mathcal{D}, \star)$ and its representative (in particular one needs to be able to identify the isomorphism class of $(\mathcal{D}, \star)$ from the good invariance components of $X$) and to obtain a convenient expression for L\'evy processes on each of these representative Lie groups. We see that in dimension $d$, there are possibly as many possibilities as the number of Lie group structures of dimension $d$, up to isomorphism. For $d = 1$ this number is $1$ (which is why there is only one case in Theorem \ref{casrexpl}), for $d = 2$ this number is $2$ (which corresponds to the two cases in Theorem \ref{casr2expl}). In dimension $d \geq 3$ there are much more and several problems arise. First, for $d \geq 3$, it is not clear whether each isomorphism class of Lie groups of dimension $d$ can appear among the Lie group structures defined on $\mathcal{D}$ by families of good invariance components of processes. If not, one needs to characterize the isomorphism classes that can appear among all isomorphism classes of Lie groups of dimension $d$. Then, it is not clear either whether, for each isomorphism class of Lie group of dimension $d$, we can build explicit isomorphisms to a canonical representative of this class as we did in dimension $1$ and $2$, and whether it is possible to find a convenient representation of L\'evy processes on the canonical representative of this class in term of functionals of L\'evy processes on $\mathbb{R}^d$ (if the answer is positive, we are not able to anticipate the kind of objects by which exponential functionals of bivariate L\'evy processes should be replaced). Finally, for the expected representations in term of functionals of L\'evy processes on $\mathbb{R}^d$, it is not clear whether it is possible to unify all these representations into a single one, as we did in dimension $2$. Moreover, for any $d \geq 3$, the number of isomorphism classes of Lie groups of dimension $d$ is very large (even if finite) and there is unfortunately no automatic way that allows to treat several isomorphism classes at once, which makes the problem very difficult to solve in any dimension $d \geq 3$. If the problem can be solved in dimension $d \geq 3$, the representations obtained will certainly be much more complicated than those obtained in the present article and depend on much more parameters, and it would be interesting to see whether some particular examples can be produced in which these general representations have a natural interpretation (as for the example produced in the next section for dimension $2$). 


%

\subsection{Some examples and particular cases} \label{examples}

A basic example is the one of a pssMp $X$ with index $\alpha \in \mathbb{R}$. In that case, the state space is the interval $E = ]0, +\infty[$. By definition of a pssMp of index $\alpha$, a family of invariance components $((f_y, c_y), \ y \in ]0, +\infty[)$, with respect to the reference point $1$, is given by $f_y(x) = yx$ and $c_y = y^{-\alpha}$. It is immediate that this family of invariance components is \textit{good} (i.e. it satisfies Definition \ref{goodcomponents}) and of class $\mathcal{C}^{\infty}$. Then, the function $g$ defined in \eqref{defgdim1} is equal to $\log(.)$ and the real number $\alpha$ defined in \eqref{defgdim1} coincides with the self-similarity index $\alpha$ of the pssMp $X$. The application of Theorem \ref{casrexpl} yields that, provided that Assumption \ref{hypsupport1} or Assumption \ref{hypsupport2} is satisfied for $E=]0, +\infty[$, we have 
\[ \zeta(X_{y_0}) = \varphi(+\infty) \ \ \ \text{and} \ \ \ \forall \ 0 \leq t < \varphi(+\infty), \ X_1(t) = \exp \left ( \xi \left (\varphi^{-1}(t) \right ) \right ), \]
where $\xi$ is a real L\'evy process and $\varphi(t) := \int_0^t e^{\alpha \xi(s)} ds$. We have thus retrieved Lamperti's representation for $X$. 

%
%

Another interesting particular case is the case where the time-changes are trivial, i.e. for a process $X$ satisfying Definition \ref{self-simoursens}, there exists a family of invariance components $((f_y, c_y), \ y \in E)$ such that $\forall y \in E, c_y = 1$. In this case the self-similarity relation simply becomes $\forall y \in E, X_y \overset{\mathcal{L}}{=} f_y(X_{y_0})$. Replacing $f_y$ by $f_y \circ f_{y_0}^{-1}$ if necessary, it is plain that the family $((f_y, 1), \ y \in E)$ satisfies Definition \ref{goodcomponents}, so it is a family of good invariance components. 
\begin{itemize}
\item When $E=I$, an open interval of $\mathbb{R}$. If the family of good invariance components $((f_y, 1), \ y \in I)$ is $\mathcal{C}^k$ for some $k \geq 1$, and if Assumption \ref{hypsupport1} or Assumption \ref{hypsupport2} is satisfied for $E=I$, then Theorem \ref{casrexpl} applies, and the $\alpha$ defined in \eqref{defgdim1} clearly equals $0$ so that the time-change function is trivial : $\varphi(t) = t$. Therefore, $X_{y_0}$ is only the image by $g^{-1}$ (where $g$ is defined in \eqref{defgdim1}) of a real L\'evy process $\xi$. 
\item When $E=\mathcal{D}$, an open simply connected domain of $\mathbb{R}^2$. If the family of good invariance components $((f_y, 1), \ y \in \mathcal{D})$ is $\mathcal{C}^k$ for some $k \geq 2$, if Assumption \ref{hypsupport1} or Assumption \ref{hypsupport2} is satisfied for $E=\mathcal{D}$, and if $Sym(X_{y_0})$ is discrete, then Theorem \ref{casr2expl} applies, and the $\alpha$ defined there equals $0$ so that the time-change function is trivial : $\varphi(t) = t$. Therefore, if $((f_y, 1), \ y \in \mathcal{D})$ are \textit{commutative}, $X_{y_0}$ is the image by $g^{-1}$ (where $g$ is defined in \eqref{defgdim2com}) of $(\xi, \eta)$, that is a L\'evy process on $\mathbb{R}^2$. If $((f_y, 1), \ y \in I)$ are \textit{not commutative}, $X_{y_0}$ is the image by $g^{-1}$ (where $g$ is defined in Lemma \ref{isomorphismecasnonclassique}) of $( \xi ( . ), \int^{.} e^{\xi(s-)} d\eta(s) )$, where $(\xi, \eta)$ is a L\'evy process on $\mathbb{R}^2$. 
\item In a more general context, if the assumptions of Theorem \ref{casgeneralexpl} are satisfied ($E$ is connected, Assumption \ref{hypsupport1} or Assumption \ref{hypsupport2} is satisfied, and $Sym(X_{y_0})$ is discrete), then the theorem applies and we see that the time-change function is trivial : $\varphi(t) = t$. Therefore $X_{y_0}$ is a L\'evy process on $E$, equipped with the group structure defined in the theorem. 
\end{itemize}


{
Let us now discuss an example, built from self-similar fragmentation, of a process that satisfies Definition \ref{self-simoursens} and for which the decomposition of Theorem \ref{casr2} appears naturally and is meaningful. For the sake of introducing this example, we recall some background on self-similar fragmentation processes. We consider a self-similar fragmentation process $X$. More precisely, $X$ is a Feller process on the space of mass-partitions  
\[ \mathcal{P} := \left \{ (x_1, x_2, ...), \ x_1 \geq x_2 \geq ... \geq 0, \ M(x_1, x_2, ...) := \sum_{i=1}^{+\infty} x_i < +\infty \right \}, \]
and $X$ satisfies 0) $M(X(t))$ is almost surely non-increasing (no mass creation), 1) the self-similarity property, 2) the branching property. Let us describe the meaning of properties 1) and 2). For $x \geq 0$, $\mathbb{P}_x$ denotes the law of the process $X$ starting from the state $(x, 0, ...)$. $X$ is said to be self-similar with index $\alpha$ if for every $x \geq 0$, the process $(x X(x^{\alpha} t), t \geq 0)$ under $\mathbb{P}_1$ is equal in law to the process $(X(t), t \geq 0)$ under $\mathbb{P}_x$. $X$ is said to fulfill the branching property if for every $(x_1, x_2, ...) \in \mathcal{P}$, the process $X$ starting from the state $(x_1, x_2, ...)$ is equal in law to the the process resulting from the non-increasing rearrangement of fragments of $X^{(1)}, X^{(2)},...$ where $X^{(1)}, X^{(2)},...$ are independent versions of $X$ with law respectively $\mathbb{P}_{x_1}, \mathbb{P}_{x_2},...$ 

Roughly speaking, in this model, a piece of mass $x > 0$ can split into fragments whose sum of the masses is less or equal to $x$. The nature and the rate at which the splittings occur is determined by a $\sigma$-finite measure on 
\[ \mathcal{P}_1 := \left \{ (x_1, x_2, ...), \ x_1 \geq x_2 \geq ... \geq 0, \ \sum_{i=1}^{+\infty} x_i \leq 1\right \}, \]
that we denote by $\nu$ and that is such that $\nu(\{(1,0,...)\})=0$ and $\int(1-x_1)\nu(dx) < +\infty$. $\nu$ is called the \textit{dislocation measure} of $X$ and characterizes the law of $X$. Reciprocally, for any $\sigma$-finite measure $\nu$ on $\mathcal{P}_1$ that satisfies these two conditions one can construct $X$, a self-similar fragmentation process with index $\alpha$ and dislocation measure $\nu$, see Chapter 3 in \cite{bertoin_2006} for details. Note that there are splittings at arbitrary small times if $\nu(\mathcal{P}_1)=+\infty$. If $\nu(\mathcal{P}_1)<+\infty$, the construction of $X$ is simpler: Under $\mathbb{P}_0$ the process stays for ever in the state $(0,0,...)$. Under $\mathbb{P}_x$ for $x > 0$ we have $X(t) = (x, 0, ...)$ until an exponential time with parameter $x^{\alpha} \nu(\mathcal{P}_1)$ where there is transition to a state $(x x_1, x x_2, ...)$ where $(x_1, x_2, ...) \in \mathcal{P}_1$ is chosen according to the law $\nu(.)/\nu(\mathcal{P}_1)$. This together with the branching property determines the dynamic of the process. 

It is known that, for an $\alpha$-self-similar fragmentation process $X$, we can construct a process $Y$ where $Y(t)$ represents the mass at time $t$ of a tagged fragment. If $\nu(\mathcal{P}_1)<+\infty$, the construction of the "tagged fragment process" and of $Y$ can be done as follows: Under $\mathbb{P}_x$ for $x > 0$, the tagged fragment is the only existing fragment (of mass $x$) until the first splitting, and after a splitting of the tagged fragment (into masses $x x_1, x x_2, ...$ where $(x_1, x_2, ...) \in \mathcal{P}_1$), the tagged fragment dies with probability $1 - \sum_{i \geq 1} x_i$ (note that $x(1 - \sum_{i \geq 1} x_i)$ is the mass that has been dissipated in the splitting), or the new tagged fragment is the $i^{th}$ fragment, among those resulting from the splitting, with probability $x_i$ (note that $x x_i$ is the mass of this fragment). Then, $Y(t)$ is the mass of the tagged fragment at time $t$. If the tagged fragment is dead at time $t$ we set $Y(t) = 0$ and the absorbing state $0$ is the cemetery point for the process $Y$ on $]0, +\infty[$. Note that, if there is mass dissipation, then $Y$ has the same positive probability to jump to $0$ at each jump and is therefore almost surely killed in finite time. If there is no mass dissipation, then there is no instantaneous killing for $Y$, but if moreover $\alpha < 0$, then $X$ reaches continuously the absorbing state $(0,0,...)$ in finite time (see Section 1.3 in \cite{bertoin_2006}) so $Y$ reaches the cemetery point $0$ in finite time so it has almost surely finite life-time. It is known that the process $Y$ can still be defined when $\nu(\mathcal{P}_1)=+\infty$, even though defining the "tagged fragment process" is then more abstract, see Chapter 3 in \cite{bertoin_2006}. Moreover, it is known (see Chapter 3 in \cite{bertoin_2006}) that $1/Y$ is a (possibly killed) pssMp with index $\alpha$, so $Y$ is a (possibly killed) pssMp with index $-\alpha$, and that the L\'evy process $\xi$ associated with $1/Y$ in the Lamperti representation is a (possibly killed) pure jump subordinator with L\'evy measure 
\[ \pi (dz) = e^{-z} \sum_{i \geq 1} \nu \left ( -\log x_i \in dz \right ), \]
and killing rate $\int_{\mathcal{P}_1} ( 1 - \sum_{j \geq 1} x_j ) d\nu$. 

For the sake of our example, let us construct another process $Z$ associated with the $\alpha$-self-similar fragmentation process $X$. $Z(t)$ represents the contribution of the tagged fragment to mass dissipation until time $t$. Similarly to $Y$, $Z$ might be interesting in some aspects of the understanding of $X$. For simplicity, let us from now one only consider self-similar fragmentation processes with finite dislocation measure. We consider the "tagged fragment process" defined above. $Z(0) = 0$, $Z$ stays constant between splitting times of the tagged fragment and, each time the tagged fragment splits, $Z$ is increased by the amount of mass dissipated in this splitting. For example, if just before a splitting at a time $t$ the mass of the tagged fragment $Y(t-)$ is equal to a value $x > 0$, and if $x x_1, x x_2, ...$ are the masses of the descents of the tagged fragment at the moment of the splitting, then 
$Z(t) = Z(t-) + x - \sum_{j \geq 1} x x_j = Z(t-) + Y(t-) (1- \sum_{j \geq 1} x_j)$. $Z$ no longer increases if the tagged fragment is dead at time $t$, so the formula $Z(t) = Z(t-) + Y(t-) (1- \sum_{j \geq 1} x_j)$ is still valid when $Y$ is in the absorbing state $0$. 

Unlike $Y$, $Z$ is not a self-similar Markov process. However, the following simple proposition ensures that $(Y,Z)$ is a (possibly killed) process that satisfies Definition \ref{self-simoursens} and for which an explicit representation in term of the exponential functional of a bivariate L\'evy process is available. 

\begin{prop} \label{exemplefrag}

The process $(Y,Z)$ is a (possibly killed) process that satisfies Definition \ref{self-simoursens} 
on $E = ]0, +\infty[ \times [0, +\infty[$. A family of invariance components is given by $((f_y, c_y), \ y \in E)$ with $f_y(x) = (\pi_1(y) \pi_1(x), \pi_2(y) + \pi_1(y) \pi_2(x))$ and $c_y = e^{\alpha \pi_1(y)}$ for the reference point $(1,0)$. Let $(\xi, \eta)$ be a (possibly killed) pure jump L\'evy process on $\mathbb{R}^2$ with L\'evy measure $\Pi (., .)$ where 
\begin{align}
\Pi (da, db) = e^{-a} \sum_{i \geq 1} \nu \left (- \log x_i \in da, 1-\sum_{j \geq 1} x_j \in db \right ), \Pi (\{ + \infty \}, db) & = b \nu \left (1-\sum_{j \geq 1} x_j \in db \right ). \label{mesbivariate} 
\end{align}
In particular, each time $\eta$ makes a jump, $\xi$ dies with a probability equal to the size of this jump so that the killing rate of $\xi$ is $\int_{\mathcal{P}_1} ( 1 - \sum_{j \geq 1} x_j ) d\nu$. Let $T(\xi) := \inf \{ t \geq 0, \ \xi(t) = +\infty \}$ denote the killing time of $\xi$ (set $T(\xi) = +\infty$ if the killing rate of $\xi$ is null). For $x > 0$ let 
\[ \forall t \in [0, T(\xi)], \ \varphi_x(t) := x^{-\alpha} \int_0^t e^{\alpha \xi(s)} ds. \] 
Then, under $\mathbb{P}_x$, $(Y,Z)$ is equal in distribution to the process defined by 
\[ (\tilde Y(t), \tilde Z(t)) := \left ( x e^{-\xi(\varphi_x^{-1}(t))}, x \int_0^{\varphi_x^{-1}(t)} e^{-\xi(s-)} d\eta(s) \right ), \ 0 \leq t \leq \varphi_x(T(\xi)), \]
which has life-time $\varphi_x(T(\xi))$. In particular, the total contribution of the tagged fragment to mass dissipation is equal in distribution to $x \int_0^{T(\xi)} e^{-\xi(s-)} d\eta(s)$. 
\end{prop} 

The proof of this proposition is very intuitive but we give all the details for the sake of clarity. 

\begin{proof}

By definition of $Y$ and $Z$ we have for any $t,s \geq 0$ that $(Y(t+s), Z(t+s)) = (\hat Y(s), Z(t) + \hat Z(s))$ where $(\hat Y, \hat Z)$ is define from $\hat X$ just as $(Y, Z)$ is define from $X$, where $\hat X$ is the fragmentation process restricted to the tagged fragment at time $t$ and its descents. By the branching property for $X$ we have that $\hat X$ follows the distribution $\mathbb{P}_{Y(t)}$ so $(Y,Z)$ is strongly Markovian. Then, from the definition of $(Y,Z)$ and the self-similarity of $X$ we see that $(Y,Z)$ satisfies Definition \ref{self-simoursens} with the asserted invariance components. 

Recall that we are in the case $\nu(\mathcal{P}_1)<+\infty$ so we can define $T_1 < T_2 < ...$ to be the (possibly finite) sequence of the splitting times of the tagged fragment (note that this sequence is infinite is and only if there is no dissipation and in this case, note that it has an accumulation point if and only if $\alpha < 0$). By convention we put $T_0 := 0$. Let $(P_k)_{k \geq 1}$ be the (possibly finite) sequence of elements of $\mathcal{P}_1$ that represent the successive splittings of the tagged fragment, that is, if $P_k = (S_1^{(k)}, S_2^{(k)}, ...) \in \mathcal{P}_1$ it means that at instant $T_k$ the tagged fragment of size $Y(T_k-)$ splits into fragments of size $Y(T_k-) S_1^{(k)} \geq Y(T_k-) S_2^{(k)} \geq ...$ and the $i^{th}$ fragment in this sequence becomes the new tagged fragment with probability $S_i^{(k)}$, or the tagged fragment dies with probability $1 - \sum_{i \geq 1} S_i^{(k)}$. 
Let us denote by $I(k)$ the index of the fragment that is chosen to become the new tagged fragment at the splitting time $T_k$. We put $I(k) = 0$ and $S_{I(k)}^{(k)} = 0$ if the tagged fragment dies at $T_k$. 

For $t \in [0, +\infty]$, let us define $\phi(t) := \int_0^t Y(s)^{\alpha} ds$ so for any $T_k$ we have $\phi(T_k) - \phi(T_{k-1}) = Y(T_{k-1})^{\alpha} (T_k - T_{k-1})$. 
The latter follows an exponential distribution with parameter $\nu(\mathcal{P}_1)$, since a fragment of mass $Y(T_{k-1})$ splits after an exponential time of parameter $Y(T_{k-1})^{\alpha} \nu(\mathcal{P}_1)$. Therefore, $\phi(T_1) < \phi(T_2) < ...$ is the sequence of the arrival times of a (possibly killed) standard Poisson process. In particular, if we time-change the tagged fragment process by $\phi^{-1}(.)$ and focus on the elements of $\mathcal{P}_1$ that represent the splittings, then we obtain a (possibly killed) Poisson point process on $[0, +\infty[ \times \mathcal{P}_1$ with intensity measure $dt \times \nu$. Let us define $(\tilde \xi, \tilde \eta)$ by $\forall t \geq 0$, 
\[ \tilde \xi(t) := \sum_{ \{ k \geq 1, \ T_k \leq \phi^{-1}(t) \} } - \log(S_{I(k)}^{(k)}), \ \ \ \tilde \eta(t) := \sum_{ \{ k \geq 1, \ T_k \leq \phi^{-1}(t) \} } \left ( 1 - \sum_{j \geq 1} S_j^{(k)} \right ), \]
where by convention $- \log(0) = +\infty$. Then, clearly, $(\tilde \xi, \tilde \eta)$ is a (possibly killed) compound Poisson process on $\mathbb{R}^2$ with L\'evy measure given by \eqref{mesbivariate}. 
We see that, under $\mathbb{P}_x$, $Y(\phi^{-1}(t))$ and $Z(\phi^{-1}(t))$ can be expressed in the following way: 
\begin{align*}
Y(\phi^{-1}(t)) & = x \prod_{ \{ k \geq 1, \ T_k \leq \phi^{-1}(t) \} } S_{I(k)}^{(k)} = x \exp \left ( - \sum_{ \{ k \geq 1, \ T_k \leq \phi^{-1}(t) \} } - \log(S_{I(k)}^{(k)}) \right ) = xe^{-\tilde \xi(t)}, \\ 
Z(\phi^{-1}(t)) & = \sum_{ \{ k \geq 1, \ T_k \leq \phi^{-1}(t) \} } Z(T_k) - Z(T_k-) = \sum_{ \{ k \geq 1, \ T_k \leq \phi^{-1}(t) \} } Y(T_k-) \left ( 1 - \sum_{j \geq 1} S_j^{(k)} \right ) \\ 
& = \sum_{ \{ k \geq 1, \ T_k \leq \phi^{-1}(t) \} } x e^{-\tilde \xi(T_k-)} ( \tilde \eta(T_k) - \tilde \eta(T_k-) ) = x \int_0^{t} e^{-\tilde \xi(s-)} d\tilde \eta(s). 
\end{align*}
Then, for any $t \geq 0$ we have 
\[ \phi^{-1}(t) = \int_0^{\phi^{-1}(t)} 1 du = \int_0^{t} \frac1{\phi'(\phi^{-1}(v))} dv = \int_0^{t} \frac1{Y(\phi^{-1}(v))^{\alpha}} dv = x^{-\alpha} \int_0^{t} e^{\alpha \tilde \xi(v)} dv = \varphi_x(t), \]
Combining with the above we deduce that $Y(\varphi_x(.)), Z(\varphi_x(.)) = (xe^{-\tilde \xi(.)}, x \int_0^{.} e^{-\tilde \xi(s-)} d\tilde \eta(s))$ so $(Y,Z)$ is equal in distribution to the process $(\tilde Y, \tilde Z)$ defined in the statement of the proposition. 

\end{proof}

It would have been too long to define rigorously the process $Z$ is the case $\nu(\mathcal{P}_1)=+\infty$ but this can be done similarly as for $Y$ and it is possible to see that Proposition \ref{exemplefrag} still holds in that case. 

The representation obtained in Proposition \ref{exemplefrag} for the process $(Y,Z)$ coincides with the one provided by Theorem \ref{casr2} (even if all technical assumptions of Theorem \ref{casr2} are not satisfied in the example of $(Y,Z)$). In the above proof of the representation for the process $(Y,Z)$, the exponential functional of the bivariate L\'evy process $(\xi, \eta)$ appearing in the representation was interpreted very simply as follows: by self-similarity the tagged fragment dissipates proportions of its masses that are idd (these proportion are given by the jumps of $\eta$) ; to get the actual total mass dissipated by the tagged fragment, one needs to multiply these proportions of mass dissipated by the mass of the tagged fragment at the dissipation times (this mass is given by $\exp(-\xi)$) and to sum, which results in the exponential functional. 

From the example of this section, we see that self-similarity in the meaning of Definition \ref{self-simoursens} can appear naturally in dimension greater than $1$ in the context of a "self-similar system" for which several interesting quantities are correlated and that, in this case, the decomposition provided by Theorem \ref{casr2} is meaningful in term of the nature of the system. Also, it is remarkable that the kind of representation (in term of the exponential functional of a bivariate L\'evy process) obtained in this particular example is actually the most general one for Markovian processes in dimension $2$ that satisfy self-similarity properties, as proved by Theorem \ref{casr2}, even if the natural interpretation of the exponential functional is no longer true for general processes. This, we believe, strengthens the interest in proving Theorem \ref{casr2} that makes appear exponential functionals in a general context. 


We could have considered the case of a fragmentation with erosion (see Chapter 3 in \cite{bertoin_2006} for details). In this case, only the bivariate L\'evy process $(\xi, \eta)$ has to be modified for Proposition \ref{exemplefrag} to remain true: Let $c > 0$ be the erosion coefficient, we take $(\xi, \eta)$ with the same L\'evy measure as in \eqref{mesbivariate} but the killing rate of $\xi$ is now $c + \int_{\mathcal{P}_1} ( 1 - \sum_{j \geq 1} x_j ) d\nu$ (i.e. we add an extra killing term $c$) and the drift components is now $(c,c)$ (while it was $(0,0)$ in the proposition). 
}

\section{Compatibility relation and construction of good components} \label{prelres}

We now prove a key result that yields a compatibility relation satisfied by invariance components of a process satisfying Definition \ref{self-simoursens}. This result is important to establish Proposition \ref{existsgoodcomponents} and then to make appear the structure of group in Proposition \ref{groupappears}. In the logic of Remark \ref{cemetery}, note that in the next two propositions, their proofs, and where they are applied, for the case where the processes involved reach the cemetery point $\Delta$ in finite time, we actually work with the extension of homeomorphisms to $E \cup \{ \Delta \}$ that send the cemetery point $\Delta$ on itself (but, for example, the extension of $f_y(.)$ is still denoted by $f_y(.)$ for simplicity). 


\begin{prop} \label{compatibility}

Let $E$ be a locally compact separable metric space, and let $X$ be a process satisfying Definition \ref{self-simoursens} on $E$ with invariance components $((f_y, c_y), \ y \in E)$ relatively to some reference point $y_0 \in E$. We assume that Assumption \ref{hypsupport1} is satisfied. Then we have 
\[ \forall y_1, y_2 \in E, \ f_{f_{y_1}(y_2)} \left ( X_{y_0}(c_{f_{y_1}(y_2)} \times .) \right ) \overset{\mathcal{L}}{=} f_{y_1} \circ f_{y_2} \left ( X_{y_0} (c_{y_1} \times c_{y_2} \times .) \right ). \]

\end{prop}

\begin{proof}

Assume that $X$ and $((f_y, c_y), \ y \in E)$ are as in the statement of the proposition. Let $y_1, y_2 \in E$ be arbitrary. As a consequence of Definition \ref{self-simoursens} we have $Supp (X_{y_1}) = f_{y_1}(Supp (X_{y_0}))$. Because of Assumption \ref{hypsupport1} $Supp (X_{y_0}) = E$ so, since $f_{y_1}$ is an homeomorphism, we get $Supp (X_{y_1}) = E$. Let us define $a := f_{y_1}(y_2)$ and fix $\epsilon > 0$. Since $a \in E = Supp (X_{y_1})$, and by definition of $Supp (X_{y_1})$, there exists $s_{\epsilon} \geq 0$ such that $P(X_{y_1}(s_{\epsilon}) \in B(a,\epsilon)) > 0$. Let us fix such a $s_{\epsilon} \geq 0$. We consider the process 
$( X_{y_1} (s_{\epsilon} + t), \ t \geq 0 )$ 
conditionally on the event $\{ X_{y_1}(s_{\epsilon}) \in B(a,\epsilon) \}$ that has positive probability. Let $\tilde X_{y_0}$ and $\hat X_{y_0}$ be independent with law $P_{y_0}$, and be independent of $X_{y_1}$. On one hand, by the Markov property at $s_{\epsilon}$ and the self-similarity (Definition \ref{self-simoursens}) we have 
\begin{eqnarray}
\mathcal{L} \left ( X_{y_1}(s_{\epsilon}+.) \big | X_{y_1}(s_{\epsilon}) \in B(a,\epsilon) \right ) = \mathcal{L} \left ( f_{X_{y_1}(s_{\epsilon})} \left ( \tilde X_{y_0}(c_{X_{y_1}(s_{\epsilon})} \times .) \right ) \big | X_{y_1}(s_{\epsilon}) \in B(a,\epsilon) \right ). \label{compatibility1}
\end{eqnarray}
On the other hand, by applying the self-similarity (Definition \ref{self-simoursens}) at the starting point $y_1$, and then the Markov property at time $c_{y_1} s_{\epsilon}$ and the self-similarity (Definition \ref{self-simoursens}) at $\hat X_{y_0}(c_{y_1} s_{\epsilon})$, we have 
\begin{align}
& \mathcal{L} \left ( X_{y_1}(s_{\epsilon}+.) \big | X_{y_1}(s_{\epsilon}) \in B(a,\epsilon) \right ) \nonumber \\
= & \mathcal{L} \left ( f_{y_1} \left ( \hat X_{y_0}(c_{y_1} \times (s_{\epsilon}+.)) \right ) \big | f_{y_1} \left ( \hat X_{y_0}(c_{y_1} s_{\epsilon}) \right ) \in B(a,\epsilon) \right ) \nonumber \\
= & \mathcal{L} \left ( f_{y_1} \left ( \hat X_{y_0}(c_{y_1} \times (s_{\epsilon}+.)) \right ) \big | \hat X_{y_0}(c_{y_1} s_{\epsilon}) \in f_{y_1}^{-1} \left ( B(a,\epsilon) \right ) \right ) \nonumber \\
= & \mathcal{L} \left ( f_{y_1} \circ f_{\hat X_{y_0}(c_{y_1} s_{\epsilon})} \left ( \tilde X_{y_0} (c_{y_1} \times c_{\hat X_{y_0}(c_{y_1} s_{\epsilon})} \times .) \right ) \big | \hat X_{y_0}(c_{y_1} s_{\epsilon}) \in f_{y_1}^{-1} \left ( B(a,\epsilon) \right ) \right ). \label{compatibility2}
\end{align}

Let $\epsilon$ decrease to $0$ via a countable sequence (for each of these $\epsilon$ we can chose an $s_{\epsilon} \geq 0$ such that $P(X_{y_1}(s_{\epsilon}) \in B(a,\epsilon)) > 0$, and the above procedure shows that \eqref{compatibility1} and \eqref{compatibility2} hold). In the right hand side of \eqref{compatibility1}, we have the law of a function of the process $\tilde X_{y_0}$ and of an independent random variable $A_{\epsilon}$ where $\mathcal{L} (A_{\epsilon}) = \mathcal{L} (X_{y_1}(s_{\epsilon}) | X_{y_1}(s_{\epsilon}) \in B(a,\epsilon))$. Clearly, $A_{\epsilon}$ converges in distribution to $a$ when $\epsilon$ goes to $0$. Using Slutsky's Lemma and the continuity of $(y,x) \longmapsto f_y(x)$ and of $y \longmapsto c_y$ assumed in Definition \ref{self-simoursens}, we deduce that the right hand side of \eqref{compatibility1} converges (in the meaning of finite-dimensional distributions) to $\mathcal{L} ( f_{a} ( \tilde X_{y_0}(c_{a} \times .) ) )$. 

Similarly, in the right hand side of \eqref{compatibility2}, we have the law of a function of the process $\tilde X_{y_0}$ and of an independent random variable $B_{\epsilon}$ where $\mathcal{L} (B_{\epsilon}) = \mathcal{L} (\hat X_{y_0}(c_{y_1} s_{\epsilon}) | \hat X_{y_0}(c_{y_1} s_{\epsilon}) \in f_{y_1}^{-1} \left ( B(a,\epsilon) \right ))$. By continuity of $f_{y_1}^{-1}$, $B_{\epsilon}$ converges in distribution to $f_{y_1}^{-1}(a)$ when $\epsilon$ goes to $0$. Using Slutsky's Lemma and the continuity of $(y,x) \longmapsto f_y(x)$ and of $y \longmapsto c_y$ assumed in Definition \ref{self-simoursens}, we deduce that the right hand side of \eqref{compatibility2} converges (in the meaning of finite-dimensional distributions) to $\mathcal{L} ( f_{y_1} \circ f_{f_{y_1}^{-1}(a)} ( \tilde X_{y_0} (c_{y_1} \times c_{f_{y_1}^{-1}(a)} \times .) ) )$. 

Since the laws in \eqref{compatibility1} and \eqref{compatibility2} are equal, we can identify the limits of the right hand sides of both expressions. We get 
\[ \mathcal{L} \left ( f_{a} \left ( \tilde X_{y_0}(c_{a} \times .) \right ) \right ) = \mathcal{L} \left ( f_{y_1} \circ f_{f_{y_1}^{-1}(a)} \left ( \tilde X_{y_0} (c_{y_1} \times c_{f_{y_1}^{-1}(a)} \times .) \right ) \right ). \]
Then, recalling that $a = f_{y_1}(y_2)$ and $\tilde X_{y_0}$ follows the law $P_{y_0}$, we obtain the asserted result. 


\end{proof}

We can now prove Proposition \ref{existsgoodcomponents} that provides the existence of good invariance components. 

\begin{proof} of Proposition \ref{existsgoodcomponents}

For $k \geq 1$, $X$ is a process satisfying Definition \ref{self-simoursens} with $\mathcal{C}^k$ invariance components. Then, let $((f_y, c_y), \ y \in \mathcal{M})$ be a family of $\mathcal{C}^k$ invariance components relatively to the reference point $y_0$ (a family of $\mathcal{C}^k$ invariance components exists by assumption, and the reference point can be set at $y_0$ thanks to the property of change of reference point mentioned in the Introduction, the change of reference point indeed preserves the fact that the family of invariance components is $\mathcal{C}^k$). Replacing, if necessary, $((f_y, c_y), \ y \in \mathcal{M})$ by $((f_y \circ f_{y_0}^{-1}, c_y / c_{y_0}), \ y \in \mathcal{M})$, we can assume that the family of invariance components satisfy \eqref{refpoint}. Let $G$ be the subgroup of $Hom(\mathcal{M}) \times \mathbb{R}_+^*$ that contains the couples $(\Psi, \lambda) \in Hom(\mathcal{M}) \times \mathbb{R}_+^*$ for which \eqref{invarsurplace} is satisfied, and let $H$ be the subgroup of $\mathbb{R}_+^*$ defined by $H := \{ \lambda \in \mathbb{R}_+^*, \ \exists \Psi \in Hom(\mathcal{M}) \ \text{s.t.} \ (\Psi, \lambda) \in G \}$. We have assumed that Assumption \ref{hypsupport1} or Assumption \ref{hypsupport2} is satisfied for $E = \mathcal{M}$. Thanks to Lemma \ref{equivassumpt12} we have that, in any case, Assumption \ref{hypsupport1} is satisfied, so Proposition \ref{compatibility} applies and yields that 
\[ \forall y_1, y_2 \in E, \ U(y_1, y_2) := \left ( f^{-1}_{f_{y_1}(y_2)} \circ f_{y_1} \circ f_{y_2}, c_{y_1} \times c_{y_2} / c_{f_{y_1}(y_2)} \right ) \in G. \]
In particular the range of the function $U_2 : (y_1, y_2) \longmapsto c_{y_1} \times c_{y_2} / c_{f_{y_1}(y_2)}$ is included in $H$. Recall from Definition \ref{self-simoursens} and Remark \ref{continuity} that $y \longmapsto c_y$, $(y,x) \longmapsto f_y(x)$, $y \longmapsto f_y(.)$, and $y \longmapsto f_y^{-1}(.)$ are continuous, and since the composition is continuous on $\mathcal{C}^0(\mathcal{M}, \mathcal{M})$, we get in particular that $U$ is continuous from $\mathcal{M} \times \mathcal{M}$ to $\mathcal{C}^0(\mathcal{M}, \mathcal{M}) \times \mathbb{R}_+^*$. We now distinguish two cases : 

\textbf{Case 1:} $U_2$ is constant. In that case, evaluating $U_2$ at a point $(x, y_0)$ (where $x$ is arbitrary) and using \eqref{refpoint}, we obtain that $U_2 \equiv 1$. This yields that \eqref{relmorphisme} is satisfied so $((f_y, c_y), \ y \in \mathcal{M})$ are good invariance components. 

\textbf{Case 2:} $U_2$ is non-constant. In that case, since $\mathcal{M} \times \mathcal{M}$ is connected and $U_2$ is continuous, we have that $U_2 (\mathcal{M} \times \mathcal{M}) \subset H$ is a subinterval of $\mathbb{R}_+^*$, non-reduced to a single point. Then, $H$ is a subgroup of $\mathbb{R}_+^*$ that contains an interval non reduced to a single point so $H = \mathbb{R}_+^*$. We now construct a continuous function $L : \mathbb{R}_+^* \longrightarrow \mathcal{C}^0(\mathcal{M}, \mathcal{M})$ such that $L[1]=id_{\mathcal{M}}$, $(L[\lambda], \lambda) \in G$ for all $\lambda \in \mathbb{R}_+^*$, and $(x, \lambda) \longmapsto L[\lambda](x)$ is of class $\mathcal{C}^k$. 

Since $((f_y, c_y), \ y \in \mathcal{M})$ are $\mathcal{C}^k$ invariance components, $U_2 : \mathcal{M} \times \mathcal{M} \longrightarrow \mathbb{R}_+^*$ is $\mathcal{C}^k$. Since it is non-constant, there exists $\lambda_0 \in \mathbb{R}_+^*, \epsilon > 0$ and a $\mathcal{C}^k$ function $l_1 : ]\lambda_0 e^{- \epsilon}, \lambda_0 e^{\epsilon}[ \longrightarrow \mathcal{M} \times \mathcal{M}$ such that $\forall \lambda \in ]\lambda_0 e^{- \epsilon}, \lambda_0 e^{\epsilon}[, U_2(l_1[\lambda]) = \lambda$ (just apply the \textit{Local Inversion Theorem} for a local coordinate of the $\mathcal{C}^k$ manifold $\mathcal{M} \times \mathcal{M}$ along which $U_2$ is nonconstant, the other coordinates being fixed). Let us define the function $l_2 : ]\lambda_0 e^{- \epsilon}, \lambda_0 e^{\epsilon}[ \longrightarrow \mathcal{C}^0(\mathcal{M}, \mathcal{M})$ by $l_2 := \lambda \longmapsto U_1 (l_1[\lambda])$, where $U_1(y_1, y_2) = f^{-1}_{f_{y_1}(y_2)} \circ f_{y_1} \circ f_{y_2}$. Then $l_2$ is continuous, moreover $(x, \lambda) \longmapsto l_2[\lambda](x)$ is of class $\mathcal{C}^k$, and clearly for any $\lambda \in ]\lambda_0 e^{- \epsilon}, \lambda_0 e^{\epsilon}[$ we have $U(l_1[\lambda]) = (U_1(l_1[\lambda]), U_2(l_1[\lambda])) = (l_2[\lambda], \lambda)$, so in particular $(l_2[\lambda], \lambda) \in G$. 

Then let us define $l_3 : ]e^{- \epsilon}, e^{\epsilon}[ \longrightarrow \mathcal{C}^0(\mathcal{M}, \mathcal{M})$ by $l_3[\lambda] := l_2(\lambda_0 \lambda) \circ (l_2(\lambda_0))^{-1}$. It is not difficult to see that $l_3$ is continuous, that $(x, \lambda) \longmapsto l_3[\lambda](x)$ is of class $\mathcal{C}^k$, that $l_3[1] = id_{\mathcal{M}}$, and that $(l_3[\lambda], \lambda) \in G$ for any $\lambda \in ]e^{- \epsilon}, e^{\epsilon}[$. Let us now define $L_{\epsilon} : \mathbb{R}_+^* \longrightarrow \mathcal{C}^0(E, E)$ by 
\begin{align*}
L_{\epsilon}[\lambda] & := l_3 \big [ \exp \big ( \log(\lambda) - \epsilon \lfloor 2\log(\lambda)/\epsilon \rfloor /2 \big ) \big ] \circ \big ( l_3[e^{\epsilon/2}] \big )^{\circ \lfloor 2\log(\lambda)/\epsilon \rfloor} \ \ \ \text{if} \ \lambda \geq 1, \\
L_{\epsilon}[\lambda] & := l_3 \big [ \exp \big ( \log(\lambda) + \epsilon \lfloor -2\log(\lambda)/\epsilon \rfloor /2 \big ) \big ] \circ \big ( l_3[e^{-\epsilon/2}] \big )^{\circ \lfloor -2\log(\lambda)/\epsilon \rfloor} \ \ \ \text{if} \ \lambda \leq 1. 
\end{align*}
In the above, $( l_3[e^{\epsilon/2}] )^{\circ n}$ (respectively $( l_3[e^{-\epsilon/2}] )^{\circ n}$) means that the function $l_3[e^{\epsilon/2}] \in \mathcal{C}^0(\mathcal{M}, \mathcal{M})$ (respectively $l_3[e^{-\epsilon/2}] \in \mathcal{C}^0(\mathcal{M}, \mathcal{M})$) is composed $n$ times by itself. Then, clearly, $L_{\epsilon}$ is continuous from $\mathbb{R}_+^*$ to $\mathcal{C}^0(\mathcal{M}, \mathcal{M})$, $L_{\epsilon}[1]=id_{\mathcal{M}}$, and $(x, \lambda) \longmapsto L_{\epsilon}[\lambda](x)$ is of class $\mathcal{C}^k$ on $\mathcal{M} \times ]e^{-\epsilon/2}, e^{\epsilon/2}[$ (since it coincides with $l_3$ on this set) and on each set of the form $\mathcal{M} \times ]e^{n\epsilon/2}, e^{(n+1)\epsilon/2}[$ for some $n \in \mathbb{Z}$. Moreover, since $(l_3[r], r) \in G$ for any $r \in ]e^{- \epsilon}, e^{\epsilon}[$ and since $G$ is a group, we see that for $\lambda \geq 1$, 
\[ (L_{\epsilon}[\lambda], \exp (\log(\lambda) - \epsilon \lfloor 2\log(\lambda)/\epsilon \rfloor /2) \times (e^{\epsilon/2})^{\lfloor 2\log(\lambda)/\epsilon \rfloor}) \in G, \] 
that is, $(L_{\epsilon}[\lambda], \lambda) \in G$. Similarly, for $\lambda < 1$, 
\[ (L_{\epsilon}[\lambda], \exp (\log(\lambda) + \epsilon \lfloor -2\log(\lambda)/\epsilon \rfloor /2) \times (e^{-\epsilon/2})^{\lfloor -2\log(\lambda)/\epsilon \rfloor}) \in G, \]
that is, $(L_{\epsilon}[\lambda], \lambda) \in G$. 

Let us now justify that there is a unique function $L : \mathbb{R}_+^* \longrightarrow \mathcal{C}^0(\mathcal{M}, \mathcal{M})$ that satisfies 
\begin{enumerate} [label=\Alph*)]
\item $L$ is continuous from $\mathbb{R}_+^*$ to $\mathcal{C}^0(\mathcal{M}, \mathcal{M})$, 
\item $L[1]=id_{\mathcal{M}}$, 
\item $\forall \lambda \in \mathbb{R}_+^*, (L[\lambda], \lambda) \in G$. 
\end{enumerate}
The function $L_{\epsilon}$ satisfies the required A), B) and C) so we only need to prove uniqueness. 

First, let us prove that a function $L$ that satisfies A), B) and C) also satisfies $L[\lambda^{-1}]=L[\lambda]^{-1}$ for all $\lambda \in \mathbb{R}_+^*$. Indeed, let $L$ satisfy A), B) and C). Using C) and the fact that $G$ is a group we get 
\[ \forall \lambda \in \mathbb{R}_+^*, \ (L[\lambda] \circ L[\lambda^{-1}], 1) \in G \ \text{and} \ (L[\lambda^{-1}] \circ L[\lambda], 1) \in G. \]
Therefore, we have that $L[\lambda] \circ L[\lambda^{-1}]$ and $L[\lambda^{-1}] \circ L[\lambda]$ are in $Sym(X_{y_0})$ for any $\lambda \in \mathbb{R}_+^*$. Moreover, $(\lambda \longmapsto L[\lambda] \circ L[\lambda^{-1}])$ and $(\lambda \longmapsto L[\lambda^{-1}] \circ L[\lambda])$ are continuous from $\mathbb{R}_+^*$ to $\mathcal{C}^0(\mathcal{M}, \mathcal{M})$ because of A). Since, by assumption, $Sym(X_{y_0})$ is discrete, we deduce that the functions $(\lambda \longmapsto L[\lambda] \circ L[\lambda^{-1}])$ and $(\lambda \longmapsto L[\lambda^{-1}] \circ L[\lambda])$ are constant. Evaluating at $\lambda = 1$ and using B), we deduce that $(\lambda \longmapsto L[\lambda] \circ L[\lambda^{-1}]) = (\lambda \longmapsto L[\lambda^{-1}] \circ L[\lambda]) \equiv id_{\mathcal{M}}$. This proves the claim that $L[\lambda^{-1}]=L[\lambda]^{-1}$. In particular, $(\lambda \longmapsto L[\lambda]^{-1})$ is continuous from $\mathbb{R}_+^*$ to $\mathcal{C}^0(\mathcal{M}, \mathcal{M})$. Note that the above argument can be adapted to show that $(\lambda \longmapsto L[\lambda])$ is a group homomorphism from $(\mathbb{R}_+^*, \times)$ to $(Hom(\mathcal{M}), \circ)$. 

We now prove uniqueness of the function satisfying A), B) and C). Let $L$ and $\tilde L$ be two functions from $\mathbb{R}_+^*$ to $\mathcal{C}^0(\mathcal{M}, \mathcal{M})$ that both satisfy A), B) and C). 
Using C) and the fact that $G$ is a group we get $\forall \lambda \in \mathbb{R}_+^*, (\tilde L[\lambda] \circ L[\lambda]^{-1}, 1) \in G$, so $\tilde L[\lambda] \circ L[\lambda]^{-1} \in Sym(X_{y_0})$ for any $\lambda \in \mathbb{R}_+^*$. We have seen that, since $L$ satisfies A), B) and C), $(\lambda \longmapsto L[\lambda]^{-1})$ is continuous. Combining with the fact that $\tilde L$ satisfies A) we get that $\lambda \longmapsto \tilde L[\lambda] \circ L[\lambda]^{-1}$ is continuous 
from $\mathbb{R}_+^*$ to $\mathcal{C}^0(\mathcal{M}, \mathcal{M})$. Since, by assumption, $Sym(X_{y_0})$ is discrete, we deduce that $(\lambda \longmapsto \tilde L[\lambda] \circ L[\lambda]^{-1})$ is constant. Evaluating at $\lambda = 1$ and using B), we deduce that $(\lambda \longmapsto \tilde L[\lambda] \circ L[\lambda]^{-1}) \equiv id_{\mathcal{M}}$. Therefore $L = \tilde L$ so the uniqueness follows. 

Now, let $L$ be the unique continuous function from $\mathbb{R}_+^*$ to $\mathcal{C}^0(\mathcal{M}, \mathcal{M})$ that satisfies $L[1] = id_{\mathcal{M}}$, and $(L[\lambda], \lambda) \in G$ for all $\lambda \in \mathbb{R}_+^*$. We have seen that $L = L_{\epsilon}$ so in particular $(x, \lambda) \longmapsto L[\lambda](x)$ is of class $\mathcal{C}^k$ on $\mathcal{M} \times ]e^{-\epsilon/2}, e^{\epsilon/2}[$ and on each set of the form $\mathcal{M} \times ]e^{n\epsilon/2}, e^{(n+1)\epsilon/2}[$ for some $n \in \mathbb{Z}$. To prove that $(x, \lambda) \longmapsto L[\lambda](x)$ is of class $\mathcal{C}^k$ on $\mathcal{M} \times \mathbb{R}_+^*$, we thus only need to justify that for each $n \in \mathbb{Z} \setminus \{0\}$, there is an open interval $\mathcal{I}_n \ni e^{n\epsilon/2}$ such that $(x, \lambda) \longmapsto L[\lambda](x)$ is of class $\mathcal{C}^k$ on $\mathcal{M} \times \mathcal{I}_n$. Let us chose $\epsilon' \in ]0, \epsilon[$ such that $\epsilon'/\epsilon$ is irrational. We can build $L_{\epsilon'}$ similarly as we built $L_{\epsilon}$. In particular $L_{\epsilon'}$ is continuous from $\mathbb{R}_+^*$ to $\mathcal{C}^0(\mathcal{M}, \mathcal{M})$, $L_{\epsilon'}[1]=id_{\mathcal{M}}$, $\forall \lambda \in \mathbb{R}_+^*,(L_{\epsilon'}[\lambda], \lambda) \in G$, and $(x, \lambda) \longmapsto L_{\epsilon'}[\lambda](x)$ is of class $\mathcal{C}^k$ on each set of the form $\mathcal{M} \times ]e^{m\epsilon'/2}, e^{(m+1)\epsilon'/2}[$ for some $m \in \mathbb{Z}$. Since $\epsilon'/\epsilon$ is irrational, we have that for any $n \in \mathbb{Z} \setminus \{0\}$, there is $m_n \in \mathbb{Z}$ such that $e^{n\epsilon/2} \in ]e^{m_n\epsilon'/2}, e^{(m_n+1)\epsilon'/2}[ =: \mathcal{I}_n$. Clearly $(x, \lambda) \longmapsto L_{\epsilon'}[\lambda](x)$ is of class $\mathcal{C}^k$ on $\mathcal{M} \times \mathcal{I}_n$, but by uniqueness we have $L = L_{\epsilon'}$, so $(x, \lambda) \longmapsto L[\lambda](x)$ is of class $\mathcal{C}^k$ on $\mathcal{M} \times \mathcal{I}_n$. We conclude that $(x, \lambda) \longmapsto L[\lambda](x)$ is of class $\mathcal{C}^k$ on $\mathcal{M} \times \mathbb{R}_+^*$. 


Then, for all $y \in \mathcal{M}$, let us define $\tilde f_y := f_y \circ L[1/c_y]$. We now justify that the family $((\tilde f_y, 1), \ y \in \mathcal{M})$ defines a family of $\mathcal{C}^k$ invariance components for $X$. Let us fix $y \in \mathcal{M}$, 
we have 
\[ \tilde f_y(X_{y_0}) = f_y \left ( L[1/c_y] (X_{y_0}) \right ) \overset{\mathcal{L}}{=} f_y (X_{y_0}(c_y \times .)) \overset{\mathcal{L}}{=} X_{y}, \]
where we have used that $(L[1/c_y], 1/c_y) \in G$ for the first equality in law, and that $((f_z, c_z), \ z \in \mathcal{M})$ is a family of invariance components for $X$ for the second equality in law. Moreover, since $\tilde f_y(x) = f_y (L[1/c_y](x))$ where $(y \longmapsto c_y)$, $((y,x) \longmapsto f_y(x))$ and $(x, \lambda) \longmapsto L[\lambda](x)$ are of class $\mathcal{C}^k$, we see that $((y,x) \longmapsto \tilde f_y(x))$ is of class $\mathcal{C}^k$. Similarly, $\tilde f_y^{-1}(x) = L[1/c_y]^{-1} (f_y^{-1} (x)) = L[c_y] (f_y^{-1} (x))$ where $(y \longmapsto c_y)$, $((y,x) \longmapsto f_y^{-1}(x))$ and $(x, \lambda) \longmapsto L[\lambda](x)$ are of class $\mathcal{C}^k$, so $((y,x) \longmapsto \tilde f_y^{-1}(x))$ is also of class $\mathcal{C}^k$. We conclude that $((\tilde f_y, 1), \ y \in \mathcal{M})$ is a family of $\mathcal{C}^k$ invariance components for $X$. 
Note that $\tilde f_{y_0} = f_{y_0} \circ L[1/c_{y_0}] = id_{\mathcal{M}} \circ L[1] = id_{\mathcal{M}} \circ id_{\mathcal{M}} = id_{\mathcal{M}}$ (we have used that $((f_y, c_y), \ y \in \mathcal{M})$ satisfies \eqref{refpoint}, as mentioned in the beginning of the proof). Moreover, the time-dilatation constants in $((\tilde f_y, 1), \ y \in \mathcal{M})$ are all equal to $1$. Therefore $((\tilde f_y, 1), \ y \in \mathcal{M})$ satisfies Definition \ref{goodcomponents}, so $((\tilde f_y, 1), \ y \in \mathcal{M})$ is even a family of $\mathcal{C}^k$ good invariance components for $X$.

\end{proof}

\begin{remarque}
The above proof actually shows a little more than the statement of Proposition \ref{existsgoodcomponents} and can possibly be used to produce good invariance components in practice. It shows that, for $X$ that satisfies the assumptions of the proposition, if we pick a family $((f_y, c_y), \ y \in \mathcal{M})$ of $\mathcal{C}^k$ invariance components, then we have the following dichotomy: either this family is a family of \textbf{good} invariance components, either this family can be canonically modified (via composition by some functions $L[\lambda](.)$ that are uniquely determined) to produce a family of $\mathcal{C}^k$ \textbf{good} invariance components $((\tilde f_y, 1), \ y \in \mathcal{M})$ for which the changes of time are trivial. Recall that, as mentioned in Subsection \ref{examples}, the case where changes of times are trivial is particularly interesting, because our results show that in that case $X_{y_0}$ is a L\'evy process on some group (in the context of Theorem \ref{casgeneralexpl}), or the image by some function of a L\'evy process on $\mathbb{R}/\mathbb{R}^2$ or of a process built from the exponential functional of a bivariate L\'evy process (in the context of Theorems \ref{casrexpl} and \ref{casr2expl}), with no time-change. In particular, if $X$ satisfies the assumptions of the proposition and is given along with a family of $\mathcal{C}^k$ invariance components, then, if the family is not good we can already conclude that $X_{y_0}$ is as just described. 
\end{remarque}


\section{Relation with groups} \label{relwithgroups}

First, let us prove Proposition \ref{levychangentpssensfacil} which says that it is possible to construct processes satisfying Definition \ref{self-simoursens} when the state space is already equipped with a group structure. 

\begin{proof} of Proposition \ref{levychangentpssensfacil}

Let us fix $y \in E$ and let $\varphi_y$ be defined by \eqref{timechangingrecip}. Since $(s \mapsto L(s))$ is c\`ad-l\`ag and $h$ is continuous and positive, the function $\varphi_y$ is well-defined, continuous, increasing and defines a bijection from $[0, +\infty[$ onto $[0, \varphi_y(+\infty)[$. As a consequence $\varphi_y^{-1}$ is well-defined on $[0, \varphi_y(+\infty)[$ and continuous (we set $\varphi_y^{-1}(t) := +\infty$ for $t \geq \varphi_y(+\infty)$). This implies in particular that $X_y$, defined by \eqref{timechangedrecip}, is well-defined and c\`ad-l\`ag on $[0, \varphi_y(+\infty)[$. 

Let us now justify that $\zeta(X_{y}) = \varphi_y(+\infty)$ a.s. Since $\varphi_y^{-1}$ is well-defined and continuous on $[0, \varphi_y(+\infty)[$, we have that for any $t \in [0, \varphi_y(+\infty)[$, the interval $[0, \varphi_y^{-1}(t)]$ is compact. Since $L$ is c\`ad-l\`ag, and since $(z \longmapsto y \star z)$ is an homeomorphism, the set $y \star L([0, \varphi_y^{-1}(t)])$ is relatively compact, which show that the trajectory of $X_{y}$ on $[0,t]$ is contained inside a compact set. Therefore, $t < \zeta(X_{y})$, and since this is true a.s. for any $t \in [0, \varphi_y(+\infty)[$ we get $\varphi_y(+\infty) \leq \zeta(X_{y})$ a.s. This proves that we have a.s. $\varphi_y(+\infty) = \zeta(X_{y})$ whenever $\varphi_y(+\infty) = +\infty$. Let $(K_n)_{n \geq 1}$ be an increasing sequence of compact sets $K_1 \subset K_2 \subset ...$ such that any compact subset of $E$ is included in $K_n$ for some $n \geq 1$. 
If $\varphi_y(+\infty) < +\infty$, then $h(y \star L(.))$ takes arbitrary large values on $[0, +\infty[$. Therefore, for any $n \geq 1$ there is a.s. $u_n \in [0, +\infty[$ such that $h(y \star L(u_n)) > \sup_{z \in K_n} h(z)$. We thus have $\tau(y \star L(.), K_n^c) \leq u_n < +\infty$ a.s. Since $\varphi_y^{-1}$ is increasing we deduce that $\tau(X_y, K_n^c) = \tau(y \star L(\varphi_y^{-1}(.)), K_n^c) = \varphi_y(\tau(y \star L(.), K_n^c)) \leq \varphi_y(u_n) < \varphi_y(+\infty)$. We deduce that $\zeta(X_y) = \lim_{n \rightarrow +\infty} \tau(X_y, K_n^c) \leq \varphi_y(+\infty)$ a.s. In conclusion $\varphi_y(+\infty) = \zeta(X_{y})$ a.s. which proves in particular that $\Delta$ can only be reached continuously by $X_y$, so $X_y$ is c\`ad-l\`ag on $[0, +\infty[$. 

We denote by $P_{y}$ the law of the process $X_y$. 
Let respectively $(\mathcal{G}_t, \ t \geq 0)$ and $(\mathcal{F}_t, \ t \geq 0)$ be the right continuous filtrations generated by the processes $L$ and $X_y$. Note that for any $t \geq 0$, $\varphi_y^{-1} (t)$ is a (possibly infinite) stopping time for $L$ and that $\mathcal{G}_{\varphi_y^{-1} (t)} = \mathcal{F}_{t}$. 

We now justify that $X_y$ satisfies the Markov property at every instant $t \geq 0$. 
Let us fix an instant $t \geq 0$ and work on $\{ \varphi_y^{-1} (t) < +\infty \} = \{ X_y(t) \neq \Delta \}$ (since $\Delta$ is an absorbing state, the case where $X_y(t) = \Delta$ is trivial). Since $\varphi_y^{-1} (t)$ is a stopping time for the L\'evy process $L$ on $(E, \star)$, the process $\tilde L(.) := (L(\varphi_y^{-1} (t)))^{-1} \star L(\varphi_y^{-1} (t) + .)$ is independent from $\mathcal{G}_{\varphi_y^{-1} (t)} = \mathcal{F}_{t}$ and has the same law as $L$. Now let $\tilde \varphi_{e_0}$ and $\tilde X_{e_0}$ be constructed from $\tilde L$ just as $\varphi_y$ and $X_y$ are constructed from $L$: for all $s \geq 0$, $\tilde \varphi_{e_0} (s) := \int_0^s \frac1{h(\tilde L(u))} du$ and for all $s \in [0, \tilde \varphi_{e_0} (+\infty)[, \tilde X_{e_0}(s) := \tilde L (\tilde \varphi_{e_0}^{-1}(s))$ (and $\tilde X_{e_0}(s) := \Delta$ for $s \geq \tilde \varphi_{e_0} (+\infty)$). Then, since it is a function of $\tilde L$, $\tilde X_{e_0}$ is independent from $\mathcal{F}_t$ and, since $\tilde L$ has the same law as $L$, $\tilde X_{e_0}$ has law $P_{e_0}$. We need to link $\varphi_y$ and $\tilde \varphi_{e_0}$. For any $u \geq 0$ we have from the definition of $\tilde L$ and the homomorphism property for $h$ that $h(\tilde L(u)) = h(L(\varphi_y^{-1} (t) + u))/h(L(\varphi_y^{-1} (t)))$. Therefore, for any $s \in [0, \tilde \varphi_{e_0} (+\infty)[$ we have 
\begin{align*}
s & = \int_{0}^{\tilde \varphi_{e_0}^{-1} (s)} \frac1{h(\tilde L(u))} du = \int_{0}^{\tilde \varphi_{e_0}^{-1} (s)} \frac{h(L(\varphi_y^{-1} (t)))}{h(L(\varphi_y^{-1} (t) + u))} du \\
& = \frac{h(y \star L(\varphi_y^{-1} (t)))}{h(y)} \int_{0}^{\tilde \varphi_{e_0}^{-1} (s)} \frac{1}{h(L(\varphi_y^{-1} (t) + u))} du \\
& = \frac{h(X_y(t))}{h(y)} \int_{\varphi_y^{-1} (t)}^{\varphi_y^{-1} (t) + \tilde \varphi_{e_0}^{-1} (s)} \frac{1}{h(L(u))} du \\
& = h(X_y(t)) \times \left [ \varphi_y \left ( \varphi_y^{-1} (t) + \tilde \varphi_{e_0}^{-1} (s) \right ) - t \right ]. 
\end{align*}
We thus get $\varphi_y ( \varphi_y^{-1} (t) + \tilde \varphi_{e_0}^{-1} (s) ) = t + s/h(X_y(t))$ from which we deduce that $\varphi_y^{-1} (t + s/h(X_y(t))) = \varphi_y^{-1} (t) + \tilde \varphi_{e_0}^{-1} (s)$. As a consequence, for any $s \in [0, \tilde \varphi_{e_0} (+\infty)/h(X_y(t))[$ we have that $\varphi_y^{-1} (t + s)$ is finite and equals $\varphi_y^{-1} (t + h(X_y(t))s/h(X_y(t))) = \varphi_y^{-1} (t) + \tilde \varphi_{e_0}^{-1} (h(X_y(t)) \times s)$. Then, 
\begin{align*}
X_y(t+s) & = y \star L(\varphi_y^{-1} (t + s)) = y \star L \left ( \varphi_y^{-1} (t) + \tilde \varphi_{e_0}^{-1} (h(X_y(t)) \times s) \right ) \\
& = y \star L \left ( \varphi_y^{-1} (t) \right ) \star \tilde L \left ( \tilde \varphi_{e_0}^{-1} (h(X_y(t)) \times s) \right ) \\
& = X_y(t) \star \tilde X_{e_0}(h(X_y(t)) \times s). 
\end{align*}
We have obtained that, on $[0, \tilde \varphi_{e_0} (+\infty)[$, $X_y(t+.) = X_y(t) \star \tilde X_{e_0}(h(X_y(t)) \times .)$, with $\tilde X_{e_0}$ independent of $\mathcal{F}_t$ and having law $P_{e_0}$. We can also see that $\varphi_y(+\infty) = t + \tilde \varphi_{e_0} (+\infty)/h(X_y(t))$ so $X_y(t+s) = \Delta$ if and only if $\tilde X_{e_0}(h(X_y(t)) \times s) = \Delta$. Combing all this we deduce the Markov property for $X_y$ at time $t$. The same procedure can be done in the case where $t$ is a stopping time so $X_y$ is strongly Markovian. 

We now justify the self-similarity of $X$. 
Let $\varphi_{e_0}$ and $X_{e_0}$ be constructed from $L$ as in \eqref{timechangingrecip} and \eqref{timechangedrecip}. Then, $X_{e_0}$ has law $P_{e_0}$. Let $y \in E$ be arbitrary, and 
define $\varphi_y$ and $X_y$ from $L$ as in \eqref{timechangingrecip} and \eqref{timechangedrecip}. Then, $X_y$ has law $P_{y}$. We have 
\[ \forall t \geq 0, \ \varphi_y(t) = \frac1{h(y)} \int_0^t \frac1{h(L(s))} ds = \frac1{h(y)} \varphi_{e_0}(t), \]
so that $\varphi_y^{-1}(.) = \varphi_{e_0}^{-1}(h(y) \times .)$ and $\varphi_y(+\infty) = \varphi_{e_0}(+\infty)/h(y)$. We then have for all $0 \leq t < \varphi_y(+\infty) = \varphi_{e_0}(+\infty)/h(y)$, 
\[ X_y(t) = y \star L \left (\varphi_y^{-1}(t) \right ) = y \star L \left (\varphi_{e_0}^{-1}(h(y) \times t) \right ) = y \star X_{e_0} \left ( h(y) \times t \right ). \]
As a consequence we have $X_y(.) = y \star X_{e_0} \left ( h(y) \times . \right )$ on $[0, \varphi_y(+\infty)[ = [0, \varphi_{e_0}(+\infty)/h(y)[$ and these two processes reach the state $\Delta$ at the same time $\varphi_y(+\infty) = \varphi_{e_0}(+\infty)/h(y)$. 
We deduce that 
\[ \left ( X_y (t), \ t \geq 0 \right ) = \left ( y \star X_{e_0}(h(y) t), \ t \geq 0 \right ), \]
where $X_y$ has law $P_y$ while $X_{e_0}$ has law $P_{e_0}$ (note that we have used the convention $y \star \Delta = \Delta$). Since the above is true for any $y \in E$ we deduce that $X$ satisfies Definition \ref{self-simoursens} with invariance components $((f_y, c_y), \ y \in E)$ given by $f_y(.) := y \star .$ and $c_y := h(y)$, for the reference point $e_0$. Since $e_0$ is the neutral element of $(E,\star)$, and since $h$ is an homomorphism from $(E,\star)$ to $(\mathbb{R}_+^*, \times)$, we see that \eqref{refpoint} and \eqref{relmorphisme} are satisfied so that the family $((f_y, c_y), \ y \in E)$ actually defines good invariance components. 

\end{proof}

We are now interested in proving a key point for our purpose: for $X$ a process satisfying Definition \ref{self-simoursens}, given along with a family of good invariance components, we want to make appear a group structure on the sate space such that $X$ is a time-changed L\'evy process on this group. In the following proposition we make appear the group structure. 

\begin{prop} \label{groupappears}

Let $X$ be a process satisfying Definition \ref{self-simoursens} on a connected locally compact separable metric space $E$ and let $((f_y, c_y), \ y \in E)$ be a family of good invariance components for $X$, relatively to some reference point $y_0$. We assume that either Assumption \ref{hypsupport1} or Assumption \ref{hypsupport2} is satisfied and that $Sym(X_{y_0})$ is discrete. Let us define an interne composition law $\star$ on $E$ by $y \star x := f_y(x)$. Then $(E, \star)$ is a topological group with neutral element $y_0$ and $(y \longmapsto c_y)$ is a continuous group homomorphism from $(E,\star)$ to $(\mathbb{R}_+^*, \times)$. 

If moreover, for $k \geq 1$, $E$ is a $\mathcal{C}^k$-differentiable manifold and $((f_y, c_y), \ y \in E)$ are $\mathcal{C}^k$ good invariance components, then $(E,\star)$ is even a $\mathcal{C}^k$-Lie group and $(y \longmapsto c_y)$ is a $\mathcal{C}^k$-Lie group homomorphism. 

\end{prop}

\begin{proof} 

\textit{Associativity:} This is the key point. According to Proposition \ref{compatibility}, we have that, 
\[ \forall y_1, y_2 \in E, \ f_{f_{y_1}(y_2)} \left ( X_{y_0} \right ) \overset{\mathcal{L}}{=} f_{y_1} \circ f_{y_2} \left ( X_{y_0} \left ( \frac{c_{y_1} \times c_{y_2}}{c_{f_{y_1}(y_2)}} \times . \right ) \right ). \]
Since $((f_y, c_y) \ y \in E)$ is a family of good invariance components it satisfies \eqref{relmorphisme}, so $c_{y_1} \times c_{y_2}/c_{f_{y_1}(y_2)} = 1$. We thus get $f_{f_{y_1}(y_2)} \left ( X_{y_0}(.) \right ) \overset{\mathcal{L}}{=} f_{y_1} \circ f_{y_2} \left ( X_{y_0}(.) \right )$, so that 
\[ \forall y_1, y_2 \in E, \ U_1(y_1, y_2) := f^{-1}_{f_{y_1}(y_2)} \circ f_{y_1} \circ f_{y_2} \in Sym(X_{y_0}). \]
Recall from Definition \ref{self-simoursens} and Remark \ref{continuity} that $(y,x) \longmapsto f_y(x)$, $y \longmapsto f_y(.)$, and $y \longmapsto f_y^{-1}(.)$ are continuous, and since the composition is continuous on $\mathcal{C}^0(E, E)$, we get in particular that $U_1 : E \times E \longrightarrow \mathcal{C}^0(E, E)$ is continuous. 
Moreover, $E \times E$ is connected. By assumption, $Sym(X_{y_0})$ is discrete so $U_1$ is constant on $E \times E$. Moreover, since $((f_y, c_y) \ y \in E)$ is a family of good invariance components it satisfies \eqref{refpoint} from which we deduce $U_1(y_0, y_0) = id_E$. Therefore $U_1 \equiv id_E$ so we get 
\[ \forall y_1, y_2 \in E, \ f_{f_{y_1}(y_2)} = f_{y_1} \circ f_{y_2}. \]
Evaluating the above functions at a point $y_3 \in E$ we obtain exactly $(y_1 \star y_2) \star y_3 = y_1 \star (y_2 \star y_3)$, and the associativity follows. 

\textit{Neutral element:} Using \eqref{refpoint} we see that for any $y \in E$, $y_0 \star y = f_{y_0}(y) = id_E(y) = y$. Also, evaluating \eqref{self-simoursenseq} at $t=0$ and using that $X_{y_0}(0)$ and $X_{y}(0)$ are almost surely equal to respectively $y_0$ and $y$, we see that $f_y(y_0)=y$, that is, $y \star y_0 = y$. This proves that $y_0$ is a neutral element for $\star$. 

\textit{Inverse:} Clearly we only need to justify the existence of a right-inverse for every $y \in E$. Since $f_y$ is bijective, $f_y^{-1}(y_0)$ is well-defined and it clearly satisfies $y \star f_y^{-1}(y_0) = y_0$. 

\textit{Continuity:} Let us denote by $y^{-1}$ the inverse of $y$. The continuity of $((y,x) \longmapsto y \star x)$ and of $(y \longmapsto y^{-1})$ follows from the continuity of $((y,x) \longmapsto f_y(x))$ and of $((y,x) \longmapsto f_y^{-1}(x))$ assumed in Definition \ref{self-simoursens}. 

\textit{Group homomorphism:} The fact that $(y \longmapsto c_y)$ is a group homomorphism from $(E,\star)$ to $(\mathbb{R}_+^*, \times)$ follows from \eqref{relmorphisme}, which is satisfied since, by assumption, $((f_y, c_y) \ y \in E)$ are good invariance components. The continuity of $(y \longmapsto c_y)$ is satisfied by definition of invariance components in Definition \ref{self-simoursens}. 

\textit{Smoothness:} If $E$ is a $\mathcal{C}^k$-differentiable manifold and $((f_y, c_y), \ y \in E)$ are $\mathcal{C}^k$ good invariance components, then it only remains to justify that the applications $((y,x) \longmapsto y \star x$), $(y \longmapsto y^{-1})$, and $(y \longmapsto c_y)$ are $\mathcal{C}^k$, but this follows from Definition \ref{smoothcomponents}. 

%
%
%

\end{proof}


In order to simplify the statement of the forthcoming propositions, let us introduce the following definition: 

\begin{defi} \label{bearinggroup}
Let $X$ be a process satisfying Definition \ref{self-simoursens} on a locally compact separable metric space $E$. Let $((f_y, c_y), \ y \in E)$ be a family of good invariance components for $X$, relatively to some reference point $y_0$. If the interne composition law $\star$, defined on $E$ by $y \star x := f_y(x)$, gives rise to a topological group structure on $E$, then we say that $(E,\star)$ is the \textbf{bearing group} of $X$ associated with the good invariance components $((f_y, c_y), \ y \in E)$. 
\end{defi}



If $(E,\star)$ is a bearing group for $X$, then the self-similarity relation can be re-expressed as 
\[ \forall y \in E, \ \left ( X_y (t), \ t \geq 0 \right ) \overset{\mathcal{L}}{=} \left ( y \star X_{y_0}(c_y t), \ t \geq 0 \right ). \]
In particular, for any starting point $z$, if $\tilde X_{y_0} \sim P_{y_0}$ is independent from $X_z$ and $S$ is a stopping time for $X_z$, then 
\begin{eqnarray}
 \left ( X_z (S+t), \ t \geq 0 \right ) \overset{\mathcal{L}}{=} \left ( X_z(S) \star \tilde X_{y_0}(c_{X_z(S)} t), \ t \geq 0 \right ). \label{relationlevy}
\end{eqnarray}


We can now prove that a process satisfying Definition \ref{self-simoursens} can be identified with a time-changed L\'evy process on the group constructed in Proposition \ref{groupappears}. 

\begin{prop} \label{levychangentps}

Let $X$ be a process satisfying Definition \ref{self-simoursens} on a connected locally compact separable metric space $E$ and let $((f_y, c_y), \ y \in E)$ be a family of good invariance components for $X$, relatively to some reference point $y_0$. We assume that either Assumption \ref{hypsupport1} or Assumption \ref{hypsupport2} is satisfied for $E$ and that $Sym(X_{y_0})$ is discrete. Let $(E,\star)$ be the bearing group associated with $((f_y, c_y), \ y \in E)$. Then there is a L\'evy process $L$ on $(E,\star)$ such that if we set 
\begin{eqnarray}
\forall t \in [0, +\infty], \ \varphi(t) := \int_0^t \frac1{c_{L(s)}} ds, \label{timechanging}
\end{eqnarray}
then 
\begin{eqnarray}
\zeta(X_{y_0}) = \varphi(+\infty) \label{tpsvie}
\end{eqnarray}
and 
\begin{eqnarray}
\forall \ 0 \leq t < \varphi(+\infty), \ X_{y_0}(t) = L \left (\varphi^{-1}(t) \right ). \label{timechanged}
\end{eqnarray}
Moreover, if we let $\xi(.) := \log(1/c_{L(.)})$, then $\xi$ is a real L\'evy process and we have $\zeta(X_{y_0}) = \int_0^{+\infty} e^{\xi(s)} ds$. 

In the particular case where $c_y = 1, \forall y \in E$, then $X_{y_0} = L$, so $X_{y_0}$ is itself a L\'evy process on $(E,\star)$. 

\end{prop}

\begin{proof}

We define for all $0 \leq t < \zeta(X_{y_0})$, $\phi (t) := \int_0^t c_{X_{y_0}(u)} du$ and $L(t) := X_{y_0} (\phi^{-1} (t))$. 

We first justify that $L$ is well-defined. Since $(s \mapsto X_{y_0}(s))$ is c\`ad-l\`ag and $(y \mapsto c_y)$ is continuous and positive, the function $\phi$ is well-defined, continuous, increasing and defines a bijection from $[0, \zeta(X_{y_0})[$ onto its range, $[0, \phi(\zeta(X_{y_0}))[$, where $\phi(\zeta(X_{y_0}))$ has to be considered as possibly finite (we will prove later that it is actually a.s. infinite). As a consequence $\phi^{-1}$ is well-defined on $[0, \phi(\zeta(X_{y_0}))[$ and continuous. This implies in particular that $L$ is well-defined on $[0, \phi(\zeta(X_{y_0}))[$ and c\`ad-l\`ag. We see from $L(t) := X_{y_0} (\phi^{-1} (t))$ that $\phi(\zeta(X_{y_0})) = \zeta(L)$. If $\phi(\zeta(X_{y_0})) < +\infty$, let us put $\phi^{-1} (t) = +\infty$ and $L(t) := \Delta$ for $t \geq \phi(\zeta(X_{y_0}))$, so that $L$ is well-defined on $[0, +\infty[$ and c\`ad-l\`ag (we see that the state $\Delta$ is a.s. reached continuously). 
Let respectively $(\mathcal{G}_t, \ t \geq 0)$ and $(\mathcal{F}_t, \ t \geq 0)$ be the right continuous filtrations generated by the processes $L$ and $X_{y_0}$. Note that for any $t \geq 0$, $\phi^{-1} (t)$ is a stopping time for $X_{y_0}$ and that for any $t \geq 0, \mathcal{G}_t = \mathcal{F}_{\phi^{-1} (t)}$. 

We now prove that $L$ is a L\'evy process on $(E,\star)$, possibly absorbed at $\Delta$. Let us fix an instant $t \geq 0$ and work on $\{ L(t) \neq \Delta\} = \{ \phi^{-1} (t) < +\infty \}$. According to \eqref{relationlevy}, the process $\tilde X_{y_0}(.) := (X_{y_0}(\phi^{-1} (t)))^{-1} \star X_{y_0}(\phi^{-1} (t) + ./c_{X_{y_0}(\phi^{-1} (t))})$ is independent from $\mathcal{F}_{\phi^{-1} (t)} = \mathcal{G}_t$ and has law $P_{y_0}$. Now let $\tilde \phi$ and $\tilde L$ be constructed from $\tilde X_{y_0}$ just as $\phi$ and $L$ are constructed from $X_{y_0}$: for all $0 \leq s < \zeta(\tilde X_{y_0})$, $\tilde \phi (s) := \int_0^s c_{\tilde X_{y_0}(u)} du$ and $\tilde L(s) := \tilde X_{y_0} (\tilde \phi^{-1} (s))$. Then, since it is a function of $\tilde X_{y_0}$, $\tilde L$ is independent from $\mathcal{G}_t$ and, since $\tilde X_{y_0}$ has law $P_{y_0}$, $\tilde L$ has the same law as $L$. We need to link $\phi$ and $\tilde \phi$. For any $0 \leq u < \zeta(\tilde X_{y_0})$ we have from the definition of $\tilde X_{y_0}$ and the homomorphism property of $(y \longmapsto c_y)$ that $c_{\tilde X_{y_0}(u)} = c_{X_{y_0}(\phi^{-1} (t) + u/c_{X_{y_0}(\phi^{-1} (t))})} /c_{X_{y_0}(\phi^{-1} (t))}$. For any $0 \leq s < \tilde \phi(\zeta(\tilde X_{y_0}))$ we have 
\begin{align*}
s & = \int_{0}^{\tilde \phi^{-1} (s)} c_{\tilde X_{y_0}(u)} du = \int_{0}^{\tilde \phi^{-1} (s)} \frac{c_{X_{y_0}(\phi^{-1} (t) + u/c_{X_{y_0}(\phi^{-1} (t))})}}{c_{X_{y_0}(\phi^{-1} (t))}} du = \int_{0}^{\tilde \phi^{-1} (s)/c_{X_{y_0}(\phi^{-1} (t))}} c_{X_{y_0}(\phi^{-1} (t) + v)} dv \\
& = \int_{\phi^{-1}(t)}^{\phi^{-1}(t)+\tilde \phi^{-1} (s)/c_{X_{y_0}(\phi^{-1} (t))}} c_{X_{y_0}(u)} du = \phi \left ( \phi^{-1}(t)+\tilde \phi^{-1} (s)/c_{X_{y_0}(\phi^{-1} (t))} \right ) - t. 
\end{align*}
We thus get $\phi \left ( \phi^{-1}(t)+ \tilde \phi^{-1} (s)/c_{X_{y_0}(\phi^{-1} (t))} \right ) = t + s$ 
from which we deduce that $\phi^{-1}(t + s)$ is finite and satisfies $c_{X_{y_0}(\phi^{-1} (t))}[\phi^{-1}(t + s) - \phi^{-1}(t)] = \tilde \phi^{-1} (s)$. As a consequence, for any $0 \leq s < \tilde \phi(\zeta(\tilde X_{y_0}))$ we have 
\begin{align*}
L(t+s) & = X_{y_0}(\phi^{-1} (t + s)) = X_{y_0} \left ( \phi^{-1} (t) + c_{X_{y_0}(\phi^{-1} (t))} [\phi^{-1} (t + s) - \phi^{-1} (t)] / c_{X_{y_0}(\phi^{-1} (t))} \right ) \\
& = X_{y_0} \left ( \phi^{-1} (t) \right ) \star \tilde X_{y_0} \left ( c_{X_{y_0}(\phi^{-1} (t))} [\phi^{-1} (t + s) - \phi^{-1} (t)] \right ) = X_{y_0} \left ( \phi^{-1} (t) \right ) \star \tilde X_{y_0} \left ( \tilde \phi^{-1} (s) \right ) \\
& = L(t) \star \tilde L(s). 
\end{align*}
We have obtained that, on $[0, \tilde \phi(\zeta(\tilde X_{y_0}))[$, $(L(t))^{-1} \star L(t+.) = \tilde L(.)$. We can also see that $\phi(\zeta(X_{y_0})) = t + \tilde \phi(\zeta(\tilde X_{y_0}))$ 
so $L(t+s) = \Delta$ if and only if $\tilde L(s) = \Delta$. Therefore we have that, on $[0, +\infty[$, $(L(t))^{-1} \star L(t+.) = \tilde L(.)$, with $\tilde L$ independent of $\mathcal{G}_t$ and having the same law as $L$. Since we have such a relation for any $t \geq 0$, and combining with the fact that $L$ is c\`ad-l\`ag, we deduce that it is a L\'evy process on $E$, possibly absorbed at $\Delta$. 

Let us assume that absorption at $\Delta$ can occur with a positive probability: $\mathbb{P}(\zeta(L) < +\infty) > 0$, and recall that, in that case, $\Delta$ is reached continuously a.s. on $\{ \zeta(L) < +\infty \}$. Note that since $L(t+.)$ reaches $\Delta$ at $s$ if and only if $\tilde L$ reaches $\Delta$ at $s$, the absorption at $\Delta$ occurs independently from the past trajectory (i.e. the fact that the absorption occurs on $[t, t+h[$ is independent from the trajectory of $L$ on $[0,t]$) so in particular, given a compact neighborhood $K$ of $y_0$, we can see that $L$ may remain in $K$ on $[0, \zeta(L)[$ with positive probability. 
This means that, with positive probability, $\Delta$ may be not reached continuously which is a contradiction. Therefore $L$ is a.s. never absorbed so it is a regular L\'evy process on $(E,\star)$ and we have a.s. $\phi(\zeta(X_{y_0})) = \zeta(L) = +\infty$. In particular, $\phi^{-1}(t)$ is defined and finite for all $t \geq 0$ (it is a bijection from $[0, +\infty[$ to $[0, \zeta(X_{y_0})[$) so the expression $L(t) = X_{y_0} (\phi^{-1} (t))$ holds for all $t \geq 0$. 

We now need to prove that $\phi^{-1} = \varphi$, for $\varphi$ defined from $L$ as in \eqref{timechanging}. For any $t \geq 0$ we have 
\begin{eqnarray}
\phi^{-1}(t) = \int_0^{\phi^{-1}(t)} 1 du = \int_0^{t} \frac1{\phi'(\phi^{-1}(v))} dv = \int_0^{t} \frac1{c_{X_{y_0}(\phi^{-1}(v))}} dv = \int_0^{t} \frac1{c_{L(v)}} dv = \varphi(t), \label{egchgttps}
\end{eqnarray}
where we have used the change of variable $v = \phi(u)$, the definition of $\phi$, the definition of $L$ in term of $X_{y_0}$ and $\phi^{-1}$, and the definition of $\varphi$ in \eqref{timechanging}. Taking the limit on both sides in \eqref{egchgttps} when $t$ goes to infinity we get $\zeta(X_{y_0}) = \varphi(+\infty)$, which is \eqref{tpsvie}. Combining \eqref{egchgttps} with the expression $L(t) = X_{y_0} (\phi^{-1} (t))$ we obtain \eqref{timechanged}. 

Let us now justify the representation of $\zeta(X_{y_0})$ in term of an exponential functional of a real L\'evy process. Recall from Proposition \ref{groupappears} that $(y \mapsto c_y)$ is a continuous group homomorphism from $(E,\star)$ to $(\mathbb{R}_+^*, \times)$, so $(y \mapsto \log(1/c_y))$ is a continuous group homomorphism from $(E,\star)$ to $(\mathbb{R}, +)$. Therefore, $\xi(.) := \log(1/c_{L(.)})$ is a real valued L\'evy process and, re-writting \eqref{tpsvie} in term of $\xi$ we get $\zeta(X_{y_0}) = \int_0^{+\infty} e^{\xi(s)} ds$. 
\end{proof}

\begin{remarque} \label{chgttimeintermofx}
We note from the previous proof that the change of time $\varphi^{-1}(t)$ in \eqref{timechanged} has an alternative expression in term of $X_{y_0}$: for all $0 \leq t < \zeta(X_{y_0})$, $\varphi^{-1}(t) = \phi (t) = \int_0^t c_{X_{y_0}(u)} du$. 
\end{remarque}

Putting together Proposition \ref{groupappears} and Proposition \ref{levychangentps} we obtain Theorem \ref{casgeneralexpl} and, combining with Proposition \ref{existsgoodcomponents}, we deduce Theorem \ref{casgeneral}. Let us write all the details for the sake of clarity. 

\begin{proof} of Theorem \ref{casgeneralexpl}

Let $E$ be a connected locally compact separable metric space, and $X$ be a process satisfying Definition \ref{self-simoursens} on $E$. $((f_y, c_y), \ y \in E)$ is a family of good invariance components associated with $X$, relatively to some reference point $y_0 \in E$. We assume that either Assumption \ref{hypsupport1} or Assumption \ref{hypsupport2} is satisfied for $E$ and that $Sym(X_{y_0})$ is discrete. 

Let us define an interne composition law $\star$ on $E$ by $y \star x := f_y(x)$. According to Proposition \ref{groupappears}, $(E,\star)$ is a topological group with neutral element $y_0$, and $(y \longmapsto c_y)$ is a continuous group homomorphism from $(E,\star)$ to $(\mathbb{R}_+^*, \times)$. Therefore, $(E,\star)$ is the \textbf{bearing group} for $X$ associated with the good invariance components $((f_y, c_y), \ y \in E)$. According to Proposition \ref{levychangentps}, there is a L\'evy process $L$, on $(E,\star)$, such that if we set $\forall t \in [0, +\infty], \ \varphi(t) := \int_0^t 1/c_{L(s)} ds$, then $\zeta(X_{y_0}) = \varphi(+\infty)$ and 
\begin{eqnarray}
\forall \ 0 \leq t < \varphi(+\infty) = \zeta(X_{y_0}), \ X_{y_0}(t) = L \left (\varphi^{-1}(t) \right ). \label{XenfctdeL0}
\end{eqnarray}
Finally, if moreover, for $k \geq 1$, $E$ is a $\mathcal{C}^k$-differentiable manifold and $((f_y, c_y), \ y \in E)$ are $\mathcal{C}^k$ good invariance components, then the last statement in Proposition \ref{groupappears} guaranties that $(E,\star)$ is even a $\mathcal{C}^k$-Lie group and $(y \longmapsto c_y)$ is a $\mathcal{C}^k$-Lie group homomorphism. 

Let us now justify Remark \ref{levyexplarbspace}. Clearly we only need to justify that for all $0 \leq t < \zeta(X_{y_0})$, $\varphi^{-1}(t) = \int_0^t c_{X_{y_0}(u)} du$. In the above proof, an application of Proposition \ref{levychangentps} yielded the existence of a process $L$ satisfying the relation \eqref{XenfctdeL0}, with $\forall t \in [0, +\infty], \ \varphi(t) = \int_0^t 1/c_{L(s)} ds$. According to Remark \ref{chgttimeintermofx}, this implies that we have $0 \leq t < \zeta(X_{y_0})$, $\varphi^{-1}(t) = \int_0^t c_{X_{y_0}(u)} du$, which is the claim. 

\end{proof}

\begin{proof} of Theorem \ref{casgeneral}

Under the assumptions of the theorem, Proposition \ref{existsgoodcomponents} applies so we can produce a family $((f_y, c_y), \ y \in E)$ of $\mathcal{C}^k$ good invariance components associated with $X$, relatively to the reference point $y_0$. Let us define an interne composition law $\star$ on $\mathcal{M}$ by $y \star x := f_y(x)$ and a $\mathcal{C}^k$-function $h : \mathcal{M} \longrightarrow \mathbb{R}_+$ by $h(y) := c_y$. Then, Theorem \ref{casgeneralexpl} guaranties that the three points in the statement of Theorem \ref{casgeneral} are true for this $\star$ and this $h$. 

\end{proof}

\section{Explicit isomorphisms between some Lie groups} \label{explicitisom}


In the previous section we have made appear a group structure on the space $E$ where a process satisfying Definition \ref{self-simoursens} is defined. If $E = I$ (respectively $\mathcal{D}$), an open interval of $\mathbb{R}$ (respectively an open simply connected domain of $\mathbb{R}^2$), this gives rise to a Lie group structure of dimension $1$ (respectively $2$). In this section we provide some explicit group isomorphisms between the Lie groups obtained and canonical examples of Lie groups in dimension $1$ and $2$, on which L\'evy processes are more classical. This allows to construct the explicit diffeomorphisms appearing in Theorems \ref{casrexpl} and \ref{casr2expl}. 

We consider an open set $\mathcal{O} \subset \mathbb{R}^n$ (typically $I \subset \mathbb{R}$ or $\mathcal{D} \subset \mathbb{R}^2$), together with an intern composition law $\star$ such that $(\mathcal{O}, \star)$ is a group. If the group operations (multiplication and inversion) are of class $\mathcal{C}^{k}$ with respect to the usual differentiation on $\mathbb{R}^n$, then we say that $(\mathcal{O}, \star)$ is a $\mathcal{C}^{k}$-Lie group. In other words, the differential manifold structure attached with the Lie group $(\mathcal{O}, \star)$ is always the natural one, arising from the fact that $\mathcal{O}$ is an open subset of $\mathbb{R}^n$. 


We first need a technical result. For a function $f : \mathcal{O} \times \mathcal{O} \longrightarrow \mathcal{O}$ , $J_1 f$ (respectively $J_2 f$) denotes the Jacobian matrix of $f$ when it is differentiated with respect to the first (respectively the second) entry. In the particular case where $\mathcal{O} = I$, an open interval of $\mathbb{R}$, we shall only write $\partial_1 f$ (respectively $\partial_2 f$) for the first (respectively second) partial derivative of $f$. For a function $g : \mathcal{O} \longrightarrow \mathcal{O}$, $Jg$ denotes the Jacobian matrix of $g$. 


\begin{lemme} \label{lemmereljacobiennes}

Let $\mathcal{O} \subset \mathbb{R}^n$ be open, let $f : \mathcal{O} \times \mathcal{O} \longrightarrow \mathcal{O}$ 
be a function and let $\star$ be the interne composition law on $\mathcal{O}$ induced by $f$. If $(\mathcal{O}, \star)$ is a $\mathcal{C}^{k}$-Lie group (let us denote its neutral element by $y_0$), then $J_2 f(u,v)$ and $[ J_2 f(u,v) ]^{-1}$ are defined and $\mathcal{C}^{k-1}$ in $(u,v) \in \mathcal{O} \times \mathcal{O}$ and we have 
\begin{eqnarray}
\forall u, v, \in \mathcal{O}, \ \left [ J_2 f \left ( f(u,v), y_0 \right ) \right ]^{-1}. J_2 f \left ( u, v \right ) = \left [ J_2 f(v,y_0) \right ]^{-1}. \label{reljacobiennes}
\end{eqnarray}

\end{lemme}

\begin{proof}


Since, by assumption, $(\mathcal{O}, \star)$ is a $\mathcal{C}^{k}$-Lie group, the application $((u,v) \longmapsto u \star v) = ((u,v) \longmapsto f(u,v))$ is $\mathcal{C}^{k}$ so its differential with respect to the second entry, $J_2 f(u,v)$, is well-defined at any $(u,v) \in \mathcal{O} \times \mathcal{O}$ and $\mathcal{C}^{k-1}$. Now let us fix an arbitrary $u \in \mathcal{O}$. The application $l_u := (v \longmapsto u \star v) = (v \longmapsto f(u,v))$ is bijective with inverse application $l_{u^{-1}} = (v \longmapsto u^{-1} \star v)$, where $u^{-1}$ denotes the inverse of $u$ for the group law $\star$. $l_{u^{-1}}$ is $\mathcal{C}^{k}$ just as $l_u$ so we get $Id = J_2 f(u^{-1}, f(u,v)). J_2 f(u,v)$. This justifies that $((u,v) \longmapsto [ J_2 f(u,v) ]^{-1})$ is defined at any $(u,v) \in \mathcal{O} \times \mathcal{O}$ and $\mathcal{C}^{k-1}$. 

The associativity property for $\star$ reads 
\[ \forall u, v, w \in \mathcal{O}, \ f \left ( f(u,v), w \right ) = f \left ( u, f(v,w) \right ). \]
Differentiating with respect to $w$ we get 
\[ \forall u, v, w \in \mathcal{O}, \ J_2 f \left ( f(u,v), w \right ) = J_2 f \left ( u, f(v,w) \right ). J_2 f(v,w). \]
Evaluating the latter relation at $w=y_0$ and using that $f(.,y_0) = id_{\mathcal{O}}$ we obtain 
\[ \forall u, v, \in \mathcal{O}, \ J_2 f \left ( f(u,v), y_0 \right ) = J_2 f \left ( u, v \right ). J_2 f(v,y_0). \]
which yields \eqref{reljacobiennes}. 

\end{proof}

Let $I$ be an open interval of $\mathbb{R}$ and let $\star$ be an intern composition law on $I$ such that $(I, \star)$ is a Lie group. Since we will need to differentiate the group operations, it is convenient to consider the function $f : I \times I \longrightarrow I$ associated with the interne composition law $\star$ ($f(y,z) = y \star z$). The following Lemma builds an explicit group isomorphism between $(I, \star)$ and the canonical group $(\mathbb{R}, +)$. 


\begin{lemme} \label{isomorphismedim1}

Let $I \subset \mathbb{R}$ be an open interval, and let $f : I \times I \longrightarrow I$ define a $\mathcal{C}^{k}$-Lie group structure $(I, \star)$ on $I$, for some $k \geq 1$. Define $g : I \longrightarrow \mathbb{R}$ via 
\begin{eqnarray}
\forall y \in I, \ g(y) := \int_{y_0}^y \frac{1}{\partial_2 f (y,y_0)} dy, \label{exprisomdim1}
\end{eqnarray}
where $y_0 \in I$ is the neutral element of $(I, \star)$. Then $g$ is well-defined, it is a $\mathcal{C}^{k}$-diffeomorphism from $I$ to $\mathbb{R}$, and a $\mathcal{C}^{k}$-Lie group isomorphism from $(I, \star)$ to $(\mathbb{R}, +)$. 

\end{lemme}

\begin{proof}

If $g$ is the sought isomorphism, then we should have 
\[ \forall y, z \in I, \ g(f(y,z)) = g(y) + g(z), \]
so differentiating with respect to $z$ and evaluating at $z=y_0$ we get: 
\[ g'(y) = g'(y_0)/\partial_2 f (y,y_0). \]
We see that the derivative $g'$ has to be proportional to the function $y \longmapsto 1/\partial_2 f (y,y_0)$. This explains why the expression \eqref{exprisomdim1} is a natural candidate for $g$. Moreover, note that for any $y \in I$, the application $(z \longmapsto f(y,z)) = (z \longmapsto y \star z)$ is $\mathcal{C}^{k}$ and it has an inverse application $(z \longmapsto y^{-1} \star z)$ which is also $\mathcal{C}^{k}$. Therefore, the derivative of $(z \longmapsto f(y,z))$ does not vanish at $y_0$ for all $y \in I$: $\forall y \in I, \partial_2 f (y,y_0) \neq 0$ and, since $((y,z) \longmapsto f(y,z))$ is $\mathcal{C}^{k}$, we have that $y \longmapsto \partial_2 f (y,y_0)$ is $\mathcal{C}^{k-1}$. 
Therefore, the expression of $g$ in \eqref{exprisomdim1} is well-defined and defines a $\mathcal{C}^{k}$ mapping from $I$ to $\mathbb{R}$. Let us thus consider $g$ defined by \eqref{exprisomdim1}. 

We now justify that $g$ is a group homomorphism from $(I, \star)$ to $(\mathbb{R}^2, +)$. Let $y, x \in I$, we have 
\begin{align}
g(f(y,x)) & = \int_{y_0}^{f(y,x)} \frac{1}{\partial_2 f (z,y_0)} dz \nonumber \\
& = \int_{y_0}^y \frac{1}{\partial_2 f (z,y_0)} dz + \int_{y}^{f(y,x)} \frac{1}{\partial_2 f (z,y_0)} dz \nonumber \\
& = g(y) + \int_{y}^{f(y,x)} \frac{1}{\partial_2 f (z,y_0)} dz \nonumber \\
& = g(y) + \int_{y_0}^{x} \frac{\partial_2 f (y,v)}{\partial_2 f (f(y,v),y_0)} dv, \label{cutpathdim1}
\end{align}
where we have made the change of variable $z = f(y,v)$. Here, the relation \eqref{reljacobiennes} from Lemma \ref{lemmereljacobiennes} reads $\frac{\partial_2 f (y,v)}{\partial_2 f (f(y,v),y_0)} = \frac{1}{\partial_2 f (v,y_0)}$. Therefore, the right hand side of \eqref{cutpathdim1} equals 
\[ g(y) + \int_{y_0}^{x} \frac{1}{\partial_2 f (v,0)} dv = g(y) + g(x), \]
where we have used \eqref{exprisomdim1} to make appear $g(x)$ in the last equality. We have obtained $g(f(y,x)) = g(y) + g(x)$ so $g$ is indeed a group homomorphism from $(I, \star)$ to $(\mathbb{R}, +)$. 

We now justify that $g$ is a $\mathcal{C}^{k}$ diffeomorphism. Clearly, $g'(.) = 1/\partial_2 f (.,y_0)$ is continuous and never vanishes so it has constant sign so $g$ is bijective from $I$ to $g(I)$ and $g(I)$ is an open interval of $\mathbb{R}$ containing $g(y_0) = 0$. Since moreover $g(I)$ is a subgroup of $(\mathbb{R},+)$ we conclude that $g(I)=\mathbb{R}$. Therefore, $g$ is bijective from $I$ to $\mathbb{R}$. Since $g$ is $\mathcal{C}^{k}$ and its derivative never vanishes, it is even a $\mathcal{C}^{k}$ diffeomorphism from $I$ to $\mathbb{R}$. Since $g$ is also a group homomorphism from $(I, \star)$ to $(\mathbb{R}, +)$, it is even a $\mathcal{C}^{k}$-Lie group isomorphism from $(I, \star)$ to $(\mathbb{R}, +)$, which concludes the proof. 

\end{proof}

We now consider the two-dimensional case. Let $\mathcal{D}$ be an open simply connected domain of $\mathbb{R}^2$ and let $\star$ be an intern composition law on $\mathcal{D}$ such that $(\mathcal{D}, \star)$ is a Lie group. Since we will need to differentiate the group operations, it is convenient to introduce $f : \mathcal{D} \times \mathcal{D} \longrightarrow \mathcal{D}$ associated with the interne composition law $\star$ ($f(y,z) := y \star z$). 

We now recall some classic facts and notations about Lie groups and their associated Lie algebras, but we formulate everything in the context of $(\mathcal{D}, \star)$. First let us assume for convenience that $(\mathcal{D}, \star)$ is a $\mathcal{C}^{\infty}$-Lie group. Given a $\mathcal{C}^{\infty}$-vector field $(V_y, y \in \mathcal{D})$ on $\mathcal{D}$ and a $\mathcal{C}^{\infty}$ function $h : \mathcal{D} \longrightarrow \mathbb{R}$ we define the $\mathcal{C}^{\infty}$ function $V.h : \mathcal{D} \longrightarrow \mathbb{R}$ by $(V.h)(y) := Dh(y)[V_y]$. The commutator of two vector fields $U, V$ is still a vector field that we denote $[U, V]$: 
\begin{eqnarray}
\forall h \in \mathcal{C}^{\infty}(\mathcal{D}, \mathbb{R}), \ [U, V].h := U.(V.h) - V.(U.h). \label{corchetdef}
\end{eqnarray}
The vector space of $\mathcal{C}^{\infty}$-vector fields on $\mathcal{D}$ equipped with $[.,.]$ is a Lie algebra. A vector field $(V_y, y \in \mathcal{D})$ on $\mathcal{D}$ is called left-invariant if $\forall y, z \in \mathcal{D}, V_{y \star z} = J_2 f(y,z) [V_{z}]$. For any $u \in \mathbb{R}^2$, we can define a left-invariant vector field $V(u)$ by $V(u)_y := J_2 f(y,y_0) [u]$. It can be seen that $(u \longmapsto V(u))$ is a linear isomorphism (we will denote its inverse by $V^{-1}$) so the vector space of left-invariant vector fields is isomorphic to $\mathbb{R}^2$. Moreover, it is a well known property of Lie groups that the commutator of two left-invariant vector fields is still a left-invariant vector field. In particular we have 
\begin{eqnarray}
[V(u), V(v)]_y = J_2 f (y, y_0). [ V(u), V(v) ]_{y_0}, \ \forall y \in \mathcal{D}, u, v \in \mathbb{R}^2. \label{corchetleftinvar}
\end{eqnarray}
Using the correspondance between $\mathbb{R}^2$ and left-invariant vector fields we can define a Lie bracket on $\mathbb{R}^2$ by 
\begin{eqnarray}
\forall u, v \in \mathbb{R}^2, \ [u, v]_L := V^{-1}([ V(u), V(v) ]). \label{corchetr2}
\end{eqnarray}
Then, by construction, the Lie algebra $(\mathbb{R}^2, [.,.]_L)$ is isomorphic to the Lie algebra of left-invariant vector fields, and is usually called the Lie algebra associated with the Lie group $\mathcal{D}$. 

Also, using the definition of the vector fields $V(u), V(v)$ and the definition \eqref{corchetdef} of the commutator, we can give an explicit expression for the vector field $[V(u), V(v)]$ in term of $f(.,.)$: for any $y \in \mathcal{D}$ and $u, v \in \mathbb{R}^2$ we have 
\begin{eqnarray}
[V(u), V(v)]_y = J_1 J_2 f (y, y_0) \left [ J_2 f (y, y_0).u, v \right ] - J_1 J_2 f (y, y_0) \left [ J_2 f (y, y_0).v, u \right ]. \label{corchetexpl}
\end{eqnarray}

In the context where $(\mathcal{D}, \star)$ is only a $\mathcal{C}^{k}$-Lie group for some $k \geq 2$, the left-invariant vector fields $V(u)$ can still be defined by $V(u)_y := J_2 f(y,y_0) [u]$ and they are $\mathcal{C}^{k-1}$. The commutator $[V(u), V(v)]$ can still be defined via \eqref{corchetdef} (and its action on $\mathcal{C}^{k}(\mathcal{D}, \mathbb{R})$) and we still have \eqref{corchetexpl} by direct computation. Finally, it is still possible to see that $[V(u), V(v)]$ is left-invariant so the relation \eqref{corchetleftinvar} still holds and the Lie algebra $(\mathbb{R}^2, [.,.]_L)$ can still be defined by \eqref{corchetr2}. 

Using the so-called \textit{exponential map} it can be seen that the Lie group $(\mathcal{D}, \star)$ is commutative if and only if the Lie bracket $[.,.]_L$ of its associated Lie algebra $(\mathbb{R}^2, [.,.]_L)$ is trivial (i.e. $[u, v]_L = 0$ for all $u, v \in \mathbb{R}^2$). 

Let us now investigate the possible forms of the Lie bracket $[.,.]_L$, to deduce another expression for $[V(u), V(v)]_{y_0}$. Let $\{ e_1, e_2 \}$ be the canonical basis in $\mathbb{R}^2$, so that every $u$ and $v$ in $\mathbb{R}^2$ can be decomposed as $u = u_1 e_1 + u_2 e_2$, $v = v_1 e_1 + v_2 e_2$. Then we have 
\begin{align*}
[u, v]_L & = [u_1 e_1 + u_2 e_2, v_1 e_1 + v_2 e_2]_L \\
& = u_1 v_1 [e_1, e_1]_L + u_1 v_2 [e_1, e_2]_L + u_2 v_1 [e_2, e_1]_L + u_2 v_2 [e_2, e_2]_L \\
& = (u_1 v_2 - u_2 v_1) [e_1, e_2]_L \\
& = \det(u,v) [e_1, e_2]_L, 
\end{align*}
where we have used the bi-linearity of $[.,.]_L$ and the fact that it is anti-symmetric. Let us set $Y_0 := [V(e_1), V(e_2)]_{y_0}$ which can be computed explicitly using \eqref{corchetexpl}. Then for any $u, v \in \mathbb{R}^2$ we have $[V(u), V(v)]_{y_0} = V([u, v]_L)_{y_0} = \det(u,v) V([e_1, e_2]_L)_{y_0} = \det(u,v) [V(e_1), V(e_2)]_{y_0} = \det(u,v) Y_0$. Thus, 
\begin{eqnarray}
\forall u, v \in \mathbb{R}^2, \ [V(u), V(v)]_{y_0} = \det(u,v) Y_0. \label{corchetdet}
\end{eqnarray}
Note in particular that we have $Y_0=0$ if and only if $(\mathcal{D}, \star)$ is commutative. 

We have now all the tools required to prove the following Lemmas \ref{isomorphismecasclassique} and \ref{isomorphismecasnonclassique} that build an explicit group isomorphism between $(\mathcal{D}, \star)$ and one of the two canonical Lie group structures on $\mathbb{R}^2$. The first canonical Lie group structure on $\mathbb{R}^2$ is of course $(\mathbb{R}^2, +)$, and the second is $(\mathbb{R}^2, T)$, defined as follows: 
\begin{eqnarray}
\forall \left( \begin{array}{c}
y_1 \\
y_2 
\end{array} \right), \left( \begin{array}{c}
z_1 \\
z_2 
\end{array} \right) \in \mathbb{R}^2, \ y T z := \left( \begin{array}{c}
y_1 + z_1 \\
y_2 + e^{y_1} z_2
\end{array} \right). \label{canononcommutdim2}
\end{eqnarray}
It is easy to see that $T$ gives rise to a group structure on $\mathbb{R}^2$ and that $(\mathbb{R}^2, T)$ is not commutative. As a consequence, the two-dimensional Lie groups $(\mathcal{D}, \star)$ that we are considering will be isomorphic to $(\mathbb{R}^2, +)$ or $(\mathbb{R}^2, T)$ depending on whether they are commutative or not.




\begin{lemme} \label{isomorphismecasclassique}

Let $\mathcal{D} \subset \mathbb{R}^2$ be an open simply connected domain, and let $f : \mathcal{D} \times \mathcal{D} \longrightarrow \mathcal{D}$ define a \textbf{commutative} $\mathcal{C}^{k}$-Lie group structure $(\mathcal{D}, \star)$ on $\mathcal{D}$, for some $k \geq 2$. Fix $M \in GL_2(\mathbb{R})$ and define $g : \mathcal{D} \longrightarrow \mathbb{R}^2$ via 
\begin{eqnarray}
\forall y \in \mathcal{D}, \ g(y) := \int_a^b M. \left [ J_2 f (\gamma(s), y_0) \right ]^{-1}. \gamma'(s) ds, \label{lemmepoincare}
\end{eqnarray}
where $y_0 \in \mathcal{D}$ is the neutral element of $(\mathcal{D}, \star)$, $J_2 f (z, .)$ denotes the Jacobian matrix of the function $f(z,.) : \mathcal{D} \longrightarrow \mathcal{D}$, and where $\gamma : [a, b] \longrightarrow \mathcal{D}$ is any path, locally $\mathcal{C}^1$, with $\gamma(a) = y_0$ and  $\gamma(b) = y$. Then $g$ is well-defined, it is a $\mathcal{C}^{k}$-diffeomorphism from $\mathcal{D}$ to $\mathbb{R}^2$, we have $Jg(y_0)=M$, and $g$ is a $\mathcal{C}^{k}$-Lie group isomorphism from $(\mathcal{D}, \star)$ to $(\mathbb{R}^2, +)$. 


\end{lemme}


\begin{proof}

If $g$ is the sought isomorphism, then, reasoning as in the beginning of the proof of Lemma \ref{isomorphismedim1}, we see that the Jacobian of $g$ has to be $y \longmapsto Q.(J_2 f (y, y_0))^{-1}$ where $Q$ is an invertible $2 \times 2$ matrix. This explains why the expression \eqref{lemmepoincare} is a natural candidate for $g$. Moreover, thanks to Lemma \ref{lemmereljacobiennes} applied with $(u,v) = (y,y_0)$, we can see that $J_2 f (y, y_0)$ is invertible for any $y \in \mathcal{D}$ and that $y \longmapsto (J_2 f (y, y_0))^{-1}$ is of class $\mathcal{C}^{k-1}$. To prove that $g$ is well-defined and $\mathcal{C}^{k}$, we thus only need to prove the existence of a primitive function for $y \longmapsto M.(J_2 f (y, y_0))^{-1}$. The idea is to apply Poincar\'e's Lemma. 


For any $y \in \mathcal{D}$, let us denote by $L_y$ the linear application that sends $z \in \mathbb{R}^2$ to $M. (J_2 f (y, y_0))^{-1}.z \in \mathbb{R}^2$. We wish to apply Poincar\'e's Lemma to the differential forms $y \longmapsto \pi_1 \circ L_y$ and $y \longmapsto \pi_2 \circ L_y$. Let us show that these differential forms are closed. For $i \in \{1, 2\}$ and $y \in \mathcal{D}$ we have 
\begin{align*}
& \frac{d}{dy_2} \left [ \pi_i \circ L_y(e_1) \right ] - \frac{d}{dy_1} \left [ \pi_i \circ L_y(e_2) \right ] = \pi_i \left [ \frac{d}{dy_2} L_y(e_1) - \frac{d}{dy_1} L_y(e_2) \right ] \\
= & \pi_i \left ( M. \left [ \frac{d}{dy_2} (J_2 f (y, y_0))^{-1}.e_1 - \frac{d}{dy_1} (J_2 f (y, y_0))^{-1}.e_2 \right ] \right ) \\
= & - \pi_i \left ( M. \left [ (J_2 f (y, y_0))^{-1}. J_1 J_2 f (y, y_0) \left [ e_2, (J_2 f (y, y_0))^{-1}.e_1 \right ] \right . \right . \\ 
& \left . \left. - (J_2 f (y, y_0))^{-1}. J_1 J_2 f (y, y_0) \left [ e_1, (J_2 f (y, y_0))^{-1}.e_2 \right ] \right ] \right ) \\
= & \pi_i \left ( M. (J_2 f (y, y_0))^{-1}. \left [ J_1 J_2 f (y, y_0) \left [ e_1, (J_2 f (y, y_0))^{-1}.e_2 \right ] - J_1 J_2 f (y, y_0) \left [ e_2, (J_2 f (y, y_0))^{-1}.e_1 \right ] \right ] \right ) \\
= & \pi_i \left ( M. (J_2 f (y, y_0))^{-1}. \left [ V((J_2 f (y, y_0))^{-1}.e_1), V((J_2 f (y, y_0))^{-1}.e_2) \right ]_y \right ) = 0. 
\end{align*}
In the last line we have recognized the expression of the commutator in \eqref{corchetexpl} and used that it is equal to $0$ since $(\mathcal{D}, \star)$ is commutative. 
Therefore 
the differential forms $y \longmapsto \pi_1 \circ L_y$ and $y \longmapsto \pi_2 \circ L_y$ are closed. Then, since $\mathcal{D}$ is simply connected, Poincar\'e's Lemma applies and shows that there are two functions $g_1 : \mathcal{D} \longrightarrow \mathbb{R}$ and $g_2 : \mathcal{D} \longrightarrow \mathbb{R}$ such that $g_i(y_0) = 0$ and for any $y \in \mathcal{D}$, $d g_i(y)$, the differential of $g_i$ at $y$, is $\pi_i \circ L_y = \pi_i ( M. (J_2 f (y, y_0))^{-1}.)$. 

Now, defining the function $g : \mathcal{D} \longrightarrow \mathbb{R}^2$ by $g := (g_1, g_2)$, 
we have obviously that $g(y_0)=0$ and that $\forall y \in \mathcal{D}, Jg(y) = M. (J_2 f (y, y_0))^{-1}$. Therefore $g$ is $\mathcal{C}^{k}$, and for any $y \in \mathcal{D}$ and any locally $\mathcal{C}^1$ path $\gamma : [a, b] \longrightarrow \mathcal{D}$ with $\gamma(a) = y_0, \gamma(b) = y$, the derivative of the function $g \circ \gamma : [a, b] \longrightarrow \mathbb{R}^2$ at $s \in [a,b]$ is clearly $M. (J_2 f (\gamma(s), y_0) )^{-1}. \gamma'(s)$. We thus have 
\[ g(y) = g(y) - g(y_0) = g(\gamma(b)) - g(\gamma(a)) = \int_a^b M. \left [ J_2 f (\gamma(s), y_0) \right ]^{-1}. \gamma'(s) ds, \]
so that \eqref{lemmepoincare} holds for any $y \in \mathcal{D}$ for the $g$ that we have just defined. Therefore $g$ is well-defined. 

Since $g$ is a primitive function of $y \longmapsto M.(J_2 f (y, y_0))^{-1}$ we have $Jg(y_0) = M.(J_2 f (y_0, y_0))^{-1}$, but $f (y_0, .)$ is the identity function on $\mathcal{D}$ so $J_2 f (y_0, y_0)$ is the identity matrix so $Jg(y_0) = M$. 

We now justify that $g$ is a group homomorphism from $(\mathcal{D}, \star)$ to $(\mathbb{R}^2, +)$. Let $y, x \in \mathcal{D}$ and choose a locally $\mathcal{C}^{1}$ path $\gamma : [0, 2] \longrightarrow \mathcal{D}$ with $\gamma(0) = y_0, \gamma(1) = y, \gamma(2) = f(y, x)$. Using two times \eqref{lemmepoincare}, we have 
\begin{align}
g(f(y,x)) & = \int_0^2 M. \left [ J_2 f (\gamma(s), y_0) \right ]^{-1}. \gamma'(s) ds \nonumber \\
& = \int_0^1 M. \left [ J_2 f (\gamma(s), y_0) \right ]^{-1}. \gamma'(s) ds + \int_1^2 M. \left [ J_2 f (\gamma(s), y_0) \right ]^{-1}. \gamma'(s) ds \nonumber \\
& = g(y) + \int_1^2 M. \left [ J_2 f (\gamma(s), y_0) \right ]^{-1}. \gamma'(s) ds. \label{cutpath}
\end{align}
Now, let $h_y$ denote the inverse function of $f(y, .)$ ($h_y := f(y^{-1}, .)$) and let us consider the locally $\mathcal{C}^{1}$ path $\tilde \gamma : [1, 2] \longrightarrow \mathcal{D}$ defined by $\forall s \in [1, 2], \tilde \gamma(s) := h_y (\gamma(s))$. In particular we have $\tilde \gamma(1) = y_0, \tilde \gamma(2) = x$ and 
\[ \forall s \in [1, 2], \ \gamma (s) = f(y, \tilde \gamma(s)), \ \gamma' (s) = J_2 f(y, \tilde \gamma(s)). \tilde \gamma'(s). \]
As a consequence, the second term in the right hand side of \eqref{cutpath} equals 
\[ \int_1^2 M. \left [ J_2 f (f(y, \tilde \gamma(s)), y_0) \right ]^{-1}. J_2 f(y, \tilde \gamma(s)). \tilde \gamma'(s) ds = \int_1^2 M. \left [ J_2 f (\tilde \gamma(s), y_0) \right ]^{-1}. \tilde \gamma'(s) ds = g(x), \]
where we have used \eqref{reljacobiennes} with $u=y, v =\tilde \gamma(s)$, and then \eqref{lemmepoincare}. Putting into \eqref{cutpath} we obtain $g(f(y,x)) = g(y) + g(x)$ so $g$ is indeed a group homomorphism from $(\mathcal{D}, \star)$ to $(\mathbb{R}^2, +)$. 

We now justify that $g$ is onto $\mathbb{R}^2$. $Jg(y_0) = M$ is invertible so, by the \textit{local inversion theorem}, $g(\mathcal{D})$ contains a neighborhood in $\mathbb{R}^2$ of $g(y_0) = (0,0)$. Since moreover $g(\mathcal{D})$ is a subgroup of $(\mathbb{R}^2,+)$ we conclude that $g(\mathcal{D})=\mathbb{R}^2$. 

Let us now justify that $g$ is injective. Assume the contrary, that is, $\ker g \neq \{ y_0 \}$. We have that $Jg(y) = M. (J_2 f (y, y_0))^{-1}$ is invertible for any $y \in \mathcal{D}$ so, thanks to the \textit{local inversion theorem}, we get that $\ker g$ is a discrete group. Let $(Q, \star)$ be the quotient group of $(\mathcal{D}, \star)$ by $\ker g$, equipped with the quotient topology (we still denote the group composition law on $Q$ by $\star$). Then, let $\hat g : Q \longrightarrow \mathbb{R}^2$ denote the quotient application of $g$. Then $\hat g$ is a continuous injective group homomorphism between $(Q, \star)$ and $(\mathbb{R}^2, +)$. $\hat g$ is also onto because $g$ is. As a consequence $\hat g$ is bijective. Let us justify that $\hat g^{-1}$ is continuous: let $(x_n)_{n \geq 1}$ be a sequence converging to some $x \in \mathbb{R}^2$, and let us justify that $\hat g^{-1}(x_n)$ converges to $\hat g^{-1}(x)$. Let $y \in \mathcal{D}$ be an antecedent of $x$ by $g$. Then, the local inversion theorem gives the existence of a neighborhood $\mathcal{U}$ of $x$ and of a neighborhood $\mathcal{V}$ of $y$ such that $g$ is a $\mathcal{C}^{k}$-diffeomorphism from $\mathcal{V}$ to $\mathcal{U}$. Let us choose $n_0$ such that $x_n \in \mathcal{U}$ for all $n \geq n_0$. For such $n \geq n_0$, let us define $y_n \in \mathcal{V}$ to be the antecedent in $\mathcal{V}$ of $x_n$ by $g$. Since $g$ is a $\mathcal{C}^{k}$-diffeomorphism from $\mathcal{V}$ to $\mathcal{U}$ we get that $y_n$ converges to $y$. Let $\pi : \mathcal{D} \longrightarrow Q$ denote the quotient projection. Since $\pi$ is continuous we get that $\pi(y_n)$ converges to $\pi(y)$. Finally, noting that $\pi(y_n) = \hat g^{-1}(x_n)$ and $\pi(y) = \hat g^{-1}(x)$, we obtain the sought convergence. In conclusion have that $\hat g$ is a bi-continuous isomorphism of groups between $(Q, \star)$ and $(\mathbb{R}^2, +)$. In particular $\hat g$ is an homeomorphism between $Q$ and $\mathbb{R}^2$ so the latter two are homeomorphic. Now, let $\gamma : [0, 1] \longrightarrow \mathcal{D}$ be a continuous path such that $\gamma(0) = y_0$ and $\gamma(1) \in \ker g \setminus \{ y_0 \}$. Clearly, $\pi \circ \gamma : [0, 1] \longrightarrow Q$ is a continuous loop that cannot be continuously reduced to a single point. Therefore the topological group of $Q$ is not trivial, so $Q$ is not homeomorphic to $\mathbb{R}^2$ (whose topological group is trivial). This is a contradiction, therefore $g$ is injective. 

Combining what is proved above, we obtain that $g$ is bijective from $\mathcal{D}$ to $\mathbb{R}^2$ and $\mathcal{C}^{k}$. Recall that $Jg(y) = M. (J_2 f (y, y_0))^{-1}$ is invertible for any $y \in \mathcal{D}$. 
Therefore we can conclude that $g$ is a $\mathcal{C}^{k}$-diffeomorphism from $\mathcal{D}$ to $\mathbb{R}^2$. 
Since $g$ is also a group homomorphism from $(\mathcal{D}, \star)$ to $(\mathbb{R}^2, +)$, it is even a $\mathcal{C}^{k}$-Lie group isomorphism from $(\mathcal{D}, \star)$ to $(\mathbb{R}^2, +)$, which concludes the proof. 


\end{proof}


\begin{lemme} \label{isomorphismecasnonclassique}

Let $\mathcal{D} \subset \mathbb{R}^2$ be an open simply connected domain, and let $f : \mathcal{D} \times \mathcal{D} \longrightarrow \mathcal{D}$ define a \textbf{non-commutative} $\mathcal{C}^{k}$-Lie group structure $(\mathcal{D}, \star)$ on $\mathcal{D}$, for some $k \geq 2$. 

Let $Y_0 := [V(e_1), V(e_2)]_{y_0}$ which can be computed explicitly using \eqref{corchetexpl}, and let $M := 
\begin{pmatrix}
\pi_2(Y_0) & -\pi_1(Y_0) \\
\pi_1(Y_0) & \pi_2(Y_0)
\end{pmatrix}$. 
Define $g_1 : \mathcal{D} \longrightarrow \mathbb{R}$ via 
\begin{eqnarray}
g_1(y) = \int_a^b \pi_1 \left ( M. [J_2 f (\gamma(s), y_0)]^{-1}. \gamma'(s) \right ) ds, \label{lemmepoincarenonclassique1}
\end{eqnarray}
where $y_0 \in \mathcal{D}$ is the neutral element of $(\mathcal{D}, \star)$, $J_2 f (z, .)$ denotes the Jacobian matrix of the function $f(z,.) : \mathcal{D} \longrightarrow \mathcal{D}$, and where $\gamma : [a, b] \longrightarrow \mathcal{D}$ is any path, locally $\mathcal{C}^1$, with $\gamma(a) = y_0$ and  $\gamma(b) = y$. 
Define $g_2 : \mathcal{D} \longrightarrow \mathbb{R}$ via 
\begin{eqnarray}
g_2(y) = \int_a^b e^{g_1(\gamma(s))} \pi_2 \left ( M. [J_2 f (\gamma(s), y_0)]^{-1}. \gamma'(s) \right ) ds, \label{lemmepoincarenonclassique2}
\end{eqnarray}
where $\gamma : [a, b] \longrightarrow \mathcal{D}$ is any path, locally $\mathcal{C}^1$, with $\gamma(a) = y_0$ and  $\gamma(b) = y$. Then $g_1$ and $g_2$ are well-defined, $g := (g_1, g_2)$ is a $\mathcal{C}^{k}$-diffeomorphism from $\mathcal{D}$ to $\mathbb{R}^2$, and a $\mathcal{C}^k$-Lie group isomorphism from $(\mathcal{D}, \star)$ to $(\mathbb{R}^2, T)$. 

\end{lemme}

\begin{proof}

If $g$ is the sought isomorphism, then, reasoning as in the beginning of the proof of Lemma \ref{isomorphismedim1}, we see that the Jacobian matrix of $g$ has to be $y \longmapsto \begin{pmatrix}
1 & 0 \\
0 & e^{g_1(y)}
\end{pmatrix}. Q. (J_2 f (y, y_0))^{-1}$ where $Q$ is an invertible $2 \times 2$ matrix. This explains why the definition of $g$ in the statement of the lemma is a natural candidate for the isomorphism. Here a again, thanks to Lemma \ref{lemmereljacobiennes} applied with $(u,v) = (y,y_0)$, we can see that $J_2 f (y, y_0)$ is invertible for any $y \in \mathcal{D}$ and that $y \longmapsto (J_2 f (y, y_0))^{-1}$ is of class $\mathcal{C}^{k-1}$. To prove that $g$ is well-defined and $\mathcal{C}^{k}$, we thus only need to prove the existence of a primitive function, $g_1$, for $y \longmapsto \pi_1 (M. (J_2 f (y, y_0))^{-1}.)$, and then to prove the existence of a primitive function, $g_2$, for $y \longmapsto e^{g_1(y)} \pi_2 ( M. (J_2 f (y, y_0))^{-1}.)$. It will appear that the $2 \times 2$ matrix $M$ has been chosen in such a way so that it is possible apply Poincar\'e's Lemma to these differential forms. 

For any $y \in \mathcal{D}$, let us denote by $d^1_y$ the linear form that sends $z \in \mathbb{R}^2$ to $\pi_1 ( M. (J_2 f (y, y_0))^{-1}.z) \in \mathbb{R}^2$. We wish to apply Poincar\'e's Lemma to the differential form $y \longmapsto d^1_y$. Let us check that it is closed. For $y \in \mathcal{D}$ we have 
\begin{align*}
& \frac{d}{dy_2} \left [ d^1_y(e_1) \right ] - \frac{d}{dy_1} \left [ d^1_y(e_2) \right ] = \pi_1 \left ( M. \left [ \frac{d}{dy_2} (J_2 f (y, y_0))^{-1}. e_1 - \frac{d}{dy_1} (J_2 f (y, y_0))^{-1}. e_2 \right ] \right ) \\
= & - \pi_1 \left ( M. \left [ (J_2 f (y, y_0))^{-1}. J_1 J_2 f (y, y_0) \left [ e_2, (J_2 f (y, y_0))^{-1}.e_1 \right ] \right . \right . \\ 
& \left . \left . - (J_2 f (y, y_0))^{-1}. J_1 J_2 f (y, y_0) \left [ e_1, (J_2 f (y, y_0))^{-1}.e_2 \right ] \right ] \right ) \\
= & \pi_1 \left ( M. (J_2 f (y, y_0))^{-1}. \left [ J_1 J_2 f (y, y_0) \left [ e_1, (J_2 f (y, y_0))^{-1}.e_2 \right ] - J_1 J_2 f (y, y_0) \left [ e_2, (J_2 f (y, y_0))^{-1}.e_1 \right ] \right ] \right ) \\
= & \pi_1 \left ( M. (J_2 f (y, y_0))^{-1}. \left [ V((J_2 f (y, y_0))^{-1}.e_1), V((J_2 f (y, y_0))^{-1}.e_2) \right ]_y \right ) \\
= & \pi_1 \left ( M. (J_2 f (y, y_0))^{-1}. J_2 f (y, y_0). \left [ V((J_2 f (y, y_0))^{-1}.e_1), V((J_2 f (y, y_0))^{-1}.e_2) \right ]_{y_0} \right ) \\
= & \pi_1 \left ( M. \left [ V((J_2 f (y, y_0))^{-1}.e_1), V((J_2 f (y, y_0))^{-1}.e_2) \right ]_{y_0} \right ) \\
= & \det((J_2 f (y, y_0))^{-1}) \times \pi_1 \left ( M. Y_0 \right ) = 0. 
\end{align*}
In the above we have recognized the expression of the commutator in \eqref{corchetexpl}, used \eqref{corchetleftinvar}, \eqref{corchetdet}, and the fact, that can easily be seen from the definition of $M$, that $\pi_1 ( M. Y_0 ) = 0$. This proves that the differential form $d^1_y$ is closed. Then, since $\mathcal{D}$ is simply connected, Poincar\'e's Lemma ensures the existence of a function $g_1 : \mathcal{D} \longrightarrow \mathbb{R}$ such that $g_1(y_0) = 0$ and for any $y \in \mathcal{D}$, $d g_1(y)$, the differential of $g_1$ at $y$, is $d^1_y = \pi_1 ( M. (J_2 f (y, y_0))^{-1}.)$. Therefore $g_1$ is $\mathcal{C}^{k}$, and for any $y \in \mathcal{D}$ and any locally $\mathcal{C}^{1}$ path $\gamma : [a, b] \longrightarrow \mathcal{D}$ with $\gamma(a) = y_0, \gamma(b) = y$, the derivative of the function $g_1 \circ \gamma : [a, b] \longrightarrow \mathbb{R}^2$ at $s \in [a,b]$ is clearly $\pi_1 ( M. (J_2 f (\gamma(s), y_0) )^{-1}. \gamma'(s))$, so \eqref{lemmepoincarenonclassique1} follows. 

Now that $g_1$ is defined, we can define $d^2_y$ for any $y \in \mathcal{D}$ as the linear form that sends $z \in \mathbb{R}^2$ to $e^{g_1(y)} \pi_2 ( M. (J_2 f (y, y_0))^{-1}.z) \in \mathbb{R}^2$. Here again, we wish to apply Poincar\'e's Lemma to the differential form $y \longmapsto d^2_y$ so we first check that it is closed. For $y \in \mathcal{D}$ we have 
\begin{align*}
& \frac{d}{dy_2} \left [ d^2_y(e_1) \right ] - \frac{d}{dy_1} \left [ d^2_y(e_2) \right ] \nonumber \\
= & \pi_2 \left ( M. \left ( \frac{d}{dy_2} \left [ e^{g_1(y)}. (J_2 f (y, y_0))^{-1} \right ]. e_1 - \frac{d}{dy_1} \left [ e^{g_1(y)}. (J_2 f (y, y_0))^{-1}. e_2 \right ] \right ) \right ) \nonumber \\
= & \pi_2 \left ( M. \left ( -e^{g_1(y)}. (J_2 f (y, y_0))^{-1}. J_1 J_2 f (y, y_0) \left [ e_2, (J_2 f (y, y_0))^{-1}.e_1 \right ]  + \left (\frac{d}{dy_2} g_1(y) \right ). e^{g_1(y)}. (J_2 f (y, y_0))^{-1}. e_1 \right . \right . \nonumber \\ 
& \left . \left . + e^{g_1(y)}. (J_2 f (y, y_0))^{-1}. J_1 J_2 f (y, y_0) \left [ e_1, (J_2 f (y, y_0))^{-1}.e_2 \right ] - \left (\frac{d}{dy_1} g_1(y) \right ). e^{g_1(y)}. (J_2 f (y, y_0))^{-1}. e_2 \right ) \right ) \nonumber \\
\end{align*}

\begin{align}
= & e^{g_1(y)} \times \pi_2 \left ( M. (J_2 f (y, y_0))^{-1}. \left ( - J_1 J_2 f (y, y_0) \left [ e_2, (J_2 f (y, y_0))^{-1}.e_1 \right ] + \left (\frac{d}{dy_2} g_1(y) \right ). e_1 \right . \right . \nonumber \\ 
& \left . \left . + J_1 J_2 f (y, y_0) \left [ e_1, (J_2 f (y, y_0))^{-1}.e_2 \right ] - \left (\frac{d}{dy_1} g_1(y) \right ). e_2 \right ) \right ) \nonumber \\
= & e^{g_1(y)} \times \pi_2 \left ( M. (J_2 f (y, y_0))^{-1}. \left ( \left [ V((J_2 f (y, y_0))^{-1}.e_1), V((J_2 f (y, y_0))^{-1}.e_2) \right ]_y \right . \right . \nonumber \\
+ & \left . \left . \left (\frac{d}{dy_2} g_1(y) \right ). e_1 - \left (\frac{d}{dy_1} g_1(y) \right ). e_2 \right ) \right ) \nonumber \\
= & e^{g_1(y)} \times \pi_2 \left ( M. (J_2 f (y, y_0))^{-1}. \left ( J_2 f (y, y_0). \left [ V((J_2 f (y, y_0))^{-1}.e_1), V((J_2 f (y, y_0))^{-1}.e_2) \right ]_{y_0} \right . \right . \nonumber \\
+ & \left . \left . \left (\frac{d}{dy_2} g_1(y) \right ). e_1 - \left (\frac{d}{dy_1} g_1(y) \right ). e_2 \right ) \right ) \nonumber \\
= & e^{g_1(y)} \times \pi_2 \left ( M. \left [ V((J_2 f (y, y_0))^{-1}.e_1), V((J_2 f (y, y_0))^{-1}.e_2) \right ]_{y_0} \right ) \nonumber \\
+ & e^{g_1(y)} \left ( \left (\frac{d}{dy_2} g_1(y) \right ) \times \pi_2 \left ( M. (J_2 f (y, y_0))^{-1}. e_1 \right ) - \left (\frac{d}{dy_1} g_1(y) \right ) \times \pi_2 \left ( M. (J_2 f (y, y_0))^{-1}. e_2 \right ) \right ). \label{nonclassic1}
\end{align}
In the above we have recognized the expression of the commutator in \eqref{corchetexpl} ans used \eqref{corchetleftinvar}. Then, using \eqref{corchetdet}, and the definition of $M$: 
\begin{align}
& \pi_2 \left ( M. \left [ V((J_2 f (y, y_0))^{-1}.e_1), V((J_2 f (y, y_0))^{-1}.e_2) \right ]_{y_0} \right ) \nonumber \\
= & \det \left ( (J_2 f (y, y_0))^{-1} \right ) \times \pi_2 \left ( M. Y_0 \right ) \nonumber \\
= & \det \left ( (J_2 f (y, y_0))^{-1} \right ) \times (\pi_1(Y_0)^2 + \pi_2(Y_0)^2). \label{nonclassic2}
\end{align}
On the other hand, recall that the differential of $g_1$ at $y$, is $\pi_1 ( M. (J_2 f (y, y_0))^{-1}.)$. In particular we have $\frac{d}{dy_2} g_1(y) = (M. (J_2 f (y, y_0))^{-1})_{1, 2}$ and $\frac{d}{dy_1} g_1(y) = (M. (J_2 f (y, y_0))^{-1})_{1, 1}$. Also, we have obviously $\pi_2 ( M. (J_2 f (y, y_0))^{-1}. e_1) = (M. (J_2 f (y, y_0))^{-1})_{2, 1}$ and $\pi_2 ( M. (J_2 f (y, y_0))^{-1}. e_2) = (M. (J_2 f (y, y_0))^{-1})_{2, 2}$. We thus get 
\begin{align}
& \left (\frac{d}{dy_2} g_1(y) \right ) \times \pi_2 ( M. (J_2 f (y, y_0))^{-1}. e_1) - \left (\frac{d}{dy_1} g_1(y) \right ) \times \pi_2 ( M. (J_2 f (y, y_0))^{-1}. e_2) \nonumber \\
= & (M. (J_2 f (y, y_0))^{-1})_{1, 2} \times (M. (J_2 f (y, y_0))^{-1})_{2, 1} - (M. (J_2 f (y, y_0))^{-1})_{1, 1} \times (M. (J_2 f (y, y_0))^{-1})_{2, 2} \nonumber \\
= & - \det \left ( M. (J_2 f (y, y_0))^{-1} \right ) = - \det \left ( (J_2 f (y, y_0))^{-1} \right ) \times \det(M) \nonumber \\
= & - \det \left ( (J_2 f (y, y_0))^{-1} \right ) \times (\pi_1(Y_0)^2 + \pi_2(Y_0)^2). \label{nonclassic3}
\end{align}
Putting \eqref{nonclassic2} and \eqref{nonclassic3} into \eqref{nonclassic1} we obtain $\frac{d}{dy_2} \left [ d^2_y(e_1) \right ] - \frac{d}{dy_1} \left [ d^2_y(e_2) \right ] = 0$, that is, the differential form $y \longmapsto d^2_y$ is closed. Then, since $\mathcal{D}$ is simply connected, Poincar\'e's Lemma ensures the existence of a function $g_2 : \mathcal{D} \longrightarrow \mathbb{R}$ such that $g_2(y_0) = 0$ and for any $y \in \mathcal{D}$, $d g_2(y)$, the differential of $g_2$ at $y$, is $d^2_y = e^{g_1(y)} \pi_2 ( M. (J_2 f (y, y_0))^{-1}.)$. Therefore $g_2$ is $\mathcal{C}^{k}$, and for any $y \in \mathcal{D}$ and any locally $\mathcal{C}^{1}$ path $\gamma : [a, b] \longrightarrow \mathcal{D}$ with $\gamma(a) = y_0, \gamma(b) = y$, the derivative of the function $g_2 \circ \gamma : [a, b] \longrightarrow \mathbb{R}^2$ at $s \in [a,b]$ is clearly $e^{g_1(\gamma(s))} \pi_2 ( M. (J_2 f (\gamma(s), y_0) )^{-1}. \gamma'(s))$, so \eqref{lemmepoincarenonclassique2} follows. 

We have seen that $g_1$ and $g_2$ are well-defined and $\mathcal{C}^{k}$, so $g := (g_1, g_2)$ is also well-defined and $\mathcal{C}^{k}$. We now justify that $g$ is a group homomorphism from $(\mathcal{D}, \star)$ to $(\mathbb{R}^2, T)$. Let $y, x \in \mathcal{D}$ and choose a locally $\mathcal{C}^{1}$ path $\gamma : [0, 2] \longrightarrow \mathcal{D}$ with $\gamma(0) = y_0, \gamma(1) = y, \gamma(2) = f(y, x)$. Using two times \eqref{lemmepoincarenonclassique1}, we have 
\begin{align}
g_1(f(y,x)) & = \int_0^2 \pi_1 \left ( M. [J_2 f (\gamma(s), y_0)]^{-1}. \gamma'(s) \right ) ds \nonumber \\
& = \int_0^1 \pi_1 \left ( M. [J_2 f (\gamma(s), y_0)]^{-1}. \gamma'(s) \right ) ds + \int_1^2 \pi_1 \left ( M. [J_2 f (\gamma(s), y_0)]^{-1}. \gamma'(s) \right ) ds \nonumber \\
& = g_1(y) + \int_1^2 \pi_1 \left ( M. [J_2 f (\gamma(s), y_0)]^{-1}. \gamma'(s) \right ) ds. \label{cutpath2}
\end{align}
Now, let $h_y$ denote the inverse function of $f(y, .)$ ($h_y := f(y^{-1}, .)$) and let us consider the locally $\mathcal{C}^{1}$ path $\tilde \gamma : [1, 2] \longrightarrow \mathcal{D}$ defined by $\forall s \in [1, 2], \tilde \gamma(s) := h_y (\gamma(s))$. In particular we have $\tilde \gamma(1) = y_0, \tilde \gamma(2) = x$ and 
\begin{eqnarray}
\forall s \in [1, 2], \ \gamma (s) = f(y, \tilde \gamma(s)), \ \gamma' (s) = J_2 f(y, \tilde \gamma(s)). \tilde \gamma'(s). \label{cheminimage}
\end{eqnarray}
As a consequence, the second term in the right hand side of \eqref{cutpath2} equals 
\[ \int_1^2 \pi_1 \left ( M. \left [ J_2 f (f(y, \tilde \gamma(s)), y_0) \right ]^{-1}. J_2 f(y, \tilde \gamma(s)). \tilde \gamma'(s) \right ) ds = \int_1^2 \pi_1 \left ( M. \left [ J_2 f (\tilde \gamma(s), y_0) \right ]^{-1}. \tilde \gamma'(s) \right ) ds = g_1(x), \]
where we have used \eqref{reljacobiennes} with $u=y, v =\tilde \gamma(s)$ and \eqref{lemmepoincarenonclassique1}. Putting into \eqref{cutpath2} we obtain 
\begin{eqnarray}
g_1(f(y,x)) = g_1(y) + g_1(x). \label{relgroup1erecomp}
\end{eqnarray}
Then, using two times \eqref{lemmepoincarenonclassique2}, we have 
\begin{align}
g_2(f(y,x)) & = \int_0^2 e^{g_1(\gamma(s))} \pi_2 \left ( M. [J_2 f (\gamma(s), y_0)]^{-1}. \gamma'(s) \right ) ds \nonumber \\
& = \int_0^1 e^{g_1(\gamma(s))} \pi_2 \left ( M. [J_2 f (\gamma(s), y_0)]^{-1}. \gamma'(s) \right ) ds \nonumber \\
& + \int_1^2 e^{g_1(\gamma(s))} \pi_2 \left ( M. [J_2 f (\gamma(s), y_0)]^{-1}. \gamma'(s) \right ) ds \nonumber \\
& = g_2(y) + \int_1^2 e^{g_1(\gamma(s))} \pi_2 \left ( M. [J_2 f (\gamma(s), y_0)]^{-1}. \gamma'(s) \right ) ds. \label{cutpath3}
\end{align}
Recall the definition of the locally $\mathcal{C}^{1}$ path $\tilde \gamma$ above, and that it satisfies \eqref{cheminimage}. Note that for all $s \in [1, 2]$ we have $g_1(\gamma(s)) = g_1(f(y, \tilde \gamma(s))) = g_1(y) + g_1(\tilde \gamma(s))$ thanks to \eqref{relgroup1erecomp}. The second term in the right hand side of \eqref{cutpath3} can thus be expressed in term of $\tilde \gamma$. It equals 
\begin{align*}
& e^{g_1(y)} \int_1^2 e^{g_1(\tilde \gamma(s))} \pi_2 \left ( M. \left [ J_2 f (f(y, \tilde \gamma(s)), y_0) \right ]^{-1}. J_2 f(y, \tilde \gamma(s)). \tilde \gamma'(s) \right ) ds \\
= & e^{g_1(y)} \int_1^2 e^{g_1(\tilde \gamma(s))} \pi_2 \left ( M. \left [ J_2 f (\tilde \gamma(s), y_0) \right ]^{-1}. \tilde \gamma'(s) \right ) ds = e^{g_1(y)} g_2(x), 
\end{align*}
where we have used \eqref{reljacobiennes} with $u=y, v =\tilde \gamma(s)$ and \eqref{lemmepoincarenonclassique2}. Putting into \eqref{cutpath3} we obtain 
\begin{eqnarray}
g_2(f(y,x)) = g_2(y) + e^{g_1(y)} g_2(x). \label{relgroup2erecomp}
\end{eqnarray}
The combination of \eqref{relgroup1erecomp} and \eqref{relgroup2erecomp} yields $g(f(y,x)) = g(y) T g(x)$ so $g$ is indeed a group homomorphism from $(\mathcal{D}, \star)$ to $(\mathbb{R}^2, T)$. 

We now justify that $g$ is onto $\mathbb{R}^2$. Since $e^{g(y_0)}=1$, we have $Jg(y_0) = M. (J_2 f (y_0, y_0)^{-1}) = M$ which is invertible so, by the \textit{local inversion theorem}, $g(\mathcal{D})$ contains a neighborhood  in $\mathbb{R}^2$ of $g(y_0) = (0,0)$. Let us justify that any neighborhood $\mathcal{V}$ of $(0,0)$ generates the group $(\mathbb{R}^2,T)$. For $y = \left( \begin{array}{c}
y_1 \\
y_2 
\end{array} \right) \in \mathbb{R}^2$ let us define 
\[ z_n := \left( \begin{array}{c}
y_1 /n \\
y_2 / \left ( \sum_{k = 0}^{n-1} e^{k y_1/n} \right )
\end{array} \right). \]
Then, we see that for any $n \geq 1$, $y$ is equal to the $T$-product of $n$ times the element $z_n$: $y = z_n T ... T z_n$. Moreover, we can see that $z_n \in \mathcal{V}$ when $n$ is large enough. Therefore $\mathcal{V}$ generates the group $(\mathbb{R}^2,T)$. Then, since $g(\mathcal{D})$ is a subgroup of $(\mathbb{R}^2,T)$ that contains a neighborhood of $0$, we conclude that $g(\mathcal{D})=\mathbb{R}^2$. 

To prove that $g$ is injective, we can readily repeat the argument that justified the injectiveness in the proof of Lemma \ref{isomorphismecasclassique}. 

Combining what is proved above, we obtain that $g$ is bijective from $\mathcal{D}$ to $\mathbb{R}^2$ and $\mathcal{C}^{k}$. Recall that $g$ has been defined in such a way that for all $y \in \mathcal{D}$, $Jg(y)$, the Jacobian of $g$ at $y$, equals $\begin{pmatrix}
1 & 0 \\
0 & e^{g_1(y)}
\end{pmatrix}. M. (J_2 f (y, y_0))^{-1}$ where $M$ is invertible and $J_2 f (y, y_0)$ is invertible for any $y \in \mathcal{D}$, as mentioned in the beginning of the proof. Therefore, for all $y \in \mathcal{D}$ the Jacobian of $g$ at $y$ is invertible so we can conclude that $g$ is a $\mathcal{C}^{k}$-diffeomorphism from $\mathcal{D}$ to $\mathbb{R}^2$. Since $g$ is also a group homomorphism from $(\mathcal{D}, \star)$ to $(\mathbb{R}^2, T)$, it is even a $\mathcal{C}^{k}$-Lie group isomorphism from $(\mathcal{D}, \star)$ to $(\mathbb{R}^2, T)$, which concludes the proof. 

\end{proof}

We will also need the following simple lemma about the group $(\mathbb{R}^2, T)$. 

\begin{lemme} \label{shapeofmorphismes}

Let $m : \mathbb{R}^2 \longrightarrow \mathbb{R}$ be a continuous group homomorphism from $(\mathbb{R}^2, T)$ to $(\mathbb{R}, +)$. Then there exists $\beta \in \mathbb{R}$ such that 
\[ \forall z=(z_1, z_2) \in \mathbb{R}^2, \ m(z) = \beta z_1. \]

\end{lemme}

A consequence of the above lemma is that, if $\tilde m$ is a continuous group homomorphism from $(\mathbb{R}^2, T)$ to $(\mathbb{R}_+^*, \times)$, then there exists $\beta \in \mathbb{R}$ such that $\forall z=(z_1, z_2) \in \mathbb{R}^2, \ \tilde m(z) = e^{\beta z_1}$. 

\begin{proof} of Lemma \ref{shapeofmorphismes}

The proof is very basic. Note that $(\lambda \longmapsto (\lambda,0))$ and $(\lambda \longmapsto (0,\lambda))$ are continuous group homomorphisms from $(\mathbb{R}, +)$ to $(\mathbb{R}^2, T)$. Therefore, if we define $m_1, m_2 : \mathbb{R} \longrightarrow \mathbb{R}$ by 
\[ m_1(\lambda) := m(\lambda,0) \ \ \ \text{and} \ \ \ m_2(\lambda) := m(0,\lambda), \]
then $m_1$ and $m_2$ are continuous group homomorphisms from $(\mathbb{R}, +)$ to $(\mathbb{R}, +)$, so there exist $\beta$ and $\beta'$ such that $\forall \lambda \in \mathbb{R}, m_1(\lambda) = \beta \lambda, m_2(\lambda) = \beta' \lambda$. Then, note that for any $z=(z_1, z_2) \in \mathbb{R}^2$, 
\begin{align}
m(z) & = m \left ( \left( \begin{array}{c}
0 \\
z_2 
\end{array} \right) T \left( \begin{array}{c}
z_1 \\
0 
\end{array} \right) \right ) = m (0, z_2) + m (z_1, 0) \nonumber \\
& = m_2(z_2) + m_1(z_1) = \beta z_1 + \beta' z_2. \label{exprm}
\end{align}
\eqref{exprm} yields $m(1,1) = \beta + \beta'$. On the other hand we have 
\begin{align*}
m(1,1) & = m \left ( \left( \begin{array}{c}
0 \\
1 
\end{array} \right) T \left( \begin{array}{c}
1 \\
0 
\end{array} \right) \right ) = m (0, 1) + m (1, 0) \\
& = m (1, 0) + m (0, 1) = m \left ( \left( \begin{array}{c}
1 \\
0 
\end{array} \right) T \left( \begin{array}{c}
0 \\
1 
\end{array} \right) \right ) \\
& = m(1, e^1) = \beta + \beta' e^1, 
\end{align*}
where we have used \eqref{exprm} for the last equality. Putting together the above and $m(1,1) = \beta + \beta'$ we deduce that $\beta' = 0$ and, subsituing $0$ to $\beta'$ in \eqref {exprm}, we obtain the asserted result. 
\end{proof}

\section{L\'evy processes on $(\mathbb{R}^2, T)$} \label{levysurlegroupnoncomdim2}

The aim of this section is to identify L\'evy processes on the group $(\mathbb{R}^2, T)$, and in particular to express these L\'evy processes in term of exponential functionals of classical L\'evy processes on $(\mathbb{R}^2, +)$. Recall that the starting point of a L\'evy process on a group is always the neutral element of the group. Our result is the following. 
\begin{prop} \label{levynonclassiques}

Let $Y$ be a L\'evy process on $(\mathbb{R}^2, T)$, then there exists a L\'evy process $(\xi, \eta)$ on $(\mathbb{R}^2, +)$ such that 
\begin{eqnarray}
\forall t \geq 0, \ Y(t) = 
\left( \begin{array}{c}
\xi(t) \\
\int_0^t e^{\xi(s-)} d \eta(s) 
\end{array} \right). \label{relfonctexpo}
\end{eqnarray}
\end{prop}

Let us mention that the general form of generators of L\'evy processes on the group $(\mathbb{R}^2, T)$ can be obtained by an application of Hunt's theorem (see Theorem 5.3.3 in \cite{applebaum_lie}) to the Lie group $(\mathbb{R}^2, T)$. 
It is possible to prove Proposition \ref{levynonclassiques} by 1) computing the generator of the right-hand-side of \eqref{relfonctexpo} thanks to Ito's Lemma for L\'evy processes, 2) expliciting Hunt's theorem in the case of the Lie group $(\mathbb{R}^2, T)$, and 3) identifying the expressions of the two generators. However, we have chosen to proceed in a more intuitive way that does not require an application of Hunt's theorem and relies on simple computations on stochastic integrals. If $Y$ is a L\'evy process on $(\mathbb{R}^2, T)$, the idea is to define 
\[ \forall t \geq 0, \ \xi(t) := \pi_1(Y(t)), \ \ \ \eta(t) := \int_0^t e^{-\pi_1(Y(s-))} d\pi_2(Y(s)), \]
and then to justify that $(\xi, \eta)$ is a L\'evy process on $(\mathbb{R}^2, +)$ and that \eqref{relfonctexpo} holds. This argument 
is very simple modulo the fact that we can write stochastic integrals with respect to the real-valued process $\pi_2(Y)$. Since we do not know \textit{a priori} what this process looks like we first justify the fact that these integrals are well-defined. We proceed in several steps. We first study the procedure of removing the big jumps for a L\'evy process on the group $(\mathbb{R}^2, T)$. This procedure is classical but we provide all the details for the sake of clarity. 

For any $M > 0$ let 
\begin{eqnarray}
A_M := \{ x \in \mathbb{R}^2, \ |x|_{\infty} \geq M \}, \label{ensgdsauts}
\end{eqnarray}
where $|(x_1, x_2)|_{\infty} := |x_1| \vee |x_2|$. For $Y$ a L\'evy process on $(\mathbb{R}^2, T)$ and $t \geq 0$, the jump of $Y$ at $t$ is $\Delta Y(t) := (Y(t-))^{-1} T Y(t)$ (note that $\Delta Y(t) = (0,0)$ if $Y$ is continuous at $t$). Let $R^M[Y]$ be the process obtained by removing the jumps of $Y$ that are in $A_M$. Let us describe briefly the construction of $R^M[Y]$. If $Y$ has no jump in $A_M$ then we put $R^M[Y] := Y$. If $Y$ does have jumps in $A_M$ then, since $Y$ is c\`ad-l\`ag, there are almost surely finitely many jumps of $Y$ in $A_M$ on finite time intervals. Let $0 < t_1 < t_2 < ...$ be the discrete sequence of times associated to the jumps in $A_M$, we then put $R^M[Y](t) := R^M[Y](t_i -) T (Y(t_i))^{-1} T Y(t)$ for $t \in [t_i, t_{i+1}[$, where by convention $t_0 := 0$ and $R^M[Y](0 -) := (0,0)$. 
It will be convenient to work with $R^M[Y]$, which has only bounded jumps, and to notice that $R^M[Y]$ coincides with $Y$ until the instant of the first jump of $Y$ in $A_M$. 
We have the following lemma about $R^M[Y]$: 

\begin{lemme} \label{removingbigjumps}

Let $Y$ be a L\'evy process on $(\mathbb{R}^2, T)$ and let $(\mathcal{F}_t)_{t \geq 0}$ be the right-continuous filtration associated with $Y$. For any $M > 0$, $R^M[Y]$ is a L\'evy process on $(\mathbb{R}^2, T)$ adapted with respect to $(\mathcal{F}_t)_{t \geq 0}$. For any $t \geq 0$ we have 
\begin{eqnarray}
\sup_{s \in [0, t]} \mathbb{E} \left [ e^{2 |\pi_1(R^M[Y](s-))|} \right ] < +\infty, \label{momentpi1}
\end{eqnarray}
and 
\begin{eqnarray}
\sup_{s \in [0, t]} \mathbb{E} \left [ |\pi_2(R^M[Y](s))|^2 \right ] < +\infty. \label{momentpi2}
\end{eqnarray}

\end{lemme}

\begin{proof}

Let $Y$ be a L\'evy process on $(\mathbb{R}^2, T)$ and let $(\mathcal{F}_t)_{t \geq 0}$ be the right-continuous filtration associated with $Y$. The fact that $R^M[Y]$ is adapted with respect to $(\mathcal{F}_t)_{t \geq 0}$ is plain from the fact that $(R^M[Y](s), \ 0 \leq s \leq t)$ has been constructed as a function of $(Y(s), \ 0 \leq s \leq t)$. We first justify that $R^M[Y]$ is a L\'evy process on $(\mathbb{R}^2, T)$. If $Y$ has a.s. no jumps in $A_M$ then $Y = R^M[Y]$ so the latter is clearly a L\'evy process on $(\mathbb{R}^2, T)$. We thus assume that $Y$ admits jumps in $A_M$. 

By construction, $R^M[Y]$ is a c\`ad-l\`ag process starting from $(0,0)$. We now justify that $R^M[Y]$ has stationary and independent increments. Let us fix $t \geq 0$ and define $S_t [Y] (.) := (Y(t))^{-1} T Y(t+.)$. Then $S_t [Y]$ is independent from $\mathcal{F}_t$ and has the same law as $Y$. For any $s \geq 0$ we have 
\begin{align*}
\Delta Y(t+s) & = (Y(t+s-))^{-1} T Y(t+s) = ((Y(t))^{-1} T Y(t+s-))^{-1} T (Y(t))^{-1} T Y(t+s) \\
& = (S_t [Y](s-))^{-1} T S_t [Y](s) = \Delta S_t [Y](s). 
\end{align*}
As a consequence, $S_t [Y]$ has a jump at $s$ if and only if $Y$ has a jump at $t + s$. Let $0 < s_1 < s_2 < ...$ be the sequence of the jumps of $S_t [Y]$ in $A_M$ (then $ t + s_1 < t + s_2 < ...$ is the sequence of the jumps of $Y$ in $A_M$ after time $t$). If $Y$ has jumps in $A_M$ on $[0, t]$, then let $u$ be the instant of the last jump of $Y$ in $A_M$. If $Y$ has no jump in $[0, t]$, then let $u := 0$. Let us prove by induction on $i$ that 
\begin{eqnarray}
R^M[S_t[Y]](s) = S_t[R^M[Y]](s), \ \forall s \in [s_i, s_{i+1}[, \label{commutationSR}
\end{eqnarray}
where by convention $s_0 := 0$. For $s \in [0, s_1[$ we have by the definitions of $R^M$, $s_1$, and $S_t$: 
\begin{align*}
R^M[S_t[Y]](s) & = S_t[Y](s) = (Y(t))^{-1} T Y(t+s) \\ 
& = (R^M[Y](u -) T (Y(u))^{-1} T Y(t))^{-1} T R^M[Y](u -) T (Y(u))^{-1} T Y(t+s) \\
& = (R^M[Y](t))^{-1} T R^M[Y](t+s) = S_t[ R^M[Y]](s). 
\end{align*}
Therefore, \eqref{commutationSR} holds for $i=0$. Let us assume that it holds for some $i \geq 0$ and prove it for $i+1$. By the induction hypothesis we have that $R^M[S_t[Y]](s_{i+1} -) = S_t[R^M[Y]](s_{i+1} -)$. Then, for any $s \in [s_{i+1}, s_{i+2}[$, using the definition of $R^M$, the previous relation and the definition of $S_t$ we have 
\begin{align*}
R^M[S_t[Y]](s) & = R^M[S_t[Y]](s_{i+1} -) T (S_t[Y](s_{i+1}))^{-1} T S_t[Y](s) \\
& = S_t[R^M[Y]](s_{i+1} -) T (S_t[Y](s_{i+1}))^{-1} T S_t[Y](s) \\ 
& = (R^M[Y](t))^{-1} T R^M[Y](t + s_{i+1} -) T ((Y(t))^{-1} T Y(t+s_{i+1}))^{-1} T (Y(t))^{-1} T Y(t+s) \\ 
& = (R^M[Y](t))^{-1} T R^M[Y](t + s_{i+1} -) T (Y(t+s_{i+1}))^{-1} T Y(t+s) \\ 
& = (R^M[Y](t))^{-1} T R^M[Y](t+s) = S_t[ R^M[Y]](s). 
\end{align*}
Then the induction is completed and we conclude that we have a.s. $R^M[S_t[Y]] = S_t[R^M[Y]]$. This means that $S_t[R^M[Y]]$ is expressed as a deterministic function, $R^M$, of the process $S_t[Y]$ that is independent from $\mathcal{F}_t$ and has the same law as $Y$. As a consequence $S_t[R^M[Y]]$ is independent from $\mathcal{F}_t$ and has the same law as $R^M[Y]$. This proves the independence and stationarity of the increments of $R^M[Y]$ which is therefore a L\'evy process. 

We now justify the claim about moments. For this we follow the proof of Theorem 2.4.7 in \cite{applebaum_2009}, just working on $(\mathbb{R}^2, T)$ instead of $(\mathbb{R}^d, +)$. Let $T_0 := 0$ and recursively for $n \geq 0$, $T_{n+1} := \inf \{ t > T_n, |(R^M[Y](T_n))^{-1} T R^M[Y](t)|_{\infty} > M \}$. By definition of $T_n$ we have a.s. that 
\begin{align}
\forall s \in [T_n, T_{n+1}[, \ & |\pi_1 ((R^M[Y](T_n))^{-1} T R^M[Y](s))| \leq M, \label{trajborne1} \\
& |\pi_2 ((R^M[Y](T_n))^{-1} T R^M[Y](s))| \leq M. \label{trajborne2}
\end{align}
Then, $|\pi_1 ((R^M[Y](T_n))^{-1} T R^M[Y](T_{n+1}))|$ equals a.s. 
\begin{align*}
& |\pi_1 ((R^M[Y](T_n))^{-1} T R^M[Y](T_{n+1}-) T \Delta R^M[Y](T_{n+1}))| \\
= & |\pi_1 ((R^M[Y](T_n))^{-1} T R^M[Y](T_{n+1}-)) + \pi_1 (\Delta R^M[Y](T_{n+1}))| \\
\leq & |\pi_1 ((R^M[Y](T_n))^{-1} T R^M[Y](T_{n+1}-))| + |\pi_1 (\Delta R^M[Y](T_{n+1}))|. 
\end{align*}
According to \eqref{trajborne1} the first term is a.s. less than $M$ and, since $R^M[Y]$ has no jump in $A_M$, the second term is a.s. less than $M$. We thus get $|\pi_1 ((R^M[Y](T_n))^{-1} T R^M[Y](T_{n+1}))| \leq 2M$ a.s. and combining with \eqref{trajborne1} we get that a.s. 
\begin{eqnarray}
\sup_{s \in [T_n, T_{n+1}]} | \pi_1 ((R^M[Y](T_n))^{-1} T R^M[Y](s))| \leq 2M. \label{trajborne3}
\end{eqnarray}
Then, $|\pi_2 ((R^M[Y](T_n))^{-1} T R^M[Y](T_{n+1}))|$ a.s. equals 
\begin{align*}
& |\pi_2 ((R^M[Y](T_n))^{-1} T R^M[Y](T_{n+1}-) T \Delta R^M[Y](T_{n+1}))| \\
= & |\pi_2 ((R^M[Y](T_n))^{-1} T R^M[Y](T_{n+1}-)) + e^{\pi_1 ((R^M[Y](T_n))^{-1} T R^M[Y](T_{n+1}-))} \pi_2 (\Delta R^M[Y](T_{n+1}))| \\
\leq & |\pi_2 ((R^M[Y](T_n))^{-1} T R^M[Y](T_{n+1}-))| + e^{\pi_1 ((R^M[Y](T_n))^{-1} T R^M[Y](T_{n+1}-))} |\pi_2 (\Delta R^M[Y](T_{n+1}))|. 
\end{align*}
According to \eqref{trajborne2} the first term is a.s. less than $M$, and according to \eqref{trajborne1} the first factor in the second term is a.s. less than $e^M$. Since $R^M[Y]$ has no jump in $A_M$, the second factor in the second term is a.s. less than $M$. We thus get that $|\pi_2 ((R^M[Y](T_n))^{-1} T R^M[Y](T_{n+1}))| \leq (1+e^M)M \leq 2 M e^M$ a.s. and, combining with \eqref{trajborne2} we get that a.s. 
\begin{eqnarray}
\sup_{s \in [T_n, T_{n+1}]} | \pi_2 ((R^M[Y](T_n))^{-1} T R^M[Y](s))| \leq 2 M e^M. \label{trajborne4}
\end{eqnarray}

Let us prove by induction that a.s. 
\begin{eqnarray}
\forall s \in [0, T_n], \ |\pi_1(R^M[Y](s))| \leq 2 n M, \ |\pi_2(R^M[Y](s))| \leq 2nM e^{2nM}. \label{trajborne5}
\end{eqnarray}
\eqref{trajborne5} is clearly true for $n=0$. Let us assume that it holds for some $n \geq 0$ and prove it for $n+1$. If the supremum of $|\pi_1(R^M[Y](s))|$ on the interval $[0, T_{n+1}]$ is attained in $[0, T_{n}[$, then the induction hypothesis implies that $\sup_{s \in [0, T_{n+1}]} |\pi_1(R^M[Y](s))| \leq 2 n M \leq 2 (n+1) M$. If the supremum of $|\pi_1(R^M[Y](s))|$ on the interval $[0, T_{n+1}]$ is attained in $[T_{n}, T_{n+1}]$ then we have 
\begin{align*}
\sup_{s \in [0, T_{n+1}]} |\pi_1(R^M[Y](s))| & = \sup_{s \in [T_n, T_{n+1}]} |\pi_1(R^M[Y](T_n) T (R^M[Y](T_n))^{-1} T R^M[Y](s))| \\
& = \sup_{s \in [T_n, T_{n+1}]} |\pi_1(R^M[Y](T_n)) + \pi_1 ((R^M[Y](T_n))^{-1} T R^M[Y](s))| \\
& \leq |\pi_1(R^M[Y](T_n))| + \sup_{s \in [T_n, T_{n+1}]} | \pi_1 ((R^M[Y](T_n))^{-1} T R^M[Y](s))|. 
\end{align*}
By the induction hypothesis the first term is a.s. less than $2nM$ and by \eqref{trajborne3} the second term is a.s. less than $2M$. In any case we thus have $\sup_{s \in [0, T_{n+1}]} |\pi_1(R^M[Y](s))| \leq 2 (n+1) M$. Then, if the supremum of $|\pi_2(R^M[Y](s))|$ on the interval $[0, T_{n+1}]$ is attained in $[0, T_{n}[$, then the induction hypothesis implies that $\sup_{s \in [0, T_{n+1}]} |\pi_2(R^M[Y](s))| \leq 2nM e^{2nM} \leq 2(n+1)M e^{2(n+1)M}$. If the supremum of $|\pi_2(R^M[Y](s))|$ on the interval $[0, T_{n+1}]$ is attained in $[T_{n}, T_{n+1}]$ then we have 
\begin{align*}
\sup_{s \in [0, T_{n+1}]} |\pi_2(R^M[Y](s))| & = \sup_{s \in [T_n, T_{n+1}]} |\pi_2(R^M[Y](T_n) T (R^M[Y](T_n))^{-1} T R^M[Y](s))| \\
& = \sup_{s \in [T_n, T_{n+1}]} |\pi_2(R^M[Y](T_n)) + e^{\pi_1(R^M[Y](T_n))} \pi_2 ((R^M[Y](T_n))^{-1} T R^M[Y](s))| \\
& \leq |\pi_2(R^M[Y](T_n))| + e^{\pi_1(R^M[Y](T_n))} \sup_{s \in [T_n, T_{n+1}]} | \pi_2 ((R^M[Y](T_n))^{-1} T R^M[Y](s))|. 
\end{align*}
By the induction hypothesis the first term is a.s. less than $2nM e^{2nM} \leq 2nM e^{2(n+1)M}$ and the first factor in the second term is a.s. less than $e^{2nM}$. By \eqref{trajborne4} the second factor of the second term is a.s. less than $2Me^M \leq 2Me^{2M}$. In any case we thus have $\sup_{s \in [0, T_{n+1}]} |\pi_2(R^M[Y](s))| \leq 2(n+1)M e^{2(n+1)M}$. In conclusion, \eqref{trajborne5} holds for $n+1$ so the induction is completed. 

Then, since $R^M[Y]$ is a L\'evy process, the sequence $(T_{n+1} - T_n)_{n \geq 0}$ is \textit{iid} so in particular $\mathbb{E}[e^{-\lambda T_n}] = \mathbb{E}[e^{-\lambda (T_n - T_{n-1})}...e^{-\lambda T_{1}}] = (\mathbb{E}[e^{-\lambda T_1}])^{n}$. Then for any $t \geq 0, z > 0$, 
\begin{align*}
\mathbb{P} \left ( |\pi_1(R^M[Y](t))| > z \right ) & \leq \mathbb{P} \left ( |\pi_1(R^M[Y](t))| > 2M \lfloor z/2M \rfloor \right ) \\
& \leq \mathbb{P} \left ( T_{\lfloor z/2M \rfloor} < t \right ) \leq e^{\lambda t} \left ( \mathbb{E} \left [ e^{-\lambda T_{1}} \right ] \right )^{\lfloor z/2M \rfloor}, 
\end{align*}
where we have used \eqref{trajborne5} (which shows that $|\pi_1(R^M[Y](t))| > 2nM \Rightarrow T_n < t$) and Chernoff's inequality. Let us choose $\lambda_{M,1}$ large enough so that $(\mathbb{E}[e^{-\lambda_{M,1} T_1}])^{1/2M} < e^{-2}$. For such a choice of $\lambda_{M,1}$ we get 
\begin{align*}
\forall t \geq 0, \ \mathbb{E} \left [ e^{2 |\pi_1(R^M[Y](t))|} \right ] & = 1 + 2 \int_0^{+\infty} e^{2 z} \mathbb{P} \left ( |\pi_1(R^M[Y](t))| > z \right ) dz \\
& \leq 1 + 2 e^{\lambda_{M,1} t} \int_0^{+\infty} e^{2 z} \left ( \mathbb{E} \left [ e^{-\lambda_{M,1} T_{1}} \right ] \right )^{\lfloor z/2M \rfloor} dz \\
& \leq 1 + e^{\lambda_{M,1} t} C_1, 
\end{align*}
where $C_1$ is some positive constant. We thus get 
\[ \sup_{s \in [0, t]} \mathbb{E} \left [ e^{2 |\pi_1(R^M[Y](s))|} \right ] < +\infty. \]
Taking left-limits and using Fatou's Lemma we obtain \eqref{momentpi1}. 

Let $z_M$ be such that $e^{3M \lfloor \log(z)/3M \rfloor} > 2M \lfloor \log(z)/3M \rfloor e^{2M \lfloor \log(z)/3M \rfloor}$ for any $z > z_M$. Then for any $t \geq 0$ and $z > z_M$, 
\begin{align*}
\mathbb{P} \left ( |\pi_2(R^M[Y](t))| > z \right ) & \leq \mathbb{P} \left ( |\pi_2(R^M[Y](t))| > e^{3M \lfloor \log(z)/3M \rfloor} \right ) \\
& \leq \mathbb{P} \left ( |\pi_2(R^M[Y](t))| > 2M \lfloor \log(z)/3M \rfloor e^{2M \lfloor \log(z)/3M \rfloor} \right ) \\
& \leq \mathbb{P} \left ( T_{\lfloor \log(z)/3M \rfloor} < t \right ) \leq e^{\lambda_{M,1} t} \left ( \mathbb{E} \left [ e^{-\lambda_{M,1} T_{1}} \right ] \right )^{\lfloor \log(z)/3M \rfloor}, 
\end{align*}
where we have used \eqref{trajborne5} (which shows that $|\pi_2(R^M[Y](t))| > 2nMe^{2nM} \Rightarrow T_n < t$) and Chernoff's inequality. Let us choose $\lambda_{M,2}$ large enough so that $(\mathbb{E}[e^{-\lambda_{M,2} T_1}])^{1/3M} < e^{-3}$. For such a choice of $\lambda_{M,2}$ we have $( \mathbb{E} [ e^{-\lambda_{M,2} T_{1}} ] )^{\lfloor \log(z)/3M \rfloor} < ( \mathbb{E} [ e^{-\lambda_{M,2} T_{1}} ] )^{-1} z^{-3}$. We thus get: 
\begin{align*}
\forall t \geq 0, \ \mathbb{E} \left [ |\pi_2(R^M[Y](t))|^2 \right ] & = 2 \int_0^{+\infty} z \mathbb{P} \left ( |\pi_2(R^M[Y](t))| > z \right ) dz \\
& \leq 2 e^{\lambda_{M,2} t} \int_0^{+\infty} z \left ( \mathbb{E} \left [ e^{-\lambda_{M,2} T_{1}} \right ] \right )^{\lfloor \log(z)/3M \rfloor} dz \\
& \leq e^{\lambda_{M,2} t} C_2, 
\end{align*}
where $C_2$ is some positive constant. This yields \eqref{momentpi2}. 



\end{proof}

We now study a procedure of re-centering for the second coordinate of $R^M[Y]$, in order to obtain a martingale. 

\begin{lemme} \label{centering}

Let $Y$ be a L\'evy process on $(\mathbb{R}^2, T)$ and let $(\mathcal{F}_t)_{t \geq 0}$ be the right-continuous filtration associated with $Y$. For any $M > 0$, there is a unique $\alpha_M \in \mathbb{R}$ such that the process 
\begin{eqnarray}
W^M[Y] := \left ( \pi_2(R^M[Y](t)) - \alpha_M \int_0^t e^{\pi_1(R^M[Y](u-))} du, \ t \geq 0 \right ) \label{centering1}
\end{eqnarray}
is a c\`ad-l\`ag martingale with respect to the filtration $(\mathcal{F}_t)_{t \geq 0}$, locally bounded in $L^2$ (that is, for any $t >0$, $\sup_{u \in [0,t]} \mathbb{E} [ (W^M[Y](s))^2 ] < +\infty$). 


\end{lemme}

\begin{proof}

Let $Y$ be a L\'evy process on $(\mathbb{R}^2, T)$ and let $(\mathcal{F}_t)_{t \geq 0}$ be the right-continuous filtration associated with $Y$. The process $(\int_0^t e^{\pi_1(R^M[Y](u-))} du, \ t \geq 0)$ is well-defined and a.s. continuous, because $R^M[Y]$ is a.s. c\`ad-l\`ag as a L\'evy process, according to Lemma \ref{removingbigjumps}. Note from Lemma \ref{removingbigjumps} that the expectations $\mathbb{E} [ \pi_2(R^M[Y](1)) ]$ and $\mathbb{E} [ \int_0^1 e^{\pi_1(R^M[Y](u-))} du ]$ are finite, and notice that the second is positive. Let us put 
\begin{eqnarray}
\alpha_M := \mathbb{E} \left [ \pi_2(R^M[Y](1)) \right ] / \mathbb{E} \left [ \int_0^1 e^{\pi_1(R^M[Y](u-))} du \right ], \label{centering2}
\end{eqnarray}
and define $W^M[Y]$ with this $\alpha_M$, as in \eqref{centering1}. Then, $W^M[Y]$ is well-defined and a.s. c\`ad-l\`ag. 

We now justify the local boundedness in $L^2$. Let us fix $t>0$ and $s \in [0, t]$. We have 
\begin{align}
\mathbb{E} \left [ (W^M[Y](s))^2 \right ] & \leq 2 \mathbb{E} \left [ (\pi_2(R^M[Y](s)))^2 \right ] + 2 \alpha_M^2 \mathbb{E} \left [ \left ( \int_0^s e^{\pi_1(R^M[Y](u-))} du \right )^2 \right ] \nonumber \\
& \leq 2 \sup_{u \in [0, t]} \mathbb{E} \left [ |\pi_2(R^M[Y](u))|^2 \right ] + 2 \alpha_M^2 t^2 \mathbb{E} \left [ \left ( \frac1{t} \int_0^t e^{\pi_1(R^M[Y](u-))} du \right )^2 \right ] \nonumber \\
& \leq 2 \sup_{u \in [0, t]} \mathbb{E} \left [ |\pi_2(R^M[Y](u))|^2 \right ] + 2 \alpha_M^2 t \mathbb{E} \left [ \int_0^t e^{2 \pi_1(R^M[Y](u-))} du \right ] \nonumber \\
& \leq 2 \sup_{u \in [0, t]} \mathbb{E} \left [ |\pi_2(R^M[Y](u))|^2 \right ] + 2 \alpha_M^2 t \int_0^t \mathbb{E} \left [ e^{2 \pi_1(R^M[Y](u-))} \right ] du \nonumber \\
& \leq 2 \sup_{u \in [0, t]} \mathbb{E} \left [ |\pi_2(R^M[Y](u))|^2 \right ] + 2 \alpha_M^2 t^2 \sup_{u \in [0, t]} \mathbb{E} \left [ e^{2 |\pi_1(R^M[Y](u-))|} \right ] \nonumber \\
& < +\infty, \label{boundedl2}
\end{align}
where we have used Jensen's inequality, Fubini's theorem, and Lemma \ref{removingbigjumps} for the finiteness of the two supremums. Since the final upper bound is independent of $s \in [0,t]$, we obtain that $W^M[Y]$ is indeed locally bounded in $L^2$. 

We now justify that $W^M[Y]$ is a martingale. Recall from Lemma \ref{removingbigjumps} that $R^M[Y]$ is adapted with respect to $(\mathcal{F}_t)_{t \geq 0}$. From this it is not difficult to see that $W^M[Y]$ is adapted with respect to $(\mathcal{F}_t)_{t \geq 0}$. Now, for any $h > 0$ let us put 
\[ \alpha_M(h) := \mathbb{E} \left [ \pi_2(R^M[Y](h)) \right ] / \mathbb{E} \left [ \int_0^h e^{\pi_1(R^M[Y](u-))} du \right ], \]
and notice that $\alpha_M$ defined in \eqref{centering2} coincides with $\alpha_M(1)$. By definition, $\alpha_M(h)$ is the unique real number that satisfies 
\[ \mathbb{E} \left [ \pi_2(R^M[Y](h)) - \alpha_M(h) \int_0^h e^{\pi_1(R^M[Y](u-))} du \right ] = 0. \]
Then recall that, since $R^M[Y]$ is a L\'evy process on $(\mathbb{R}^2, T)$ adapted to $(\mathcal{F}_t)_{t\geq0}$, for any $h>0$ and $u > 0$ we have $R^M[Y](h+u) = R^M[Y](h) T S_h[R^M[Y]](u)$ where $S_h[R^M[Y]]$ is equal in law to $R^M[Y]$ and is independent from $\mathcal{F}_h$. Then $\mathbb{E} [ \pi_2(R^M[Y](2h)) - \alpha_M(h) \int_0^{2h} e^{\pi_1(R^M[Y](u-))} du ]$ equals 
\begin{align}
& \mathbb{E} \left [ \pi_2 \left (R^M[Y](h) T S_h[R^M[Y]](h) \right ) - \alpha_M(h) \int_0^h e^{\pi_1(R^M[Y](u-))} du \right. \nonumber \\
- & \left. \alpha_M(h) \int_0^{h} e^{\pi_1(R^M[Y](h) T S_h[R^M[Y]](u-))} du \right ] \nonumber \\
= & \mathbb{E} \left [ \pi_2(R^M[Y](h)) - \alpha_M(h) \int_0^h e^{\pi_1(R^M[Y](u-))} du \right ] \nonumber \\
+ & \mathbb{E} \left [ e^{\pi_1(R^M[Y](h))} \left ( \pi_2(S_h[R^M[Y]](h)) - \alpha_M(h) \int_0^{h} e^{\pi_1(S_h[R^M[Y]](u-))} du \right ) \right ] \nonumber \\
= & 0 + \mathbb{E} \left [ e^{\pi_1(R^M[Y](h))} \mathbb{E} \left [ \pi_2(S_h[R^M[Y]](h)) - \alpha_M(h) \int_0^{h} e^{\pi_1(S_h[R^M[Y]](u-))} du \big | \mathcal{F}_h \right ] \right ] \nonumber \\
= & \mathbb{E} \left [ e^{\pi_1(R^M[Y](h))} \times \mathbb{E} \left [ \pi_2(R^M[Y](h)) - \alpha_M(h) \int_0^{h} e^{\pi_1(R^M[Y](u-))} du \right ] \right ] \nonumber \\
= & \mathbb{E} \left [ e^{\pi_1(R^M[Y](h))} \times 0 \right ] = 0. \label{propmart}
\end{align}
In the above we have used two times the definition of $\alpha_M(h)$ and one time the fact that $S_h[R^M[Y]]$ is equal in law to $R^M[Y]$ and is independent from $\mathcal{F}_h$. By unicity we obtain $\alpha_M(2h) = \alpha_M(h)$ and the previous reasoning can be used to prove by induction that we have actually $\alpha_M(nh) = \alpha_M(h)$, for all $n \in \mathbb{N}$. In particular, the function $\alpha_M(.)$ is constant on the set of positive dyadic numbers. Then, reasoning as in \eqref{boundedl2}, we see that for any $t>0$ the families $(\pi_2(R^M[Y](s)), 0 \leq s \leq t)$ and $(\int_0^{s} e^{\pi_1(R^M[Y](u-))} du, 0 \leq s \leq t)$ are bounded in $L^2$ and therefore uniformly integrable. Since $R^M[Y]$ is a L\'evy process, these two family are also stochastically continuous with respect to $s \in [0, t]$. We deduce that these family are continuous in $L^1$ with respect to $s \in [0, t]$. This proves the continuity of the function $\alpha_M(.)$. Combining with the constantness on the set of positive dyadic numbers, we obtain that this function is constant equal to $\alpha_M$. As a consequence, for any $h>0$ we have 
\[ \mathbb{E} \left [ \pi_2(R^M[Y](h)) - \alpha_M \int_0^h e^{\pi_1(R^M[Y](u-))} du \right ] = 0. \]
Using this and reasoning as in \eqref{propmart} we can now prove that for any $t,s >0$
\[ \mathbb{E} \left [ W^M[Y](t+s) \big | \mathcal{F}_t \right ] = W^M[Y](t). \]
This concludes the proof of the fact that $W^M[Y]$ is a martingale with respect to the filtration $(\mathcal{F}_t)_{t \geq 0}$. 

Conversely, the fact that $W^M[Y]$ is a martingale implies that $\mathbb{E} [ W^M[Y](1) ] = \mathbb{E} [ W^M[Y](0) ] = \mathbb{E} [ 0 ] = 0$. This can hold only for $\alpha_M$ being as in \eqref{centering2}. This proves unicity for the choice of $\alpha_M$ for which $W^M[Y]$ is a martingale. 




\end{proof}


We are now ready to prove Proposition \ref{levynonclassiques}. 

\begin{proof} of Proposition \ref{levynonclassiques}

Let $Y$ be a L\'evy process on $(\mathbb{R}^2, T)$ and let $(\mathcal{F}_t)_{t \geq 0}$ be the right-continuous filtration associated with $Y$. We want to define 
\begin{eqnarray}
\forall t \geq 0, \ \xi(t) := \pi_1(Y(t)), \ \ \ \eta(t) := \int_0^t e^{-\pi_1(Y(s-))} d\pi_2(Y(s)). \label{levynonclassiques1}
\end{eqnarray}

We first need to justify that the stochastic integral in \eqref{levynonclassiques1} is well-defined. Let us fix $M > 0$, and let $W^M[Y]$ denote the process defined in \eqref{centering1}. According to Lemma \ref{centering}, $W^M[Y]$ is a c\`ad-l\`ag martingale with respect to the filtration $(\mathcal{F}_t)_{t \geq 0}$, locally bounded in $L^2$. Also, $( \int_0^t e^{\pi_1(R^M[Y](u-))} du, \ t \geq 0 )$ is continuous, adapted with respect to $(\mathcal{F}_t)_{t \geq 0}$ and has locally bounded variation. Then, since by definition of $W^M[Y]$ in \eqref{centering1} we have $\pi_2(R^M[Y]) = W^M[Y] + \alpha_M \int_0^. e^{\pi_1(R^M[Y](u-))} du$, we deduce that the process $\pi_2(R^M[Y])$ (that is c\`ad-l\`ag and adapted with respect to $(\mathcal{F}_t)_{t \geq 0}$) satisfies the definition of a decomposable process, just before Theorem II.3.9 in \cite{protter}. By this theorem, $\pi_2(R^M[Y])$ is a semimartingale, in the meaning of \cite{protter}, for the filtration $(\mathcal{F}_t)_{t \geq 0}$. In the remainder, the concept of semimartingale has to be always understood in the meaning of \cite{protter}. 

Let $T_M := \inf \{ t \geq 0, \ \Delta Y(t) \in A_M \}$, where $A_M$ is defined in \eqref{ensgdsauts}. $T_M$ is a stopping time for the filtration $(\mathcal{F}_t)_{t \geq 0}$ and $T_M$ converges a.s. to $+\infty$ when $M$ goes to $+\infty$ (otherwise, with positive probability, a finite time interval could contain infinitely many big jumps of $Y$, which would not be compatible with $Y$ being a.s. c\`ad-l\`ag). For $Z$ a real process starting at $0$ and adapted to $(\mathcal{F}_t)_{t \geq 0}$, we define $Z$ stopped at $T_M-$ by $Z^{T_M-}(t) := Z(t) \mathds{1}_{0 \leq t < T_M} + Z(T_M-) \mathds{1}_{t \geq T_M}$. For any $M>0$ we have clearly $(\pi_2(R^M[Y]))^{T_M-} = (\pi_2(Y))^{T_M-}$. Moreover $\pi_2(Y)$ is c\`ad-l\`ag and adapted with respect to $(\mathcal{F}_t)_{t \geq 0}$. By Theorem II.3.6 in \cite{protter}, we can conclude that $\pi_2(Y)$ is a semimartingale for the filtration $(\mathcal{F}_t)_{t \geq 0}$. 

The process $( e^{-\pi_1(Y(t-))}, \ t \geq 0 )$ is adapted with respect to $(\mathcal{F}_t)_{t \geq 0}$ and left continuous with right limits. Section II.4 in \cite{protter} ensures that the stochastic integral of such a process, with respect to a semimartingale for $(\mathcal{F}_t)_{t \geq 0}$, is well-defined (as a limit of stochastic integrals of simple predictable processes with respect to the semimartingale in question), a.s. c\`ad-l\`ag, and adapted with respect to $(\mathcal{F}_t)_{t \geq 0}$. 

We deduce that the process $(\xi, \eta)$ defined in \eqref{levynonclassiques1} is indeed well-defined, a.s. c\`ad-l\`ag, and adapted with respect to $(\mathcal{F}_t)_{t \geq 0}$. 
We now justify that $(\xi, \eta)$ is a L\'evy process on $(\mathbb{R}^2, +)$. We only need to justify the independence and stationarity of the increments. Let us fix $t \geq 0$ and define $\tilde Y (.) := (Y(t))^{-1} T Y(t+.)$. Then $\tilde Y$ is independent from $\mathcal{F}_t$ and has the same law as $Y$. Since $Y(t+.) = Y(t) T \tilde Y(.)$, it is not difficult to see that we have a.s. 
\[ \forall s \geq 0, \ \tilde Y(s) = \left( \begin{array}{c}
\pi_1(Y(t+s)) - \pi_1(Y(t)) \\
e^{-\pi_1(Y(t))} [ \pi_2(Y(t+s)) - \pi_2(Y(t)) ]
\end{array} \right). \]
Let $(\tilde \xi, \tilde \eta)$ be constructed from $\tilde Y$ just as $(\xi, \eta)$ is constructed from $Y$ in \eqref{levynonclassiques1}. Then $(\tilde \xi, \tilde \eta)$ is independent from $\mathcal{F}_t$ and has the same law as $(\xi, \eta)$. Moreover, we have a.s. that for any $s \geq 0$, 
\begin{align*} 
\left( \begin{array}{c}
\tilde \xi(s) \\
\tilde \eta(s) 
\end{array} \right) 
& = 
\left( \begin{array}{c}
\pi_1(\tilde Y(s)) \\
\int_0^s e^{-\pi_1(\tilde Y(u-))} d\pi_2(\tilde Y(u)) 
\end{array} \right) 
= 
\left( \begin{array}{c}
\pi_1(Y(t+s)) - \pi_1(Y(t)) \\
\int_0^{s} e^{-\pi_1(Y(t+u-)) + \pi_1(Y(t))} e^{-\pi_1(Y(t))} d\pi_2(Y(t+u))
\end{array} \right) \\
& = 
\left( \begin{array}{c}
\pi_1(Y(t+s)) - \pi_1(Y(t)) \\
\int_0^{s} e^{-\pi_1(Y(t+u-))} d\pi_2(Y(t+u))
\end{array} \right) 
= 
\left( \begin{array}{c}
\pi_1(Y(t+s)) - \pi_1(Y(t)) \\
\int_t^{t+s} e^{-\pi_1(Y(u-))} d\pi_2(Y(u))
\end{array} \right) \\
& = 
\left( \begin{array}{c}
\xi(t+s) \\
\eta(t+s) 
\end{array} \right) 
- 
\left( \begin{array}{c}
\xi(t) \\
\eta(t) 
\end{array} \right). 
\end{align*}
This proves that $(\xi, \eta)$ has stationary increments. Moreover, recall that $((\xi(s), \eta(s)), \ 0 \leq s \leq t)$ is measurable with respect to $\mathcal{F}_t$ and that $(\tilde \xi, \tilde \eta)$ is independent from $\mathcal{F}_t$, so the independence of of increments for $(\xi, \eta)$ follows. $(\xi, \eta)$ is thus indeed a L\'evy process on $(\mathbb{R}^2, +)$. 

It now only remains to justify that the expression \eqref{relfonctexpo} is satisfied. Since $\eta$ is defined via \eqref{levynonclassiques1}, Theorem II.5.19 in \cite{protter} ensures that $\eta$ is a semimartingale for the filtration $(\mathcal{F}_t)_{t \geq 0}$ (alternatively, this follows from the fact that $\eta$ is a L\'evy process, adapted to $(\mathcal{F}_t)_{t \geq 0}$, and the corollary of Theorem II.3.9 in \cite{protter}). The process $( e^{\pi_1(Y(t-))}, \ t \geq 0 )$ is adapted with respect to $(\mathcal{F}_t)_{t \geq 0}$ and left continuous with right limits. Therefore, by Section II.4 in \cite{protter}, this process can be integrated with respect to $\eta$, and the stochastic integral in the meaning of \cite{protter} coincides with the classical stochastic integral with respect to the L\'evy process $\eta$ (see for example Section 4 in \cite{applebaum_2009} for the definition of stochastic integrals with respect to real L\'evy processes). Indeed, both integrals are limits of stochastic integrals, with respect to $\eta$, of simple predictable processes approaching the integrand. By the associativity property in Theorem II.5.19 of \cite{protter}, we have $\int_0^. e^{\pi_1(Y(s-))} d \eta(s) = \int_0^. e^{\pi_1(Y(s-))} e^{-\pi_1(Y(s-))} d\pi_2(Y(s))$. Then we get that, a.s. for any $t \geq 0$, 
\begin{align*} 
\left( \begin{array}{c}
\xi(t) \\
\int_0^t e^{\xi(s-)} d \eta(s) 
\end{array} \right) 
& = 
\left( \begin{array}{c}
\pi_1(Y(t)) \\
\int_0^t e^{\pi_1(Y(s-))} d \eta(s) 
\end{array} \right)
= 
\left( \begin{array}{c}
\pi_1(Y(t)) \\
\int_0^t e^{\pi_1(Y(s-))} e^{-\pi_1(Y(s-))} d\pi_2(Y(s)) 
\end{array} \right) \\
& = 
\left( \begin{array}{c}
\pi_1(Y(t)) \\
\int_0^t 1 d\pi_2(Y(s)) 
\end{array} \right) 
= 
\left( \begin{array}{c}
\pi_1(Y(t)) \\
\pi_2(Y(t)) 
\end{array} \right)
= Y(t). 
\end{align*}
Therefore, \eqref{relfonctexpo} holds a.s. for our definition of $(\xi, \eta)$. 

\end{proof}

\section{Proof of the main theorems} \label{proofsofmainth}


\subsection{Proof of Theorems \ref{casrexpl} and \ref{casr}}



We first prove Theorem \ref{casrexpl}. Let $I$ be an open interval of $\mathbb{R}$, and $X$ be a process satisfying Definition \ref{self-simoursens} on $I$. For some $k \geq 1$, $((f_y, c_y), \ y \in E)$ is a family of $\mathcal{C}^k$ good invariance components associated with $X$, relatively to some reference point $y_0 \in I$. We assume that either Assumption \ref{hypsupport1} or Assumption \ref{hypsupport2} is satisfied for $E=I$. Lemma \ref{equivassumpt12} then guaranties that Assumption \ref{hypsupport1} is satisfied. 
According to Lemma \ref{discretendim1}, $Sym(X_{y_0})$ is discrete. 

We can thus apply Proposition \ref{groupappears}. Let us define an interne composition law $\star$ on $I$ by $y \star x := f_y(x)$. According to Proposition \ref{groupappears}, $(I,\star)$ is a $\mathcal{C}^k$-Lie group (for the natural differential structure on $I$, arising from the fact that it is an open subset of $\mathbb{R}$) with neutral element $y_0$, and $(y \longmapsto c_y)$ is a $\mathcal{C}^k$-Lie group homomorphism from $(I,\star)$ to $(\mathbb{R}_+^*, \times)$. According to Proposition \ref{levychangentps}, there is a L\'evy process $L$, on $(I,\star)$, such that if we set $\forall t \in [0, +\infty], \ \varphi(t) := \int_0^t 1/c_{L(s)} ds$, then $\zeta(X_{y_0}) = \varphi(+\infty)$ and 
\begin{eqnarray}
\forall \ 0 \leq t < \varphi(+\infty) = \zeta(X_{y_0}), \ X_{y_0}(t) = L \left (\varphi^{-1}(t) \right ). \label{XenfctdeL}
\end{eqnarray}
Let us define $g : I \longrightarrow \mathbb{R}$ as in \eqref{defgdim1}. Then, according to Lemma \ref{isomorphismedim1}, $g$ is well-defined, is a $\mathcal{C}^{k}$ diffeomorphism from $I$ to $\mathbb{R}$, and even a $\mathcal{C}^{k}$-Lie group isomorphism from $(I, \star)$ to $(\mathbb{R}, +)$. Let us define the real valued process $\xi$ as $\xi(t) := g(L(t))$. Since $g$ is a continuous group homomorphism from $(I, \star)$ to $(\mathbb{R}, +)$ we have clearly that $\xi$ is a real L\'evy process. Then, we have $\forall t \in [0, +\infty], \ \varphi(t) = \int_0^t 1/c_{g^{-1}(\xi(s))} ds$ and \eqref{XenfctdeL} can be re-written as 
\[ \forall \ 0 \leq t < \varphi(+\infty) = \zeta(X_{y_0}), \ X_{y_0}(t) = g^{-1} \left ( \xi \left (\varphi^{-1}(t) \right ) \right ). \]
Then, since $g^{-1}$ is a continuous group homomorphism from $(\mathbb{R}, +)$ to $(I,\star)$, and $(y \longmapsto c_y)$ is a continuous group homomorphism from $(I,\star)$ to $(\mathbb{R}_+^*, \times)$, we get that $(x \longmapsto c_{g^{-1}(x)})$ is a continuous group homomorphism from $(\mathbb{R}, +)$ to $(\mathbb{R}_+^*, \times)$. Therefore, there exists $\alpha \in \mathbb{R}$ such that $c_{g^{-1}(x)} = e^{-\alpha x}, \forall x \in \mathbb{R}$. Clearly, this $\alpha$ is equal to $-\log(c_{g^{-1}(1)})$ and "$\varphi(t) = \int_0^t 1/c_{g^{-1}(\xi(s))} ds$" can be re-written as $\varphi(t) = \int_0^t e^{\alpha \xi(s)} ds$, which terminates the proof of the asserted representation for $X$. 


Let us now justify Remark \ref{levyexpldim1}. Clearly we only need to justify that for all $0 \leq t < \zeta(X_{y_0})$, $\varphi^{-1}(t) = \int_0^t c_{X_{y_0}(u)} du$. In the above proof, an application of Proposition \ref{levychangentps} yielded the existence of a process $L$ satisfying the relation \eqref{XenfctdeL}, with $\forall t \in [0, +\infty], \ \varphi(t) = \int_0^t 1/c_{L(s)} ds$. According to Remark \ref{chgttimeintermofx}, this implies that we have $0 \leq t < \zeta(X_{y_0})$, $\varphi^{-1}(t) = \int_0^t c_{X_{y_0}(u)} du$, which is the claim. 

We now justify Theorem \ref{casr}. Under the assumptions of the theorem, Lemma \ref{equivassumpt12} guaranties that Assumption \ref{hypsupport1} is satisfied. Then, Lemma \ref{discretendim1} applies so the conditions are satisfied to apply Proposition \ref{existsgoodcomponents}. We can thus produce a family $((f_y, c_y), \ y \in I)$ of $\mathcal{C}^k$ good invariance components associated with $X$, relatively to the reference point $y_0$. Then Theorem \ref{casrexpl} applies and yields the direct part of the theorem. The reciprocal is straightforward to verify, following the procedure from the proof of Proposition \ref{levychangentpssensfacil}. 

\subsection{Proof of Theorems \ref{casr2expl} and \ref{casr2}}

We first prove Theorem \ref{casr2expl}. Let $\mathcal{D}$ be an open simply connected domain of $\mathbb{R}^2$, and $X$ be a process satisfying Definition \ref{self-simoursens} on $\mathcal{D}$. For some $k \geq 2$, $((f_y, c_y), \ y \in E)$ is a family of $\mathcal{C}^k$ good invariance components associated with $X$, relatively to some reference point $y_0 \in \mathcal{D}$. We assume that either Assumption \ref{hypsupport1} or Assumption \ref{hypsupport2} is satisfied for $E=\mathcal{D}$ and that $Sym(X_{y_0})$ is discrete. 

Let us define an interne composition law $\star$ on $\mathcal{D}$ by $y \star x := f_y(x)$. According to Proposition \ref{groupappears}, $(\mathcal{D},\star)$ is a $\mathcal{C}^k$-Lie group (for the natural differential structure on $\mathcal{D}$, arising from the fact that it is an open subset of $\mathbb{R}^2$) with neutral element $y_0$, and $(y \longmapsto c_y)$ is a $\mathcal{C}^k$-Lie group homomorphism from $(\mathcal{D},\star)$ to $(\mathbb{R}_+^*, \times)$. According to Proposition \ref{levychangentps}, there is a L\'evy process $L$, on $(\mathcal{D},\star)$, such that if we set $\forall t \in [0, +\infty], \ \varphi(t) := \int_0^t 1/c_{L(s)} ds$, then $\zeta(X_{y_0}) = \varphi(+\infty)$ and 
\begin{eqnarray}
\forall \ 0 \leq t < \varphi(+\infty) = \zeta(X_{y_0}), \ X_{y_0}(t) = L \left (\varphi^{-1}(t) \right ). \label{XenfctdeL2}
\end{eqnarray}

We now distinguish two cases: 

1) If $((f_y, c_y), \ y \in \mathcal{D})$ are \textit{commutative invariance components}, then the Lie group $(\mathcal{D},\star)$ is commutative. Let us define $g : \mathcal{D} \longrightarrow \mathbb{R}^2$ as in \eqref{defgdim2com}. Since the matrix $M$, defined a little before \eqref{defgdim2com}, is always invertible, Lemma \ref{isomorphismecasclassique} ensures that $g$ is well-defined, is a $\mathcal{C}^{k}$ diffeomorphism from $\mathcal{D}$ to $\mathbb{R}^2$, and even a $\mathcal{C}^{k}$-Lie group isomorphism from $(\mathcal{D}, \star)$ to $(\mathbb{R}^2, +)$. Let us define the $\mathbb{R}^2$ valued process $(\xi, \eta)$ as $(\xi(t), \eta(t)) := g(L(t))$, where $L$ is the L\'evy process on $(\mathcal{D},\star)$ that appears in \eqref{XenfctdeL2}. Since $g$ is a continuous group homomorphism we have that $(\xi, \eta)$ is a L\'evy process on $(\mathbb{R}^2, +)$. Also, replacing $L$ by $g^{-1}(\xi, \eta)$ we obtain 
\begin{eqnarray}
\forall t \in [0, +\infty], \ \varphi(t) = \int_0^t \frac1{c_{L(s)}} ds = \int_0^t \frac1{c_{g^{-1}(\xi(s), \eta(s))}} ds, \label{chgttpsfctlevy}
\end{eqnarray}
and \eqref{XenfctdeL2} can be re-written as \eqref{represxdim2com} (but with $\varphi(.)$ as in \eqref{chgttpsfctlevy}). It thus only remains to prove that $\varphi(t) = \int_0^t e^{\alpha \xi(s)} ds$ where $\alpha$ is as defined in the statement of the theorem, that is, $\alpha = -\log(c_{g^{-1}(1,0)})$. 

Note that $(z \longmapsto c_{g^{-1}(z)})$ is a $\mathcal{C}^k$-Lie group homomorphism from $(\mathbb{R}^2, +)$ to $(\mathbb{R}_+^*, \times)$. Therefore, there are $\beta_1, \beta_2 \in \mathbb{R}$ such that $\forall z = (z_1, z_2) \in \mathbb{R}^2, c_{g^{-1}(z)} = e^{\beta_1 z_1 + \beta_2 z_2}$. Differentiating at $(0,0)$ with respect to respectively $z_1$ and $z_2$ we get 
\begin{align}
\beta_1 \ = \ ^tY_0.Jg^{-1}(0,0).e_1 \ = \ ^tY_0.(Jg(y_0))^{-1}.e_1 \ = \ ^tY_0.M^{-1}.e_1 \ \ \ \text{and} \ \ \ \beta_2 \ = \ ^tY_0.M^{-1}.e_2. \label{exprbeta}
\end{align}
In the above expression we have used that, as defined in the statement of the theorem, $Y_0$ is the gradient at $y_0$ of the application $(y \longmapsto c_{y})$, and the fact that $Jg(y_0) = M$ (see Lemma \ref{isomorphismecasclassique}). 

If $Y_0 = \begin{pmatrix}
0 \\
0
\end{pmatrix}$ then we have clearly from \eqref{exprbeta} that $\beta_2=0$. If $Y_0 \neq \begin{pmatrix}
0 \\
0
\end{pmatrix}$ then recall that $M = \begin{pmatrix}
\pi_1(Y_0) & \pi_2(Y_0) \\
-\pi_2(Y_0) & \pi_1(Y_0)
\end{pmatrix}$ so $M^{-1} = \frac1{||Y_0||^2} \begin{pmatrix}
\pi_1(Y_0) & -\pi_2(Y_0) \\
\pi_2(Y_0) & \pi_1(Y_0)
\end{pmatrix}$, where $||.||$ denotes the Euclidian norm of vectors of $\mathbb{R}^2$. Combining with \eqref{exprbeta} we get $\beta_2 = 0$. Therefore, in any case, we have $\forall z = (z_1, z_2) \in \mathbb{R}^2, c_{g^{-1}(z)} = e^{\beta_1 z_1}$ and, evaluating at $z = (1,0)$, $\beta_1 = \log(c_{g^{-1}(1,0)}) = -\alpha$. Putting the obtained expression of $c_{g^{-1}(.)}$ into \eqref{chgttpsfctlevy} we get $\varphi(t) = \int_0^t e^{\alpha \xi(s)} ds$ as required, which terminates the proof. 

2) If $((f_y, c_y), \ y \in \mathcal{D})$ are not \textit{commutative invariance components}, then the Lie group $(\mathcal{D},\star)$ is non-commutative. Let us define $g : \mathcal{D} \longrightarrow \mathbb{R}^2$ as in Lemma \ref{isomorphismecasnonclassique}. Then, according to that lemma, $g$ is well-defined, is a $\mathcal{C}^{k}$ diffeomorphism from $\mathcal{D}$ to $\mathbb{R}^2$, and even a $\mathcal{C}^{k}$-Lie group isomorphism from $(\mathcal{D}, \star)$ to $(\mathbb{R}^2, T)$. Let us define the $\mathbb{R}^2$ valued process $Y$ as $Y(t) := g(L(t))$, where $L$ is the L\'evy process on $(\mathcal{D},\star)$ that appears in \eqref{XenfctdeL2}. Since $g$ is a continuous group homomorphism we have that $Y$ is a L\'evy process on $(\mathbb{R}^2, T)$. According to Proposition \ref{levynonclassiques}, there is a L\'evy process $(\xi, \eta)$ on $(\mathbb{R}^2, +)$ such that $Y(t) = (\xi(t), \int_0^t e^{\xi(s-)} d \eta(s)), \forall t \geq 0$. Therefore, we can replace $L(.)$ by $g^{-1}(\xi(.), \int_0^. e^{\xi(s-)} d \eta(s))$. We obtain 
\begin{eqnarray}
\forall t \in [0, +\infty], \ \varphi(t) = \int_0^t \frac1{c_{L(s)}} ds = \int_0^t \frac1{c_{g^{-1}(\xi(s), \int_0^s e^{\xi(u-)} d \eta(u))}} ds, \label{chgttpsfctlevy2}
\end{eqnarray}
and \eqref{XenfctdeL2} can be re-written as \eqref{represxdim2noncom} (but with $\varphi(.)$ as in \eqref{chgttpsfctlevy2}). It thus only remains to prove that $\varphi(t) = \int_0^t e^{\alpha \xi(s)} ds$ where $\alpha$ is as defined in the statement of the theorem, that is, $\alpha = -\log(c_{g^{-1}(1,0)})$. 

$(z \longmapsto c_{g^{-1}(z)})$ is a $\mathcal{C}^k$-Lie group homomorphism from $(\mathbb{R}^2, T)$ to $(\mathbb{R}_+^*, \times)$. Therefore, according to Lemma \ref{shapeofmorphismes}, there is $\beta \in \mathbb{R}$ such that $\forall z = (z_1, z_2) \in \mathbb{R}^2, c_{g^{-1}(z)} = e^{\beta z_1}$ and, evaluating at $z = (1,0)$, $\beta = \log(c_{g^{-1}(1,0)}) = -\alpha$. Putting the obtained expression of $c_{g^{-1}(.)}$ into \eqref{chgttpsfctlevy2} we get $\varphi(t) = \int_0^t e^{\alpha \xi(s)} ds$ as required, which terminates the proof. 

Here, we still have the alternative expression $\varphi^{-1}(t) = \int_0^t c_{X_{y_0}(u)} du$ for $0 \leq t < \zeta(X_{y_0})$. The justification is just as in the proof of Theorem \ref{casrexpl}. 

We now justify Theorem \ref{casr2}. Under the assumptions of the theorem, Proposition \ref{existsgoodcomponents} applies. We can thus produce a family $((f_y, c_y), \ y \in I)$ of $\mathcal{C}^k$ good invariance components associated with $X$, relatively to the reference point $y_0$. Then Theorem \ref{casr2expl} applies and yields the direct part of the theorem. Indeed, if the good invariance components are commutative then Theorem \ref{casr2expl} shows that \eqref{timechangingdim2}, \eqref{tpsviedim2} and \eqref{represdim2} are true for $\psi = g^{-1}$, $\beta = 0$ (and $g$, $(\xi, \eta)$, $\alpha$ are given in the first point of Theorem \ref{casr2expl}). If the good invariance components are not commutative then Theorem \ref{casr2expl} shows that \eqref{timechangingdim2}, \eqref{tpsviedim2} and \eqref{represdim2} are true for $\psi = g^{-1}$, $\beta = 1$ (and $g$, $(\xi, \eta)$, $\alpha$ are given in the second point of Theorem \ref{casr2expl}). The reciprocal is straightforward to verify, following the procedure from the proof of Proposition \ref{levychangentpssensfacil}.

\section{Techinal results} \label{technical}

We now state and prove Lemmas \ref{equivassumpt12} and \ref{discretendim1}, about the assumptions that are discussed in the end of Subsection \ref{introgssmp} and used all along the paper. 

\begin{lemme} \label{equivassumpt12}

Let $E$ be a connected locally compact separable metric space, and let $X$ be a process satisfying Definition \ref{self-simoursens} on $E$ with invariance components $((f_y, c_y), \ y \in E)$ relatively to some reference point $y_0 \in E$. Then, Assumption \ref{hypsupport2} implies Assumption \ref{hypsupport1}. 

\end{lemme}


\begin{proof}

Let us assume that Assumption \ref{hypsupport2} holds. Let $y$ be such that $y \in \mathring{\widehat{Supp(X_{y})}}$. Clearly we have $Supp(X_{y}) = f_y(Supp(X_{y_0}))$ so, since $f_y$ is an homeomorphism that maps $y_0$ to $y$, we have $y_0 \in \mathring{\widehat{Supp(X_{y_0})}}$. Similarly, $Supp(X_{z}) = f_z(Supp(X_{y_0}))$ with $f_z$ an homeomorphism that maps $y_0$ to $z$, so $z \in \mathring{\widehat{Supp(X_{z})}}$ for all $z \in E$. In other words, for all $z \in E$, $Supp(X_{z})$ is a neighborhood of $z$. 


We now justify that $z \in Supp(X_{y_0}) \Rightarrow Supp(X_{z}) \subset Supp(X_{y_0})$. Let $z \in Supp(X_{y_0})$, $v \in Supp(X_z)$ and $\epsilon > 0$, we need to prove the existence of $r \geq 0$ such that $\mathbb{P} (X_{y_0}(r) \in B(v, \epsilon)) > 0$. By definition of $Supp(X_z)$, there is $t \geq 0$ such that $\mathbb{P} (X_z(t) \in B(v, \epsilon)) > 0$. By self-similarity (Definition \ref{self-simoursens}), $\mathbb{P} (f_z(X_{y_0}(c_z t)) \in B(v, \epsilon)) > 0$. Since $(y,x) \longmapsto f_y(x)$ and $y \longmapsto c_y$ are continuous, and $X_{y_0}$ is stochastically continuous, we have that $f_w(X_{y_0}(c_w t))$ converges in distribution to $f_z(X_{y_0}(c_z t))$ as $w$ goes to $z$. Therefore, there is a neighborhood $\mathcal{U}$ of $z$ such that 
\begin{eqnarray}
\forall w \in \mathcal{U}, \ \mathbb{P} (X_w(t) \in B(v, \epsilon)) = \mathbb{P} (f_w(X_{y_0}(c_w t)) \in B(v, \epsilon)) > 0. \label{preuveequiv1}
\end{eqnarray}
Then, since $\mathcal{U}$ is a neighborhood of $z \in Supp(X_{y_0})$, there is $s \geq 0$ such that 
\begin{eqnarray}
\mathbb{P} (X_{y_0}(s) \in \mathcal{U}) > 0. \label{preuveequiv2}
\end{eqnarray}
The combination of \eqref{preuveequiv2}, \eqref{preuveequiv1}, and the Markov property at time $s$ yields that $\mathbb{P} (X_{y_0}(t+s) \in B(v, \epsilon)) > 0$, proving the claim. 

We have established that for each $z \in Supp(X_{y_0})$, $Supp(X_{y_0})$ contains $Supp(X_{z})$ which is a neighborhood of $z$. Therefore $Supp(X_{y_0})$ is open. Since it also close (see \eqref{defsupp}) and nonempty (it contains $y_0$, as previously justified), and since $E$ is connected, we conclude that $Supp(X_{y_0}) = E$, that is, Assumption \ref{hypsupport1}. 


\end{proof}


\begin{lemme} \label{discretendim1}

Let $I \subset \mathbb{R}$ be an open interval, $y_0 \in I$, $X_{y_0}$ be a Markovian process on $I$ that satisfies Assumption \ref{hypsupport1} with $E=I$, and let $Sym(X_{y_0})$ be defined as in \eqref{defsym}. Then $Sym(X_{y_0})$ is discrete. 

\end{lemme}

\begin{proof}

Let $h$ be an increasing element of $Sym(X_{y_0})$, and let us fix $x \in I$. Because of Assumption \ref{hypsupport1} and of the definition of the support (see \eqref{defsupp}), we have that for any $\epsilon > 0$, we can find $t_1, t_2 \geq 0$ such that $\mathbb{P} (X_{y_0}(t_1) \in ]x, x+ \epsilon[) > 0$ and $\mathbb{P} (X_{y_0}(t_2) \in ]x- \epsilon, x[) > 0$. Therefore we have 
\begin{eqnarray}
\mathbb{P} (X_{y_0}(t_1) > x) > \mathbb{P} (X_{y_0}(t_1) > x + \epsilon) \ \ \ \text{and} \ \ \ \mathbb{P} (X_{y_0}(t_2) > x) < \mathbb{P} (X_{y_0}(t_2) > x - \epsilon). \label{presquecroissance}
\end{eqnarray}
$h^{-1} \in Sym(X_{y_0})$ and is also increasing so we have 
\begin{align*}
\mathbb{P} (X_{y_0}(t_1) > x) = \mathbb{P} (h^{-1}(X_{y_0}(t_1)) > x) = \mathbb{P} (X_{y_0}(t_1) > h(x)) \\
\mathbb{P} (X_{y_0}(t_2) > x) = \mathbb{P} (h^{-1}(X_{y_0}(t_2)) > x) = \mathbb{P} (X_{y_0}(t_2) > h(x)). 
\end{align*}
Combining with \eqref{presquecroissance} we obtain $\mathbb{P} (X_{y_0}(t_1) > h(x)) > \mathbb{P} (X_{y_0}(t_1) > x + \epsilon)$ and $\mathbb{P} (X_{y_0}(t_2) > h(x)) < \mathbb{P} (X_{y_0}(t_2) > x - \epsilon)$ which implies $h(x) \in ]x - \epsilon, x +\epsilon[$. Since this is true for any $\epsilon > 0$ we get $h(x) = x$, and since $x \in I$ is arbitrary we conclude that $h = id_I$. 

Since any $h \in Sym(X_{y_0})$ is an homeomorphisms of $I$, it is either increasing or decreasing. Let $h$ be a decreasing element of $Sym(X_{y_0})$, if such an element exists. If $\tilde h$ is any other decreasing element of $Sym(X_{y_0})$, then $\tilde h^{-1} \circ h$ is an increasing element of $Sym(X_{y_0})$ so, by the previous part, it equals $id_I$, so $\tilde h = h$. In conclusion $Sym(X_{y_0})$ contains at most two elements so it is discrete. 

\end{proof}

\bibliographystyle{plain}
\bibliography{thbiblio}

\end{document}